\documentclass{article}
\usepackage[lang = american]{ems-emss} 
\usepackage{tikz}
\usetikzlibrary{cd}
\usepackage{breakurl}
\usepackage{xurl}
\hypersetup{breaklinks=true,  colorlinks=true,  pdfusetitle=true}

\newlist{inlinelist}{enumerate*}{1}
\setlist[inlinelist]{label=\alph*, itemjoin={{, }}, itemjoin*={{, and }}} 
\setlist[enumerate]{label=(\arabic*)}

\numberwithin{equation}{section}

\newtheorem{theorem}{Theorem}[section]
\newtheorem{corollary}[theorem]{Corollary}
\newtheorem{lemma}[theorem]{Lemma}
\newtheorem{proposition}[theorem]{Proposition}
\newtheorem{conjecture}[theorem]{Conjecture}
\newtheorem{claim}[theorem]{Claim}

\theoremstyle{definition}
\newtheorem{definition}[theorem]{Definition}
\newtheorem{example}[theorem]{Example}
\newtheorem{remark}[theorem]{Remark}

\newcommand{\N}{\mathbb{N}}
\newcommand{\Z}{\mathbb{Z}}
\newcommand{\Q}{\mathbb{Q}}
\newcommand{\R}{\mathbb{R}}
\renewcommand{\C}{\mathbb{C}}
\newcommand{\T}{\mathbb{T}}

\newcommand{\CP}{\mathbb{CP}}
\newcommand{\RP}{\mathbb{RP}}
\newcommand{\K}{\mathbb{K}}
\renewcommand{\L}{\mathbb{L}}
\renewcommand{\k}{\mathbb{k}}

\newcommand{\bV}{\mathbf{V}}

\newcommand{\VV}{\mathcal{V}}
\newcommand{\A}{{\mathcal{A}}}
\newcommand{\cC}{\mathcal{C}}

\newcommand{\RR}{\mathcal{R}}
\newcommand{\M}{\mathcal{M}}
\newcommand{\F}{\mathcal{F}}
\newcommand{\cG}{\mathcal{G}}

\newcommand{\cE}{\mathcal{E}}
\newcommand{\XX}{{\mathcal{X}}}
\newcommand{\ZZ}{{\mathcal{Z}}}
\newcommand{\wC}{\,\widehat{\mathcal{\!C}}}

\newcommand{\fa}{\mathfrak{a}}

\newcommand{\fB}{\mathfrak{B}}
\newcommand{\B}{\mathfrak{B}}

\newcommand{\g}{{\mathfrak{g}}}
\newcommand{\h}{{\mathfrak{h}}}
\newcommand{\m}{{\mathfrak{m}}}
\newcommand{\fM}{\mathfrak{M}}

\newcommand{\fr}{\mathfrak{r}}
\newcommand{\fL}{\mathfrak{L}}

\newcommand{\dga}{\ensuremath{\textsc{dga}}}
\newcommand{\cga}{\ensuremath{\textsc{cga}}}
\newcommand{\cdga}{\ensuremath{\textsc{cdga}}}	

\DeclareMathOperator{\rank}{rank}
\DeclareMathOperator{\gr}{gr}
\DeclareMathOperator{\im}{im}

\DeclareMathOperator{\id}{id}
\DeclareMathOperator{\ab}{{ab}}

\DeclareMathOperator{\Sym}{Sym}
\DeclareMathOperator{\Sp}{Sp}
\DeclareMathOperator{\ch}{char}

\DeclareMathOperator{\SL}{SL}
\DeclareMathOperator{\SU}{SU}
\DeclareMathOperator{\SO}{SO}
\DeclareMathOperator{\Hom}{{Hom}}
\DeclareMathOperator{\Tor}{{Tor}}
\DeclareMathOperator{\Ext}{{Ext}}
\DeclareMathOperator{\Hilb}{{Hilb}}

\DeclareMathOperator{\proj}{pr}
\DeclareMathOperator{\ev}{ev}
\DeclareMathOperator{\Ev}{Ev}
\DeclareMathOperator{\Lie}{Lie}

\DeclareMathOperator{\init}{in}

\DeclareMathOperator{\Tors}{Tors}
\DeclareMathOperator{\TC}{TC}
\DeclareMathOperator{\Prim}{Prim}

\DeclareMathOperator{\Char}{Char}
\DeclareMathOperator{\Pf}{Pf}
\DeclareMathOperator{\pf}{pf}
\DeclareMathOperator{\Det}{Det}
\DeclareMathOperator{\Conf}{Conf}

\DeclareMathOperator{\Ann}{Ann}
\DeclareMathOperator{\orb}{orb}
\DeclareMathOperator{\Heis}{\mathcal{H}}

\providecommand{\lcs}{\mathop{\rm LCS} \nolimits}

\newcommand{\surj}{\twoheadrightarrow}
\newcommand{\inj}{\hookrightarrow}

\newcommand{\isom}{\xrightarrow{
   \,\smash{\raisebox{-0.4ex}{\ensuremath{\scriptstyle\cong}}}\,}}
\newcommand{\abs}[1]{\left| #1 \right|}
\def\set#1{{\{ #1\}}}

\newcommand{\uX}{\underline{X}}
\newcommand{\oX}{\overline{X}}

\newcommand{\one}{\mathbf{1}}
\newcommand{\zero}{\mathbf{0}}
\newcommand{\dR}{\scriptscriptstyle{\rm dR}}

\newcommand{\apl}{A_{\scriptscriptstyle{{\rm PL}}}}
\newcommand{\qA}{\mathrm{q}A}
\newcommand{\qg}{{\mbox{\rm{\small{q}}}}\g}
\newcommand{\q}{\mathrm{q}}
\DeclareMathOperator{\rat}{\scalebox{0.6}{$\Q$}}

\newcommand{\bwedge}{\mbox{\normalsize $\bigwedge$}}

\newcommand{\boplus}{\mbox{\normalsize $\bigoplus$}}

\newcommand{\sbm}[1]{{\let\amp=&\left(\begin{smallmatrix}#1\end{smallmatrix}\right)}}

\title{Formality and finiteness in rational homotopy theory}
\titlemark{Formality and finiteness in rational homotopy theory}

\emsauthor{1}{Alexander~I.~Suciu}{A.~I.~Suciu}
\emsaffil{1}{Department of Mathematics, Northeastern University, 
Boston, MA 02115, USA \email{a.suciu@northeastern.edu}}

\begin{document}

\maketitle


\classification[20F14, 20F40, 55N25]{55P62}


\begin{abstract}
We explore various formality and finiteness properties in the 
differential graded algebra models for the Sullivan algebra 
of piecewise polynomial rational forms on a space.  
The $1$-formality property of the space may be reinterpreted 
in terms of the filtered and graded formality properties of 
the Malcev Lie algebra of its fundamental group, 
while some of the finiteness properties of the space are mirrored 
in the finiteness properties of algebraic models associated with it. 
In turn, the formality and finiteness properties of 
algebraic models have strong implications on the geometry 
of the cohomology jump loci of the space. We illustrate the 
theory with examples drawn from complex algebraic geometry, 
actions of compact Lie groups, and $3$-dimensional manifolds. 
\end{abstract}

\tableofcontents

\section{Introduction}
\label{sect:intro}

\subsection{Rational homotopy type}
\label{subsec:intro-rht}

Homotopy theory is the study of topological spaces up to 
homotopy equivalences. Typical examples of homotopy 
type invariants of a space $X$ are the homology 
groups $H_n(X,\Z)$ and the homotopy groups $\pi_n(X)$. 
The question whether one can reconstruct the homotopy type 
of a space from homological data goes back to the beginnings 
of Algebraic Topology.  Poincar\'{e} realized that homology is 
not enough:  for a path-connected space $X$, the first homology 
group, $H_1(X,\Z)$, only records the abelianization of the 
fundamental group, $\pi_1(X)$.  Even for simply-connected 
spaces, homology by itself fails to detect maps such as the 
Hopf map, $S^3\to S^2$.  On the other hand, if one considers 
the de~Rham algebra of differential forms, one can 
reconstitute in a purely algebraic fashion all the higher 
homotopy groups of $S^n$, modulo torsion.  

As founded by Quillen \cite{Qu69} and Sullivan \cite{Sullivan77}, 
rational homotopy theory is the study of rational homotopy types 
of spaces. Instead of considering the groups $H_n(X,\Z)$ and 
$\pi_n(X)$, one considers the rational homology groups $H_n(X,\Q)$ 
and the rational homotopy groups $\pi_n(X)\otimes \Q$. 
These objects are $\Q$-vector spaces, and hence the 
torsion information is lost, yet this is compensated by 
the fact that computations are easier in this setting.

\subsection{Models for spaces and groups}
\label{subsec:intro-models}
In his seminal paper, \cite{Sullivan77}, Sullivan attached in a functorial way to 
every space $X$ a commutative differential graded algebra over $\Q$, 
denoted $\apl(X)$. This $\cdga$ is constructed from piecewise polynomial rational 
forms and is weakly equivalent (through $\dga$s) with the cochain 
algebra $(C^*(X,\Q),d)$ so that, under the resulting identification 
of graded algebras, $H^{*}(\apl(X)) \cong H^{*}(X,\Q)$, the induced 
homomorphisms in cohomology correspond. 

We say that two $\cdga$s $A$ and $B$ are weakly equivalent 
(written $A\simeq B$) if there is a zig-zag of $\cdga$ maps 
inducing isomorphisms in cohomology and connecting $A$ to $B$. 
If those maps only induce isomorphisms in degree at most 
$q$ (for some $q\ge 0$) and monomorphisms in degree $q+1$, we say 
$A$ and $B$ are $q$-equivalent (written $A\simeq_q B$) 

Let $\k$ be a coefficient field of characteristic $0$. 
We say that a $\k$-$\cdga$ $(A,d)$ is a {\em model}\/ for a 
space $X$ if $A\simeq \apl(X)\otimes_{\Q} \k$, and likewise 
for a {\em $q$-model}. 
For instance, if $X$ is a smooth manifold, de~Rham's algebra 
$\Omega^*_{\dR}(X)$ is a model for $X$ over $\R$, 
leading to the following basic principle in rational homotopy theory:
``The manner in which a closed form which is zero in 
cohomology actually becomes exact contains geometric 
information,'' cf.~\cite{DGMS}. 

Given a connected $\cdga$ $A$, Sullivan constructed a {\em minimal model}\/ 
for it, $\rho\colon \M(A)\to \A$, where $\rho$ is a quasi-isomorphism 
and $\M(A)$ is a $\cdga$ obtained by iterated Hirsch extensions, 
starting from $\k$, so that its differential is decomposable. These 
properties uniquely characterize the minimal model of $A$ (up to 
isomorphism). The $q$-minimal models $\M_q(A)$ are similarly 
defined for all $q\ge 0$; they are generated by elements of degrees 
at most $q$, and the structural morphisms $\rho_q\colon \M_q(A)\to A$ 
are only $q$-quasi-isomorphisms. 

A minimal model for a connected space $X$, denoted $\M(X)$, is a 
minimal model for the Sullivan algebra $\apl(X)$. The isomorphism 
type of $\M(X)$ is uniquely defined by the rational homotopy type of $X$. 
The $q$-minimal models $\M_q(X)$ are defined analogously; moreover, 
if $G$ is a finitely generated group, we set $\M_1(G)=\M_1(K(G,1))$. 
When $X$ is a nilpotent CW-complex 
with finite Betti numbers, Sullivan \cite{Sullivan77} showed that 
$\pi_n(X)\otimes \Q\cong (V^n)^{\vee}$ for all $n\ge 2$, 
where $V=\bigoplus_{n} V^n$ and $\M(X)=\big( \bwedge V, d)$ is a 
minimal model for $X$ over $\Q$. 

\subsection{Formality}
\label{subsec:intro-formal}

As formulated in \cite{Sullivan77}, \cite{DGMS}, the notion of formality 
plays a central role in rational homotopy theory. 
We say that a path-connected space $X$ is {\em formal}\/ 
if its Sullivan algebra, $\apl(X)$, is weakly equivalent to 
its cohomology algebra, $H^*(X,\Q)$, endowed with the 
zero differential. The notion of {\em $q$-formality}\/ 
(for some $q\ge 0$) is defined accordingly. In general, 
partial formality is a much weaker property than (full) formality; 
nevertheless, if $H^{\ge q+2}(X;\Q)=0$, then $X$ is $q$-formal 
if and only if $X$ is formal, see \cite{Mc10}. One may also 
talk about ($q$-) formality over a field $\k$, 
but it turns out that all these formality notions are independent 
of the choice of the coefficient field, as long as $\ch(\k)=0$. 

Various conditions on the connectivity of the space or the structure 
of its cohomology algebra guarantee formality.  For instance, 
if $X$ is a $k$-connected CW-complex ($k\ge 1$) of dimension 
$n$ and $n \le 3k+1$, then $X$ is formal, cf.~\cite{Sta83}; 
moreover, if $X$ is a closed manifold of dimension $n$, the 
conclusion remains valid for $n \le 4k+2$, cf.~\cite{Mi79}
Also, if $H^*(X,\Q)$ is the quotient of a free $\cga$ by an ideal generated 
by a regular sequence, then $X$ is formal, cf.~\cite{Sullivan77}.  

A classical obstruction to formality is provided by the Massey products 
(of order $3$ and higher): If a space $X$ is formal, then all Massey 
products in the cohomology algebra $H^*(X,\Q)$ vanish---in fact, 
vanish uniformly. Furthermore, if $X$ is $q$-formal, for some 
$q\ge 1$, then all Massey products in $H^{\le q+1}(X,\Q)$ vanish.

A simply-connected space (or, more generally, a nilpotent space) 
$X$ is formal if its rational homotopy type is determined by $H^*(X,\Q)$. 
In the general case, the weaker $1$-formality property allows one to 
reconstruct the rational pro-unipotent completion of the fundamental 
group, solely from the cup products of degree $1$ cohomology classes. 
Formal spaces lend themselves to various algebraic computations that 
provide valuable homotopy information. For instance, a result of 
Papadima--Yuzvinsky \cite{PY99}, which is valid for all formal spaces 
$X$, states:  The Bousfield--Kan completion $\Q_{\infty}X$ is 
aspherical if and only if $H^*(X,\Q)$ is a Koszul algebra.
 
\subsection{Finiteness in $\cdga$ models}
\label{subsec:intro-finite}

A recurring theme in topology is to determine the geometric and homological 
finiteness properties of spaces and groups. A prototypical such problem is 
to determine whether a path-connected space $X$ is homotopy equivalent 
to a CW-complex with finite $q$-skeleton, for some $1\le q \le \infty$, 
in which case we say $X$ is {\em $q$-finite}.  Another question  
is to decide whether a group $G$ is finitely generated, and if so, 
whether it admits a finite presentation; more generally, whether 
it has a classifying space $K(G,1)$ with finite $q$-skeleton. 

A fruitful approach to this type of question is to compare 
the finiteness properties of the spaces or groups under 
consideration to the corresponding finiteness properties 
of algebraic models for such spaces and groups. 
By analogy with the aforementioned 
topological notion, we say that a $\k$-$\cdga$ $A$ is 
{\em $q$-finite}\/ if it is connected (i.e., $A^0= \k \cdot 1$) 
and $\dim A^i < \infty$ for $i\le q$. 

A natural question then is:
When does a $q$-finite space $X$ admit a $q$-finite $q$-model $(A,d)$?
A necessary criterion is given in \cite{PS-jlms}: If a space $X$ does admit 
such a model, then $\dim H^i(\M_q(X))< \infty$, for all $i\le q+1$. 
For instance, if $G= F_n/F_n''$ is the free metabelian group of rank $n\ge 2$
then $b_2(\M_1(G))= \infty$, and so $G$ admits no $1$-finite $1$-model. 
Other finiteness criteria, based on the nature of the cohomology jump loci 
(see \cite{PS-jlms}, \cite{Su-mm}), are discussed below.

\subsection{Malcev and holonomy Lie algebras}
\label{subsec:intro-malcev}
In his landmark paper on rational homotopy theory, Quillen \cite{Qu69} 
defined the Malcev Lie algebra, $\m(G)$, of a finitely generated group $G$ 
as the (complete, filtered) Lie algebra of primitive elements in the $I$-adic 
completion of the group algebra $\Q[G]$, where $I$ is the augmentation ideal.  
The associated graded Lie algebra with respect to this filtration, $\gr(\m(G))$, 
is isomorphic to $\gr(G,\Q)$ the rational graded Lie algebra associated 
to the lower central series filtration of $G$, cf.~\cite{Qu68}. 

As shown by Sullivan \cite{Sullivan77} (see also \cite{CP}, \cite{GM13}), 
the Lie algebra dual to $\M_1(G)$ is isomorphic to the Malcev Lie 
algebra $\m(G)$. It follows that $G$ is $1$-formal if and only if 
$\m(G)$ is the LCS completion of a finitely generated, quadratic 
Lie algebra. A weaker condition was given in \cite{SW-forum}: 
we say that $G$ is {\em filtered formal}\/ 
if $\m(G)$ is the completion of $\gr(G,\Q)$ with respect to its 
LCS filtration. As shown in \cite{SW-ejm}, this condition is 
equivalent to the existence of a Taylor expansion, 
$G\to \widehat{\gr}(\Q[G])$.

Now suppose $G$ has a $1$-finite $1$-model $(A,d)$ over $\Q$. 
Building on a classical construction of K.-T.~Chen \cite{Chen73}, 
the {\em holonomy Lie algebra}\/ $\h(A)$ was defined in \cite{MPPS} 
as the quotient of the 
free Lie algebra on the dual vector space $A_1=(A^1)^{\vee}$ 
by the ideal generated by the image of the map $(d+ \mu)^{\vee}$, 
where $d\colon A^1\to A^2$ is the differential and $\mu\colon 
A^1\wedge A^1\to A^2$ is the multiplication map. Then, as shown 
in  \cite{BMPP}, \cite{PS-jlms} (generalizing a result from \cite{Bez}), 
the Malcev Lie algebra $\m(G)$ is isomorphic to the LCS 
completion of $\h(A)$.  Moreover, the following complete 
finiteness criterion in degree $1$ was given in \cite{PS-jlms}: 
A  finitely generated group $G$ admits a $1$-finite $1$-model if 
and only if $\m(G)$ is the LCS completion of a finitely presented 
Lie algebra.

\subsection{Cohomology jump loci}
\label{subsec:intro-cvs}

The cohomology jump loci of a space $X$ are of two 
basic types: the characteristic varieties, $\VV^i_k(X)$, 
defined in terms of homology with coefficients in rank 
one local systems, and the resonance varieties, $\RR^i_k(X)$ 
or $\RR^i_k(A)$, constructed from information encoded in either 
the cohomology ring $H^*(X,\C)$, or in a $\cdga$ model $A$ for $X$. 

The characteristic varieties and the related Alexander invariants of 
spaces and groups have their origin in the study of the Alexander 
polynomials of knots and links.  The basic topological idea in defining 
these invariants is to take the homology of the maximal abelian cover 
of a connected CW-complex $X$ and view it as a module over the group ring 
of $H_1(X,\Z)$. One then studies the support loci of these modules, 
or, alternatively, the jump loci $\VV^i_k(X)$, viewed as subsets of 
the character group $\Char(X)=\Hom(\pi_1(X),\C^*)$. 

The formality and finiteness properties of a space and its algebraic 
models put strong constraints on the geometric structure of the 
cohomology jump loci. To start with, let $X$ be a $q$-finite space, 
for some $q\ge 1$. Then the tangent cone at the trivial character 
$\one$ to the variety $\VV^i_k(X)$ is included in $\RR^i_k(X)$, 
for all $i\le q$ and $k\ge 0$, see \cite{Li02}.  

Now suppose $X$ admits a $q$-finite $q$-model $A$; 
then $\TC_{\one}(\VV^i_k(X))=\RR^i_k(A)$, for all $i\le q$, 
see \cite{DP-ccm}. Moreover, as a consequence of \cite{DPS-duke}, 
\cite{DP-ccm}, all irreducible components of these resonance 
varieties are rationally defined 
linear subspaces of $H^1(A)=H^1(X,\C)$, and, by \cite{BW20}, 
all the components of $\VV^i_k(X)$ passing through 
$\one$ are algebraic subtori of $\Char(X)$.
Finally, suppose $X$ is $q$-formal.  Then, for $i\le q$,   
all the components of $\RR^i_k(X)$ are rationally defined 
linear subspaces of $H^1(X,\C)$.

\subsection{Models for K\"{a}hler manifolds and smooth algebraic varieties}
\label{subsec:intro-algmod1}

One of the foundational papers of rational homotopy theory is the one 
by Deligne, Griffiths, Morgan, and Sullivan \cite{DGMS}, where the authors 
use Hodge theory and the $dd^{c}$-lemma to establish the formality of 
all compact K\"{a}hler manifolds, and thus, of all smooth, complex algebraic 
projective varieties.

In \cite{Mo}, Morgan constructed a finite-dimensional model for 
any smooth, complex, quasi-projec\-tive variety $X$ by using 
a normal crossings divisors compactification $\overline{X}$. 
This model was refined by Dupont in \cite{Dp15}, by allowing 
those divisors to intersect like hyperplanes in a hyperplane 
arrangement. These models are not always formal, 
but still retain good partial formality properties; for instance, if $X$ is the 
complement of a hypersurface in $\CP^n$, then $X$ is $1$-formal, but 
not formal, in general.

The structure of the characteristic varieties of compact K\"{a}hler manifolds 
and smooth, quasi-projective varieties was determined in 
\cite{Cat91}, \cite{Be92}, \cite{GL87}, \cite{Si93}, \cite{Ar}, \cite{BW-asens}: 
If $X$ is such a space, then each variety $\VV^i_k(X)$ is a finite union of 
torsion-translated subtori of $\Char(X)$. The key to understanding the degree-$1$ 
cohomology jump loci is the (finite) set $\mathcal{E}(X)$ of holomorphic,  
surjective maps $f\colon X\to \Sigma$ for which the generic fiber is connected 
and the target is a smooth curve $\Sigma$ with $\chi(\Sigma)<0$, up to reparametrization 
at the target. As an application of these techniques, we obtained in \cite{PS-jlms} 
the following result. Let $X$ be a smooth quasi-projective variety with $b_1(X)>0$, 
and let $A$ be a Dupont model for $X$; then $\pi_1(X)$ surjects onto a free, non-cyclic 
free group if and only if $\RR^1_1(A)\ne \{ \zero\}$.

\subsection{Models for compact Lie group actions}
\label{subsec:intro-algmod2}

Let $M$ be a compact, connected, smooth manifold that 
supports an almost free action by a compact, connected Lie group $K$.  
Under a partial formality assumption on the orbit space $M/K$ 
and a regularity assumption on the characteristic classes 
of the action, we constructed in \cite{PS-jlms} an algebraic model for $M$
with commensurate finiteness and partial formality properties.  
The existence of such a model has various implications on the 
structure of the cohomology jump loci of  $M$ and of the 
representation varieties of $\pi_1(M)$.  

In many ways, Sasakian geometry is an odd-dimensional analog of K\"{a}hler 
geometry.  More explicitly, every compact Sasakian manifold $M$ 
admits an almost-free circle action with orbit space 
a K\"{a}hler orbifold.  Furthermore, the Euler class of the action 
coincides with the K\"{a}hler class of the base, $h\in H^2(M/S^1,\Q)$, 
and this class satisfies the Hard Lefschetz property.
As shown by Tievsky in \cite{Ti}, every Sasakian manifold $M$ as above 
has a rationally defined, finite-dimensional model over $\R$ of the form 
$(H^{*}(N,\R)\otimes  \bwedge(t), d)$, 
where the differential $d$ vanishes on $H^{*}(N,\R)$ and sends $t$ to $h$.  
Using this model, it is shown in \cite{PS-jlms}  that compact Sasakian 
manifolds of dimension $2n+1$ are 
$(n-1)$-formal, and that their fundamental groups are filtered-formal. 

\subsection{Models for closed $3$-manifolds}
\label{subsec:intro-algmod3}
With a few exceptions (such as rational homology spheres, knot complements, 
or Seifert manifolds), the rational homotopy theory of $3$-dimensional 
manifolds is very difficult to handle. Part of the reason is that not 
only $3$-manifolds may fail to be formal, and even fail to have a  
$1$-finite $1$-model. Nevertheless, much is known about the Alexander 
polynomial, $\Delta_M$, of a closed, orientable $3$-manifold $M$.
and the way this polynomial relates to the cohomology jump 
loci of $M$, see \cite{DPS-mz}, \cite{FS}, \cite{Su-edinb}, \cite{Su-mm}.  
In turn, these invariants inform on the formality 
and finiteness properties of $M$. 

For instance, we showed in \cite{Su-mm} the following: If 
$b_1(M)$ is even and positive, and if $\Delta_M\ne 0$, then 
$M$ is not $1$-formal. On the other hand, if $\Delta_M\ne 0$, yet 
$\Delta_M(\one) =0$  and the tangent cone to the zero set of $\Delta_M$ 
is not a finite union of rationally defined linear subspaces, then $M$ 
admits no $1$-finite $1$-model. 

When the $3$-manifold $M$ fibers over $S^1$, more can be said. For 
instance, if $b_1(M)$ is even, then, as shown in \cite{PS-plms10}, 
$M$ is not $1$-formal.  On the other hand, if $M$ is $1$-formal 
and the algebraic monodromy has $1$ as an eigenvalue, then, 
as shown in \cite{PS-forum}, there are an even number of 
$1\times 1$ Jordan blocks for this eigenvalue, 
and no higher size Jordan blocks. 

\subsection{Organization}
\label{subsec:intro-org}
The paper in divided in roughly five parts.

Part I (Sections \ref{sect:dga}, \ref{sect:formality}, \ref{sect:min-mod}) 
treats the general theory of (commutative) differential graded algebras, 
formality and its variants, Massey products, descent properties, Hirsch 
extensions, and Sullivan's minimal models.

Part II (Sections \ref{sect:Lie-algebras}, \ref{sect:malcev-lie}, \ref{sect:holo}) 
deals with several of the Lie algebras that appear in this theory (graded and 
filtered Lie algebras, Malcev Lie algebras, and holonomy Lie algebras) and 
discusses some of their properties and interconnections.

Part III (Sections \ref{sect:algmod}, \ref{sect:algmod-groups}, \ref{sect:formal}) 
contains the basics of rational homotopy theory, such as completions, 
rationalizations, and algebraic models for spaces and groups, focussing 
mostly on the formality and finiteness properties of such models.

Part IV (Sections \ref{sect:alex-res}, \ref{sect:cjl}) brings into play 
the Alexander invariants and the cohomology jump loci of spaces 
and suitable algebraic models, and connects the characteristic 
and resonance varieties to various formality and finiteness 
properties.

Part V (Sections \ref{sect:models-qp},  \ref{sect:models-act},  \ref{sect:models-3d}) 
applies these general theories in three particular contexts: that of K\"ahler manifolds 
and smooth, quasi-projective varieties; compact Lie group actions on manifolds; 
and closed, orientable $3$-manifolds. 

\section{Differential graded algebras}
\label{sect:dga}

\subsection{Graded algebras}
\label{subsec:graded-algebras}

Throughout this work, $\k$ will denote a ground field of characteristic $0$.
Unless otherwise specified, all tensor products will be over $\k$. 

A  {\em graded $\k$-vector space}\/ is a vector space $A$ over $\k$, 
together with a direct sum decomposition, $A=\bigoplus_{n\geq 0}A^n$,  
into vector subspaces. An element $a\in A^n$ is said 
to be homogeneous; we write $\abs{a}=n$ for its degree, 
and put $\bar{a}=(-1)^{\abs{a}}a$.

A {\em graded algebra}\/ over $\k$ is a graded $\k$-vector space, 
$A^{*}=\bigoplus_{n\geq 0}A^n$, equipped with an associative 
multiplication map, $\cdot \colon A\times A\to A$, making $A$ 
into a $\k$-algebra with unit $1\in A^0$ such that 
$\abs{a\cdot b}=\abs{a}+\abs{b}$ for all homogenous 
elements $a,b\in A$.  A graded algebra $A$ is said to be 
{\em graded-commutative}\/ (for short, a $\cga$), if 
$a\cdot b = (-1)^{\abs{a}\abs{b}} b \cdot a$ for all 
homogeneous $a,b\in A$. A morphism between two 
graded algebras is a $\k$-linear map $\varphi\colon A\to B$ 
that preserves gradings and satisfies $\varphi(a\cdot b)=\varphi(a)\cdot 
\varphi(b)$ for all $a,b\in A$.

A graded $\k$-algebra $A$ is of {\em finite-type} (or, locally finite) 
if all the graded pieces $A^n$ are finite-dimensional. We say 
that $A$ is {\em $q$-finite} (for some integer $q\ge 0$) if 
$\dim_{\k} A^n<\infty$ for $n \le q$, and we say that $A$ is 
{\em finite-dimensional}\/ (as a $\k$-vector space) if 
$\dim_{\k} A<\infty$. Finally, we say that $A$ is 
{\em connected}\/ if $A^0$ is the $\k$-span of the unit $1$ 
(and thus $A^0=\k$). 

The most basic example of a $\k$-$\cga$ is the {\em free}\/ 
commutative graded algebra on a graded $\k$-vector space $V^*$; 
denoted by $\bwedge V$, this (connected) algebra is 
the tensor product of the symmetric algebra on $V^{\mathrm{even}}$ 
with the exterior algebra on $V^{\mathrm{odd}}$.

\subsection{Differential graded algebras}
\label{subsec:dgas}
The next notion, which plays a key role in the theory described here, 
unifies the concept of a graded algebra with that of a cochain complex.

\begin{definition}
\label{def:dga}
A {\em differential graded algebra}\/ (for short, a $\dga$) over 
a field $\k$ is a graded $\k$-algebra, $A^*$, equipped with a 
differential $d\colon A\rightarrow A$ of degree $1$ satisfying  
the graded Leibniz rule: $d(ab)=d(a)\cdot b+\bar{a}\cdot d(b)$ 
for all homogeneous $a, b\in A$. 
\end{definition}

Viewing $(A,d)$ as a cochain complex, we write 
$Z^n(A)= \ker (d\colon A^n \to A^{n+1})$ for the space 
of $n$-cocycles and $B^n(A)= \im (d\colon A^{n-1} \to A^n)$ 
for the space of $n$-coboundaries, and we let 
$H^n(A)=Z^n(A)/B^n(A)$ be the $n$-th cohomology group 
of $(A,d)$.  The direct sum of those vector spaces, 
$H^{*}(A)=\bigoplus_{i\ge 0} H^i(A)$, 
inherits the structure of a graded algebra from $A$.  When 
$H^*(A)$ is of finite-type, we denote by $b_n(A)=\dim_{\k} H^n(A)$ 
the Betti numbers of $A$. Given an $n$-cocyle $a$, we 
denote by $[a]\in H^n(A)$ its cohomology class.

A {\em commutative differential graded algebra}\/ (for short, a $\cdga$) 
is a $\dga$ $A=(A^*,d)$ for which the underlying graded 
algebra is graded-commutative. In this case, the cohomology 
algebra $H^*(A)$ inherits the structure of a $\cga$.

 If $A$ is a connected 
$\dga$, then the differential $d\colon A^0\to A^1$ vanishes; indeed,  
$d(1)=d(1\cdot1)=d(1)\cdot 1+1\cdot d(1)$, and so $d(1)=0$, 
since $\ch(\k)=0$.  Therefore, $H^0(A)=\k$ and the cohomology 
algebra $H^*(A)$ is also connected.  

\subsection{Weak equivalences}
\label{subsec:qiso}
A morphism between two $\dga$s is a $\k$-linear map, $\varphi\colon A\to  B$, 
which preserves gradings, multiplicative structures, and differentials; in other words,  
$\varphi$ is a map of graded $\k$-algebras which is also a map of cochain 
complexes. Such a map induces a morphism, $\varphi^*\colon H^{*} (A)\to H^{*} (B)$, 
between the respective cohomology algebras.  We say that $\varphi$ is a 
{\em quasi-isomor\-phism}\/ if $\varphi^*$ is an isomorphism. 

A {\em weak equivalence}\/ between two $\dga$s, $A$ and $B$, is a finite 
sequence of quasi-isomor\-phisms (going either way) connecting $A$ to $B$; 
for instance, 
\begin{equation}
\label{eq:ziggy}
\begin{tikzcd}
A  & A_1 \ar[l, pos=0.4, "\varphi_1"'] \ar[r, "\varphi_2"] & \cdots 
& A_{\ell-1}   \ar[l] \ar[r, "\varphi_{\ell}"] & B .
\end{tikzcd}
\end{equation}
Note that a weak equivalence induces a well-defined isomorphism 
$H^*(A)\cong H^*(B)$. If a  weak equivalence between $A$ and $B$ 
exists, the two $\dga$s are said to be {\em weakly equivalent}, 
written $A\simeq B$. Evidently, $\simeq$ is an equivalence 
relation among $\dga$s. 

An analogous notion holds in the category of commutative $\dga$s.  
Namely, if $A$ and $B$ are two $\cdga$s, we say that $A\simeq B$ 
is there is a zig-zag of quasi-isomorphisms as in \eqref{eq:ziggy} 
that go through $\cdga$s $A_i$.
The following theorem resolves a long-standing question, 
by showing that weak equivalence among $\cdga$s holds even if one 
allows the zig-zags from \eqref{eq:ziggy} to go through $\dga$s. 

\begin{theorem}[\cite{CPRW}]
\label{thm:cprw}
Let $A$ and B be two $\k$-$\cdga$s. Then 
$A \simeq B$ as $\dga$s if and only if $A \simeq B$ as $\cdga$s.
\end{theorem}

All these concepts have partial analogues. Fix an integer $q\ge 0$. 
We say that a $\dga$ (or $\cdga$) morphism $\varphi\colon A\to B$ is a 
{\em $q$-quasi-isomorphism}\/ if $\varphi^*$ is an isomorphism 
in degrees up to $q$ and a monomorphism in degree $q+1$. 
A {\em $q$-equivalence}\/ between two $\dga$s (or $\cdga$s), 
$A$ and $B$, is a zig-zag of $q$-quasi-isomorphisms of $\dga$s 
(or $\cdga$s) connecting $A$ to $B$. If such a zig-zag exists,  
we say that $A$ and $B$ are {\em $q$-equivalent}\/ 
and write this as $A\simeq_{q} B$. Again, $\simeq_q$ 
is an equivalence relation among either $\dga$s or 
$\cdga$s. We do not know whether Theorem \ref{thm:cprw} 
holds with $\simeq$ replaced by $\simeq_q$, for arbitrary $q$, but 
we expect it does.

Clearly, if $A\simeq B$, then $A\simeq_q B$ for all $q\ge 0$, and if 
$A\simeq_q B$, then $A\simeq_n B$ for all $n\le q$. 
Moreover, if $A$ is of finite-type and $A\simeq B$, then 
$B$ is also of finite-type and $b_n(A)=b_n(B)$ for all $n\ge 0$. 
Likewise, if $A$ is $q$-finite and $A\simeq_q B$, then 
$B$ is also $q$-finite and $b_n(A)=b_n(B)$ for all $n\le q$. 
The next lemma shows that every $q$-finite $\cdga$
may be replaced (up to $q$-equivalence) by a 
finite-dimensional one, whose graded pieces 
vanish above degree $q+1$.

\begin{lemma}[\cite{PS-jlms}]
\label{lem:truncate}
Let $A$ be a $q$-finite $\cdga$. There is then a natural $q$-equivalence, 
$A\simeq_q A[q]$, where $A[q]$ is a finite-dimensional $\cdga$ with 
$A[q]^i=0$ for all $i>q+1$.
\end{lemma}

The construction of $A[q]$ is done in two steps: first one replaces $A$ 
by its truncation, $\overline{A}=A/A^{>q+1}=\bigoplus_{i\le q+1}A^i $, 
and then one defines a sub-$\cdga$ $A[q]\subset \overline{A}$ by setting 
\begin{equation}
\label{eq:truncate}
A[q]=\bigoplus_{i\le q} A^i \oplus \bigg(dA^q + 
\sum_{\substack{i,j \le q\\ i+j =q+1}} A^i \cdot A^j\bigg).
\end{equation}
An analogous result holds for $\dga$s.

\subsection{Homotopies and equivalences}
\label{subsec:homeq} 
Let $A$ be a $\k$-$\dga$, and let $ \bwedge (t, dt)$ 
be the free $\k$-$\dga$ generated by elements $t$ in degree 
$0$ and $u$ in degree $1$, and differential $d$ given by $d(t)=u$ 
and $d(u)=0$. 
For each $s\in \k$, let $\ev_s\colon  \bwedge (t, dt) \to \k$ be the 
$\dga$ map sending $t$ to $s$ and $dt$ to $0$.  This 
map induces another $\dga$ map,
\begin{equation}
\label{eq:Evs dga}
\begin{tikzcd}[column sep=18pt]
\Ev_s \coloneqq \id\otimes \ev_s \colon A \otimes \bwedge (t, dt) \ar[r]&
A \otimes \k =A. 
\end{tikzcd}
\end{equation}

Two $\dga$ maps, $\varphi_0, \varphi_1\colon A\to B$, are 
said to be {\em homotopic}\/ if there is a $\dga$ map, 
$\Phi\colon A\to B \otimes \bwedge (t, dt)$, 
such that $\Ev_s\circ \Phi = \varphi_s$ for $s=0,1$. 
It is readily seen that homotopic $\dga$ maps induce the same 
map in cohomology. 

We say that two $\dga$ morphisms, $\varphi\colon A \to B$ and 
$\varphi'\colon A' \to B'$, are {\em weakly equivalent}\/ (written 
$\varphi\simeq \varphi'$) if there are two zig-zags of 
weak equivalences of $\dga$s, and $\dga$ maps 
$\varphi_1, \dots, \varphi_{\ell -1}$ such that 
the following diagram commutes, up to homotopy:
\begin{equation}
\label{eq:ziggy-zagg}
\begin{tikzcd}[row sep=22pt]
A \ar[d, "\varphi"] & A_1 \ar[d, "\varphi_1"] \ar[l, "\psi_1"']  \ar[r, "\psi_2"] 
& \cdots & A_{\ell-1}  \ar[d, "\varphi_{\ell-1}"] \ar[l]\ar[r, "\psi_{\ell}"] 
& A' \ar[d, "\varphi'"]\, \phantom{.}
\\
B& B'_1 
 \ar[l, "\psi'_1"']  \ar[r, "\psi'_2"] & \cdots 
& B'_{\ell-1} \ar[l]\ar[r, "\psi'_{\ell}"] & B'  .  
\end{tikzcd}
\end{equation}
The notion of {\em $q$-equivalence}\/ (written 
$\varphi\simeq_q \varphi'$) is defined similarly, 
and so are the analogous notions in the $\cdga$ category.

\subsection{Poincar\'{e} duality}
\label{subsec:pd}
Let $A$ be a finite-dimensional, 
commutative graded algebra over a field $\k$ of characteristic $0$. 
We say that $A$ is a {\em Poincar\'{e} duality algebra}\/ of dimension $n$ 
(for short, an $n$-\textsc{pda}) if $A^i=0$ for $i>n$ and $A^n=\k$, 
while the bilinear forms $A^{i}\otimes A^{n-i}\to A^n=\k$ 
given by the product are non-degenerate, for all $0\le i\le n$ 
(in particular, $A$ is connected).  If $M$ is a closed, 
connected, orientable, $n$-dimensional manifold, then, 
by Poincar\'{e} duality, the cohomology algebra $A=H^{*}(M,\k)$ 
is an $n$-\textsc{pda}. 

Now let $A=(A^{*}, d)$ be a $\cdga$.  We say that $A$ is a 
{\em Poincar\'{e} duality differential graded algebra}\/ of dimension $n$ 
(for short, an $n$-\textsc{pd-cdga}) if the underlying algebra $A$ is an 
$n$-\textsc{pda}, and, moreover,  $H^n(A)=\k$, or, equivalently, 
$d A^{n-1}=0$. 

Clearly, if $A$ is an $n$-\textsc{pda}, then $(A,0)$ is an 
$n$-\textsc{pd-cdga}.  Hasegawa showed in \cite{Has} that 
the minimal model for the classifying space of a finitely-generated 
nilpotent group is a \textsc{pd-cdga}. Noteworthy is the following 
result of Lambrechts and Stanley \cite{LS-asens} 

\begin{theorem}[\cite{LS-asens}]
\label{thm:LS}
Let $(A,d)$ be a $\cdga$ such that $H^1(A)=0$ and $H^{*}(A)$ is an $n$-\textsc{pda}.
Then $A$ is weakly equivalent to an $n$-\textsc{pd-cdga}.
\end{theorem}

\section{Formality}
\label{sect:formality}

\subsection{Formal $\dga$s}
\label{subsec:formal-dga}
In this section, we cover the notion of {\em formality}. Introduced in 
\cite{DGMS}, \cite{Sullivan77} and further developed in 
\cite{HS79}, \cite{Mo}, \cite{GM13}, \cite{Tanre}, \cite{PS-formal}, 
\cite{Mc10}, \cite{Sa17}, \cite{SW-forum}, 
and many other works, this notion plays a central role 
in rational homotopy theory.

\begin{definition}[\cite{DGMS}, \cite{Sullivan77}]
\label{def:formal}
A $\dga$ $(A,d_A)$ is said to be {\em formal}\/ if it is weakly equivalent to 
$(H^{*}(A),0)$, its cohomology algebra endowed with the zero differential. 
\end{definition}

Note that $A$ is formal if and only if it is weakly equivalent to some 
$\dga$ $B$ with zero differential, since, in that case, we necessarily have 
$(B,0)\simeq (H^{*}(A),0)$.  In like manner, we say that a $\cdga$ 
$(A,d_A)$ is formal if $(A,d_A)\simeq (H^*(A),0)$ via a weak 
equivalence through $\cdga$s. 

\begin{example}
\label{ex:non-formal-cdga}
Let $A=\bwedge(a_1,a_2,b)$ be the 
free $\cga$ on generators $a_1,a_2$ in degree $n$ and $b$ in degree 
$2n-1$, equipped with the differential $d$ given by $d a_i=0$ and 
$d b=a_1a_2$. If $n\ge 2$, the cdga $(A,d)$ is not formal, since 
$H^*(A)$ is generated by the elements $u_i=[a_i]$, 
and so any weak equivalence from $(A,d)$ to $(H^*(A), 0)$ 
would need to take the non-zero element $b$ to $0$, by degree reasons.
\end{example}

Formality is preserved under weak equivalences; that is, if $A\simeq B$, 
then $A$ is formal if and only if $B$ is formal. Moreover, as shown by 
Halperin and Stasheff in \cite{HS79}, a $\cdga$ $(A,d)$ with $H^{*}(A)$ 
of finite-type is formal if and only if the identity map of $H^{*}(A)$ 
can be realized by a weak equivalence between $(A,d)$ and $(H^{*}(A),0)$.

The next result, originally proved directly by Salehi in \cite{Sa17}, now 
follows at once from Theorem \ref{thm:cprw}.

\begin{corollary}[\cite{Sa17}]
\label{cor:salehi}
Let $A$ be a $\k$-$\cdga$. Then $A$ is formal as a $\dga$ if 
and only if $A$ is formal as a $\cdga$.
\end{corollary}

\subsection{Intrinsic formality}
\label{subsec:intrinsic-formal}
We now present two variants of formality, the first being a rigid type  
of formality and the second formality up to a degree. 

A strong form of formality was introduced by Sullivan 
in \cite{Sullivan77}, and developed in \cite{HS79}, \cite{FH82}, and 
\cite{Lupton}.  We say that a $\k$-cga $H$ is {\em intrinsically formal}\/ 
if any $\k$-$\dga$ $(A,d)$ whose cohomology is isomorphic to 
$H$ must be formal, that is, $(A,d)\simeq (H, 0)$. 
A similar notion holds for $\cdga$s; by Theorem \ref{thm:cprw}, 
if $H$ is intrinsically formal in the category of $\dga$s,  
it is also instrinsically formal in the category of  $\cdga$s.
Plainly, if $A$ is a $\dga$ or a $\cdga$ such that $H^*(A)$ 
is intrinsically formal, then $A$ is formal. 
The following results of Sullivan and Halperin--Stasheff provide 
large classes of intrinsically formal $\cga$s.  

\begin{theorem}[\cite{Sullivan77}]
\label{thm:regular-intrinsic}
Let $H$ be the quotient of a finitely generated, free $\cga$ 
by an ideal generated by a regular sequence, that is, a sequence 
$r_1,\dots, r_n$ of homogeneous elements in $H$ such that $r_i$ 
is not a zero-divisor in $H/(r_1,\dots, r_{i-1})$, for each $i\le n$. 
Then $H$ is intrinsically formal.
\end{theorem}

Algebras of the form $H=\bwedge V/(r_1,\dots, r_{n})$ as above 
are sometimes called {\em hyperformal}, see \cite{FH82}, \cite{Lupton}. 
In particular, exterior algebras and polynomial algebras are 
hyperformal, and thus intrinsically formal.

\begin{theorem}[\cite{HS79}]
\label{thm:odd-intrinsic}
Let $H$ be a connected $\cga$ such that $H^i=0$ for $1\le i\le k$ 
and for $i>3k+1$. Then $H$ is intrinsically formal.
\end{theorem}

\subsection{Partial formality}
\label{subsec:partial-formal}
The notion of formality can also be relaxed, as follows. Fix an integer 
$q\ge 0$. We say that a $\dga$ (or a $\cdga$) $A$ is 
{\em $q$-formal}\/ if $(A,d_A)\simeq_q (H^*(A),0)$ as $\dga$s 
(or $\cdga$s). We do not know whether the analog of 
Corollary \ref{cor:salehi} holds in this context, but we 
expect it does,

Clearly, if $A$ is formal, then $A$ is $q$-formal, 
for all $q\ge 0$, and if $A$ is $q$-formal, then it is $n$-formal 
for every $n\le q$.  It is readily seen that connected $\dga$s 
are $0$-formal. Moreover, if $A\simeq_q B$, then $A$ is 
$q$-formal if and only if $B$ is $q$-formal.  

\begin{example}
\label{ex:gen-heisenberg}
Let $A=\bwedge(a_1,\dots, a_{2q}, b)$ be the 
exterior algebra on specified generators in degree $1$, 
equipped with the differential $d$ given by $d a_i=0$ and 
$d b=a_1a_2+\cdots +a_{2q-1}a_{2q}$.  It follows from  
\cite[Remark 5.4]{Mc10} that the cdga $(A,d)$ is $(q-1)$-formal 
but not $q$-formal.
\end{example}

We refer to M\u{a}cinic \cite{Mc10} for a more thorough discussion of partial 
formality and related notions (see also \cite{PS-formal}, \cite{SW-forum}).

\subsection{Field extensions and formality}
\label{subsec:partial-formal-ext}

As is well-known, a finite-type, rationally defined $\cdga$ is formal 
 if and only it is formal over any field of characteristic $0$.  
This foundational result was proved independently and in 
various degrees of generality by Sullivan \cite{Sullivan77}, 
Neisendorfer and Miller \cite{NM78}, and Halperin and 
Stasheff \cite{HS79}. We conclude this section with a 
discussion of this topic and some recent generalizations 
from \cite{SW-forum} to partially formal $\cdga$s.

Given a $\dga$~$(A,d)$ over a field $\k$ of characteristic $0$ 
and a field extension $\k\subset \K$, we let $(A\otimes_{\k} \K, d\otimes \id_\K)$ 
be the $\cdga$ over $\K$ obtained by extending scalars.

\begin{theorem}[\cite{HS79}] 
\label{thm:Halperin-Stasheff}
Let $(A,d_A)$ and $(B,d_B)$ be two $\cdga$s over $\k$ 
whose cohomology algebras are connected and of finite type. 
Suppose there is an isomorphism of graded algebras, $f\colon H^{*}(A)\to H^{*}(B)$, 
and suppose  $f\otimes \id_{\K}\colon H^{*}(A)\otimes \K\to H^{*}(B)\otimes \K$
can be realized by a weak equivalence between $(A\otimes \K,d_A\otimes \id_{\K})$ 
and $(B\otimes \K,d_B\otimes \id_{\K})$. Then $f$ can be realized
by a weak equivalence between $(A,d_A)$ and $(B,d_B)$.
\end{theorem}

This theorem has the following corollary. The result is stated without proof 
in \cite[Corollary 6.9]{HS79}; a complete proof is provided in 
\cite[Corollary 4.17]{SW-forum}. Self-contained proofs under 
some additional hypotheses were previously given 
in \cite[Theorem 12.1] {Sullivan77} and \cite[Corollary~5.2]{NM78}.

\begin{corollary}[\cite{HS79}]
\label{cor:Kformal}
Let $A=(A,d_A)$ be a connected $\k$-$\cdga$ with $H^{*}(A)$ of finite-type. 
Then $A$ is formal if and only if the $\K$-$\cdga$~$(A\otimes \K,d_A\otimes \id_{\K})$ 
is formal.
\end{corollary}

These classical formality results were generalized in \cite[Theorem 4.19]{SW-forum}, 
which extends the descent-of-formality results of Sullivan, Neisendorfer--Miller, 
and Halperin--Stasheff to the partially formal setting.

\begin{theorem}[\cite{SW-forum}]
\label{thm:i-formalField}
Let $(A,d_A)$ be a $\cdga$~over $\k$, and let $\k\subset \K$ be a 
field extension. Suppose $H^{\le q+1}(A)$ is finite-dimensional 
and $H^0(A)=\k$.  Then $(A,d_A)$ is $q$-formal 
if and only if $(A\otimes \K,d_A\otimes \id_{\K})$ 
is $q$-formal.
\end{theorem}

\subsection{Formality of $\dga$ maps}
\label{subsec:formal-dga-maps} 
The notion of formality may be extended from the objects to the morphisms 
of the $\dga$ category, as follows.

\begin{definition}
\label{def:formal-map}
A $\dga$ morphism $\varphi\colon A\to B$ is said to be {\em formal}\/ if 
there is a diagram of the form \eqref{eq:ziggy-zagg} connecting $\varphi$ 
to the induced homomorphism $\varphi^*\colon H^*(A)\to H^*(B)$ between 
cohomology algebras (viewed as $\dga$s with zero differentials). 
Likewise, $\varphi$ is said to be {\em $q$-formal}, for some $q\ge 0$, if 
$\varphi\simeq_q \varphi^*$.
\end{definition}

Note that in the first case both $A$ and $B$ need to be formal $\dga$s, 
while in the second case both $A$ and $B$ need to be $q$-formal. 
Also note that if $\varphi$ is formal and $\varphi\simeq \varphi'$, 
then $\varphi'$ is also formal, and similarly for $q$-formality.

Completely analogous notions may be defined for $\cdga$ morphisms. 
although we do not know whether a statement analogous 
to Theorem \ref{thm:cprw} holds in this context.  Nevertheless, 
a descent of formality result analogous to Corollary \ref{cor:Kformal}
holds.

\begin{theorem}[\cite{Sullivan77}, \cite{VP79}]
\label{thm:formal-maps}
Let $\varphi\colon A \to B$ be a morphism between 
connected $\k$-$\cdga$s with finite Betti numbers, 
and let $\k\subset \K$ be a field extension.
Then $\varphi$ is formal if and only if 
$\varphi \otimes \id_{\K}\colon A\otimes \K\to B\otimes \K$ is formal.
\end{theorem}

We do not know whether a statement in the spirit of 
Theorems \ref{thm:i-formalField} and \ref{thm:formal-maps} 
holds for $q$-formal maps. 

\begin{example}
\label{ex:hopf-model}
Fix an even integer $n\ge 2$ and consider the $\cdga$ morphism 
$\varphi\colon (A,d)\to (B,0)$, where $A=\bwedge(a,b)$, 
with $\abs{a}=n$, $\abs{b}=2n-1$, and differential given 
by $d(a)=0$ and $d(b)=a^2$; $B=\bwedge(c)$ with $\abs{c}=2n-1$; 
and  $\varphi$ is given by $\varphi(a)=0$ and $\varphi(b)=c$. 
Then $H^*(A)=\bwedge(u)$ with $\abs{a}=n$, and the 
$\cdga$ map $\psi \colon A\to H^*(A)$ given by $\psi(a)=u$ and 
$\psi(b)= 0$ induces the identity in cohomology. Nevertheless, the map 
$\varphi^* \colon \widetilde{H}^*(A)\to \widetilde{H}^*(B)$ is trivial, 
and so the morphism $\varphi$, which is non-trivial, cannot be formal.
\end{example}  

\subsection{Massey products}
\label{subsec:massey}
A well-known obstruction to formality is provided by the higher-order 
Massey products, introduced by W.S.~Massey in \cite{Ma-58}, and 
studied for instance in \cite{Kraines}, \cite{May}, \cite{Porter}, \cite{TO}, 
\cite{RT}, \cite{BT-2}, and \cite{Porter-Suciu-topappl}.  

Let $(A,d)$ be a $\k$-$\dga$ and let $u_1, \dots, u_n$ 
be elements in $H^{*}(A)$; without loss of essential generality, 
we may assume that $n\ge 3$ and each $u_i$ is homogeneous 
and of positive degree.  
A {\em defining system}\/ for $u_1, \dots, u_n$ is a collection 
of elements $a_{i,j}\in A$ such that 
$d a_{i-1,i}=0$ and $[a_{i-1,i}]=u_i$ for $1\le i\le n$ and 
$d a_{i,j}=\sum_{i<r<j} \bar{a}_{i,r} a_{r,j}$ for 
$0\le i<j\le n$ and $(i,j)\ne (0,n)$.
It is readily verified that the element 
\begin{equation}
\label{eq:massey-prod}
\alpha\coloneqq \sum_{0<r<n} \bar{a}_{0,r} a_{r,n}
\end{equation}
is a cocycle, of degree $\abs{\alpha}=2-n+\sum_{i=1}^{n} \abs{u_i}$. 
The $n$-fold Massey product $\langle u_1, \dots, u_n \rangle$, 
then, is the subset of $H^{*}(A)$ consisting of the cohomology classes 
$[\alpha]$ corresponding to all possible defining systems for $u_1, \dots, u_n$.
We say that the Massey product is {\em defined}\/ if there is at least one such 
defining system, or, equivalently, $\langle u_1, \dots, u_n \rangle\ne \emptyset$,  
in which case the indeterminancy of the Massey product is the subset 
$\{u-v\mid u,v\in \langle u_1, \dots, u_n \rangle\}\subseteq H^*(A)$.
When a Massey product is defined, we say it {\em vanishes}\/ if it contains 
the element $0$; otherwise, it is non-vanishing.

The simplest Massey triple products are as follows. 
Let $u_1, u_2, u_3$ be elements in $H^1(A)$ 
such that $u_1u_2 = u_2 u_3 = 0$. We may then choose 
$1$-cocycles $a_{0,1}, a_{1,2}, a_{2,3}$  representing $u_1, u_2, u_3$  
and $1$-cochains $a_{0,2}$ and $a_{1,3}$ such that 
$da_{0,2} = - a_{0,1} a_{1,2}$ and $da_{1,3} = - a_{1,2} a_{2,3}$, 
The triple product $\langle u_1,u_2, u_3 \rangle$ is then  
the subset of $H^2(A)$ consisting of the cohomology classes 
$-[a_{0,1} a_{1,3} + a_{0,2}a_{2,3}]$, for all such choices of 
defining systems. Due to the ambiguity in the choices made, 
$\langle u_1, u_2, u_3\rangle$ may be viewed as a coset 
of $u_1 \cdot H^1(A)+H^1(A)\cdot u_3$ in $H^2(A)$.

\begin{example}
\label{ex:non-formal-cdga-again}
Let $(A,d)$ be the cdga from Example \ref{ex:non-formal-cdga} with $n=1$; 
namely, $A$ is the exterior algebra on generators $a_1,a_2,b$ in degree $1$ 
and differential given by $d a_i=0$ and $d b=a_1a_2$. Letting $u_i=[a_i]\in H^1(A)$, we 
have that the triple Massey products $\langle u_1, u_1, u_2\rangle=\{[a_1b]\}$ and 
$\langle u_1, u_2, u_2\rangle=\{[ba_2]\}$ are defined, have $0$ indeterminacy, 
and are non-van\-ishing; in fact, the two cohomology classes generate $H^2(A)$.
Therefore, $A$ is not formal.
\end{example}

Massey products enjoy the following (partial) functoriality properties.

\begin{proposition}[\cite{Kraines}, \cite{May}]
\label{prop:massey-natural}
 Let $\varphi\colon A\to B$ be a $\dga$ morphism, and let 
$\varphi^*\colon H^*(A)\to H^*(B)$ be the induced morphism 
in cohomology; then 
\begin{equation}
\label{eq:massey-func}
\varphi^*(\langle u_1, \dots, u_n\rangle) 
\subseteq \langle \varphi^*(u_1), \dots, \varphi^*(u_n)\rangle .
\end{equation}
Moreover, if $\varphi$ is a quasi-isomorphism, then \eqref{eq:massey-func} 
holds as equality.
\end{proposition}

In particular, if $\langle u_1, \dots, u_n\rangle$ is defined, then 
$\langle \varphi^*(u_1), \dots, \varphi^*(u_n)\rangle$ is also defined; 
and if, in addition, $\langle \varphi^*(u_1), \dots, \varphi^*(u_n)\rangle$ is 
non-vanishing, then $\langle u_1, \dots, u_n\rangle$ 
is also non-vanishing. As another consequence, the following holds: 
if $A\simeq B$, then all Massey products in $H^*(A)$ vanish 
if and only if all Massey products in $H^*(B)$ vanish.

Finally, if the map $\varphi\colon A\to B$ is a $q$-quasi-isomorphism, 
for some $q\ge 0$, then \eqref{eq:massey-func} holds as equality in 
degrees up to $q+1$. Thus, if $A\simeq_{q} B$, then all Massey 
products in $H^{\le q+1}(A)$ vanish if and only if all Massey 
products in $H^{\le q+1}(B)$ vanish.

\subsection{Massey products and formality}
\label{subsec:massey-formal}

The vanishing of Massey products provides a well-known 
obstruction to formality. An analogous statement 
holds for partial formality. For completeness, we 
make a formal statement and sketch the proof.

\begin{proposition}
\label{prop:massey-formal}
Let $(A,d)$ be a $\k$-$\dga$.  If $A$ formal, then all Massey 
products in $H^{*}(A)$ vanish. Furthermore, if $A$ is $q$-formal, 
for some $q\ge 1$, then all Massey products in $H^{\le q+1}(A)$ vanish.
\end{proposition}

\begin{proof}
First suppose $d=0$, so that $H^*(A)=A$. Let 
$\langle u_1, \dots, u_n\rangle$ be a Massey product. 
We may then choose a defining system with $a_{i-1,i}=u_i$ 
and all other $a_{i,j}=0$, which implies that the cocycle 
$\alpha$ from \eqref{eq:massey-prod} is equal to $0$; 
thus, $\langle u_1, \dots, u_n\rangle$ vanishes. 

Now suppose $(A,d)$ is formal, that is, $(A,d)\simeq (B,0)$. 
As we just saw, all Massey products vanish in $H^*(B)$; hence, 
they must also vanish in $H^*(A)$. The case when $(A,d)$ is 
$q$-formal is treated completely analogously.
\end{proof}

In general, formality is stronger than the mere vanishing 
of all Massey products; in fact, it is equivalent to the
{\em uniform}\/ vanishing of all such products. This 
phenomenon will be illustrated in Theorem \ref{thm:nf-Massey-vanish}, 
where we shall see examples of non-formal $\cdga$s for which 
all Massey products vanish.

\section{Minimal models}
\label{sect:min-mod}

\subsection{Hirsch extensions}
\label{subsec:hirsch}
Given a graded $\k$-vector space $V^*$, recall that $\bwedge V$  
denotes the free graded, graded-com\-mutative algebra generated by $V$. 
Choosing a homogeneous basis $\XX=\{x_{\alpha}\}_{\alpha\in J}$ 
for $V$, this algebra may be identified with 
$\bwedge \XX\coloneqq \bigotimes_{\alpha} \bwedge (x_{\alpha})$, 
where $\bwedge (x)$ is the exterior (respectively, polynomial) 
algebra on a single generator $x$ of odd (respectively, even) degree. 

Now let $A=(A^*,d_A)$ be an arbitrary $\cdga$. 
A {\em Hirsch extension}\/ $A$ (of degree $n$) is an inclusion, 
$(A,d_A)\inj \left(A\otimes \bwedge V,d\right)$, where 
$V$ is a $\k$-vector space concentrated in 
degree $n$ and the differential $d$ sends $V$ into $A^{n+1}$. 
We say this is a {\em finite}\/ Hirsch extension if $V$  
is finite-dimensional.  The next lemma follows straight 
from the definitions.

\begin{lemma}
\label{lem:hi}
Let $\alpha\colon (A,d_A)\inj (A\otimes \bwedge V,d)$ be the 
inclusion map of a Hirsch extension of degree $n+1$.  
Then $\alpha$ is an $n$-quasi-isomorphism.  
\end{lemma}

Now suppose $V$ is an oddly-graded, finite-dimensional 
vector space, with homogeneous basis $\{ t_i \in V^{m_i} \}$. Given a 
degree $1$ linear map, $\tau\colon V^{*} \to Z^{*+1} (A)$, 
we define the corresponding Hirsch extension as the $\cdga$ 
$\left(A\otimes_{\tau} \bwedge V , d\right)$ where the differential 
$d$ extends the differential on $A$, while $dt_i=\tau (t_i)$. 

\begin{proposition}[\cite{Le}]
\label{prop:hirsch}
The isomorphism type of the $\cdga$  $(A\otimes_{\tau} \bwedge V, d)$ 
depends only on $A$ and the homomorphism 
induced in cohomology, $[\tau ]\colon V^{*} \to H^{* +1} (A)$. 
Moreover, $[\tau ]$ and $[\tau ]\circ g$ give isomorphic extensions, for any
automorphism $g$ of the graded vector space $V$.
\end{proposition}

The above result is proved in \cite[Lemmas II.2 and II.3] {Le} 
in the case when all the degrees $m_i$ are equal;  
the same argument works in the general case.

\begin{proposition}[\cite{PS-imrn19}]
\label{prop:circlepd}
Let $A= B \otimes_{\tau} \bwedge (t_i)$ be a Hirsch extension with 
variables $t_i$ of odd degree $m_i$. If $B$ is an $n$-\textsc{pd-cdga}, then
$A$ is an $m$-\textsc{pd-cdga}, where $m=n+ \sum m_i$.
\end{proposition}

\subsection{Minimal $\cdga$s}
\label{subsec:min-cdga}
The following key definition is due to Sullivan \cite{Sullivan77}.

\begin{definition}[\cite{Sullivan77}]
\label{def:min-x}
A $\cdga$ $A=(A^{*},d)$ is said to be {\em minimal}\/  if 
the following conditions are satisfied.
\begin{enumerate}
\item \label{m1} 
$A=\bwedge \XX$ is the free $\cga$ on positive-degree 
generators $\XX=\{x_{\alpha}\}_{\alpha\in J}$ indexed by a well-ordered 
set $J$
\item \label{m2} 
$d x_{\alpha}\in \bwedge (\{x_{\beta} \mid \beta<\alpha\})$ for all $\alpha$.
\item  \label{m3} 
$d x_{\alpha}\in \bwedge^{\! +}\XX\cdot \bwedge^{\! +}\XX$ for all $\alpha$, 
where $\bwedge^{\! +}\XX$ is the ideal generated by $\XX$.
 \end{enumerate}
 \end{definition}
Letting $V^{*}$ be the graded vector space generated 
by the set $\XX$, we may also write $A=(\bwedge V, d)$. 
We say that $A$ is {\em $q$-minimal}\/ (for some $q\ge 1$) 
if $A$ is minimal and $V^i=0$ for all $i>q$, or, equivalently, 
$\deg(x_{\alpha})\le q$ for all $\alpha$.

Here is an alternative interpretation of this notion. The $\cdga$ 
$(A,d)$ is minimal if $A=\bigcup_{j\ge 0}A_j$, where 
$A_0=\k$, each $A_{j}$ is a Hirsch extension of $A_{j-1}$, 
and the differential $d$ is decomposable, i.e.,  
$dA\subset A^+\wedge A^+$, where $A^+=\bigoplus_{n\ge 1}A^{n }$.
This yields an increasing, exhausting filtration of $A$ by the sub-$\dga$s $A_j$. 
The decomposability of the differential is automatically satisfied if $A$ is 
generated in degree $1$.

The next lemma illustrates some of the usefulness of the notion 
of $1$-minimality.

\begin{lemma}[\cite{DeS18}]
\label{lem:surj}
Let $A$ be a $1$-minimal $\cdga$, and let $\varphi, \psi\colon A\to B$ 
be two homotopic $\cdga$ morphisms.  
Suppose $d_B=0$ and  $\varphi^1\colon A^1\to B^1$ is surjective.  
Then $\psi^1\colon A^1\to B^1$  is also surjective.  
\end{lemma}

\subsection{Minimal models}
\label{subsec:min-mod}
Let $A$ be a $\cdga$. We say that a $\cdga$ $\M$ is a {\em minimal model}\/ for $A$ 
if $\M$ is a minimal $\cdga$ and there exists a quasi-isomorphism $\rho\colon \M \to A$. 
Likewise, we say that a minimal $\cdga$ $\M$ is a {\em $q$-minimal model}\/ for 
$A$ if $\M$ is generated by elements of degree at most $q$, and there 
exists a $q$-quasi-isomorphism $\rho\colon \M \to A$. A basic result in 
rational homotopy theory is the following existence and uniqueness theorem, 
first proved for minimal models by Sullivan \cite{Sullivan77}, and for 
partial minimal models by Morgan \cite{Mo}. 

\begin{theorem}[\cite{Mo}, \cite{Sullivan77}]
\label{thm:mm} 
Let $A$ be a $\k$-$\cdga$ with $H^{0}(A)=\k$ . Then $A$ admits a minimal model, 
$\M(A)$, unique up to isomorphism. Likewise, for each $q\ge 0$, there is a 
$q$-minimal model, $\M_q(A)$, unique up to isomorphism.
\end{theorem}

By construction, $\M(A)=(\bwedge V,d)$ and $\M_q(A)=(\bwedge V^{\! \le q},d)$, 
for some graded vector space $V$. It follows that the minimal model  
$\M(A)$ is isomorphic to a minimal model built from the $q$-minimal 
model $\M_q(A)$ by means of Hirsch extensions in degrees $q+1$ 
and higher.  Thus, in view of Lemma \ref{lem:hi}, $\M_q(A)\simeq_{q} \M(A)$.  

Applying Lemma \ref{lem:truncate}, we obtain the following finiteness 
criterion for $\cdga$s.

\begin{proposition}
\label{prop:betti-minimal}
Let $A$ be a $q$-finite $\cdga$. Then $b_{i}(\mathcal{M}_q(A))<\infty$ for all $i\le q+1$.
\end{proposition}

The minimal model comes with a structural 
quasi-isomorphism, $\rho\colon \M(A)\to A$. If $\rho'\colon \M'(A)\to A$ 
is another minimal model, there is an isomorphism $\psi\colon \M(A)\isom \M'(A)$ 
such that $\rho'\circ \psi\simeq \rho$. Furthermore, the minimal model is functorial: 
if $\varphi\colon A\to B$ is a morphism between two $\cdga$s with 
connected homology, there is an induced morphism of $\cdga$s, 
$\M(\varphi)\colon \M(A)\to \M(B)$, such that $\rho_B\circ \M(\varphi) \simeq
\varphi\circ \rho_A$. Similar results hold for the partial minimal models.

The above considerations imply the following: 
two $\cdga$s with connected homology are weakly isomorphic  
if and only if their minimal models are isomorphic.  
Alternatively, if $A$ and $A'$ are two $\cdga$s with connected homology, 
then $A\simeq A'$  if and only if there is a minimal 
$\cdga$ $\M$ and a short zig-zag of quasi-isomorphisms, 
\begin{equation}
\label{eq:amqa}
\begin{tikzcd}[column sep=24pt]
A  & \mathcal{M}\ar[r, "\rho'"]  \ar[l, "\ \rho" ']  & A'.
\end{tikzcd}
\end{equation}
Analogous results hold for $q$-minimal models.

\subsection{Minimality and formality}
\label{subsec:minmod-formal}

In \cite{DGMS}, Deligne, Griffiths, Morgan, and Sullivan gave a very practical 
interpretation of formality in the context of minimal $\cdga$s.

\begin{theorem}[\cite{DGMS}]
\label{thm:formal-dgms}
Let $A=(\bwedge V,d)$ be a minimal $\cdga$. Then $A$  is formal 
if and only if each subspace $V^i=A^i\cap V$ decomposes as a 
direct sum, $V^i=N^i\oplus C^i$, where $C^i=Z^i(A)\cap V$
and any cocycle in the ideal of $A$ generated by $\bigoplus N^i$ is 
a coboundary.
\end{theorem}

As noted in \cite{DGMS}, choosing complements $N^i$ to $C^i$ with  
the specified property is equivalent to choosing a $\cdga$-morphism 
$(A,d)\to (H^*(A),0)$ inducing the identity in cohomology.  Furthermore, 
the existence of splittings $V^i=N^i\oplus C^i$ such that any cocycle in 
the ideal generated by $\bigoplus_i N^i$ is a coboundary is one way of 
saying that one may make uniform choices of subspaces spanned by 
defining systems so that all the cocycles representing Massey products 
are coboundaries---a stronger condition than saying that each 
individual Massey product vanishes.

Work of Sullivan \cite{Sullivan77} and Morgan \cite{Mo} shows that a 
$\cdga$ $(A,d)$ is formal if and only of there exists a quasi-isomorphism 
$\psi\colon \M(A)\to (H^{*}(A),0)$. Likewise, M\u{a}cinic showed in 
\cite{Mc10} 
that $A$ is $q$-formal if and only if there exists a 
$q$-quasi-isomorphism $\psi_q\colon \M_q(A)\to (H^{*}(A),0)$. 
The following lemma provides a convenient criterion for partial formality. 

\begin{lemma}[\cite{SW-forum}]
\label{lem:formalityeq} 
Let $A$ be a $\k$-$\cdga$, and suppose that $\dim_{\k} H^{q+1}(\M_q(A))< \infty$.  
Then $A$ is $q$-formal if and only if $\M_q(A)$ is $q$-formal.
 \end{lemma}

Minimal models are also relevant when considering the formality of 
morphisms of $\cdga$s. Indeed, let $\varphi\colon A\to B$ be a 
$\cdga$ map; then $\varphi$ is formal (in the sense of 
Definition \ref{def:formal-map}) if and only if there is a 
diagram of the form 
\begin{equation}
\label{eq:phif}
\begin{tikzcd}[column sep=24pt, row sep=22pt]
A  \ar{d}{\varphi}
&
\M(A) \ar[swap, pos=.4]{l}{\rho_A} 
\ar{d}{\M(\varphi)}
\ar{r}{\psi_A} 
& (H^{*}(A), 0) \ar{d}{\varphi^*}
\\
B 
&\M(B)\ar[swap, pos=.4]{l}{\rho_B} \ar{r}{\psi_B}
& (H^{*}(B), 0) 
\end{tikzcd}
\end{equation}
which commutes up to homotopy. 

Analogous statements hold for $q$-formal maps, with the middle 
arrow replaced by the morphism $\M_q(\varphi)\colon \M_q(A)\to \M_q(B)$.

\subsection{The dual of a $1$-minimal $\cdga$}
\label{subsec:dual-min}
Let $A=(A^{*},d)$ be a minimal $\cdga$ over $\k$, generated in degree $1$. 
Following \cite{Mo}, \cite{Kohno}, \cite{FHT2}, let us consider the filtration
\begin{equation}
\label{eq:filtration-minimal}
\k=A(0)\subset A(1)\subset A(2)\subset \cdots \subset A=\bigcup_{i\ge 0} A(i),
\end{equation}
where $A(1)$ is the subalgebra of $A$ generated by the cocycles in $A^1$,
and $A(i)$ for $i>1$ is the subalgebra of $A$ generated by all elements 
$x\in A^1$ such that $dx\in A(i-1)$.  Each inclusion $A(i-1)\subset A(i)$ is a 
Hirsch extension of the form $A(i)=A(i-1)\otimes \bwedge V_{i}$, where 
\begin{equation}
\label{eq:vi}
V_i\coloneqq \ker\big(H^2(A(i-1))\to H^2(A)\big).
\end{equation}
Taking degree $1$ pieces in the filtration \eqref{eq:filtration-minimal}, 
we obtain the filtration $\k=A(0)^1\subset A(1)^1\subset  \cdots \subset A^1$.  
Clearly, $A^1$ is a $1$-minimal $\cdga$.

Let us assume now that each of the aforementioned Hirsch extensions 
is finite, that is, $\dim V_i<\infty$ for all $i$.  Using the fact that 
$d(V_i)\subset A(i-1)$, we infer that each 
dual vector space $\fL_i=(A(i)^1)^{\vee}$ 
acquires the structure of a $\k$-Lie algebra by setting
\vspace*{-3pt}
\begin{equation}
\label{eq:duality}
\langle [u^{\vee},v^{\vee}], w\rangle \coloneqq 
\langle  u^{\vee}\wedge v^{\vee}, dw \rangle 
\end{equation}
for $u,v,w\in A(i)^1$.  Clearly, $d(V_1)=0$, and thus $\fL_1=(V_1)^{\vee}$ is 
an abelian Lie algebra. Using the vector space decompositions 
\begin{equation}
\label{eq:decomp}
A(i)^1=A(i-1)^1 \oplus V_i \: \text{ and } \:
A(i)^2=A(i-1)^2\oplus (A(i-1)^1\otimes V_i)\oplus \bwedge^2 V_i,
\end{equation}
one easily sees that  the canonical projection $\fL_i\surj \fL_{i-1}$, defined as 
the dual of the inclusion map $A(i-1)\inj A(i)$, has kernel $V_i^{\vee}$, 
and this kernel is central inside $\fL_i$. 
Therefore, we obtain a tower of finite-dimensional,  
nilpotent $\k$-Lie algebras, 
\begin{equation}
\label{eq:nilp lie tower}
\begin{tikzcd}[column sep=20pt]
0 & \fL_1 \ar[l]& \fL_2\ar[l, two heads] & \cdots\ar[l, two heads] & \fL_i 
\ar[l, two heads]& \cdots \ar[l, two heads] .
\end{tikzcd}
\end{equation}

Let $\fL=\fL(A)$ be the inverse limit of this tower, equipped with the 
inverse limit filtration. Then $\fL$ is a complete, filtered Lie algebra 
with the property that $\fL/\widehat{\gamma}_{i+1}( \fL)=\fL_i$ 
for each $i\ge 1$.  Conversely, from a tower as in \eqref{eq:nilp lie tower}, 
one can construct a sequence of finite Hirsch extensions as 
in \eqref{eq:filtration-minimal}.  Let $A(i)=A(i-1)\otimes \bwedge V_{i}$ 
be one of the $\cdga$s in this sequence, with differential given by \eqref{eq:duality}.  
Then $A(i)$ coincides with the Chevalley--Eilenberg complex
$\mathcal{C}(\fL_i)=(\bwedge \fL_i^{\vee},d)$ associated to the finite-dimensional 
Lie algebra $\fL_i=\fL(A(i))$; that is, the $\cdga$ whose underlying 
graded algebra is the exterior algebra on the dual vector space 
$\fL_i^{\vee}$, and whose differential is the extension by the graded 
Leibniz rule of the dual of the signed Lie bracket on the algebra generators.  
By the definition of Lie algebra cohomology, then, 
\begin{equation}
\label{eq:ce}
H^{*}(A(i))\cong H^{*}(\fL_i, \k) .
\end{equation}

The direct limit of the above sequence of Hirsch 
extensions, $A= \bigcup_{i\ge 0}A(i)$, 
is a minimal $\k$-$\cdga$ generated in degree $1$.  
We obtain in this fashion an adjoint correspondence that sends $A$ to the  
pronilpotent Lie algebra $\fL=\fL(A)$  and conversely, sends a pronilpotent Lie 
algebra $\fL$ to the minimal algebra $A=A(\fL)$.     
Under this correspondence, filtration-preserving $\cdga$ morphisms   
$A \to B$ get sent to filtration-preserving Lie morphisms $\fL(B) \to \fL(A)$, 
and the other way around.

\subsection{Positive weights}
\label{subsec:posWeights}

Following Body, Mimura, Shiga, and Sullivan \cite{BMSS}, 
Morgan \cite{Mo}, and Sullivan \cite{Sullivan77}, we say that a commutative 
graded algebra $A^*$ has {\em positive weights}\/ if each graded 
piece admits a vector space decomposition 
\begin{equation}
\label{eq:weights-alg}
A^i=\bigoplus_{\alpha\in\Z}A^{i,\alpha}
\end{equation}
with $A^{i,\alpha}=0$ for $\alpha\leq 0$, such that $xy\in A^{i+j,\alpha+\beta}$ for 
$x\in A^{i,\alpha}$ and $y\in A^{j,\beta}$.  Furthermore, we say that a $\cdga$ 
$(A,d)$ has \emph{positive weights}\/ if the underlying $\cga$ $A^{*}$ has 
positive weights, and the differential is homogeneous with respect to those 
weights, that is, $d x\in A^{i+1,\alpha}$ for $x\in A^{i,\alpha}$.

Now let $(A,d)$ be a minimal $\cdga$ generated in degree one, 
endowed with the canonical filtration $\{A_i\}_{i\ge 0}$ constructed in 
\eqref{eq:filtration-minimal}, where each sub-$\cdga$ $A_i$ is given by 
a Hirsch extension of the form $A_{i-1}\otimes \bwedge V_i$. 
The underlying $\cga$ $A$ possesses a natural set of positive 
weights, which we will refer to as the {\em Hirsch weights}: simply declare 
$V_i$ to have weight $i$, and extend those weights to $A$ multiplicatively.  
We say that the $\cdga$ $(A,d)$ has {\em positive Hirsch weights}\/ if 
the differential $d$ is homogeneous with respect to those weights.  
If this is the case, each sub-$\cdga$ $A_i$ also has positive Hirsch weights. 

\begin{lemma}[\cite{SW-forum}]
\label{lem:positivegraded}
Let $A$ be a minimal $\cdga$ generated in degree one, 
with dual Lie algebra $\fL=\fL(A)$. Then $A$ has positive Hirsch 
weights if and only if $\fL=\widehat{\gr}(\fL)$.
\end{lemma}

The next example (extracted from \cite{SW-forum}) shows that the 
hypothesis of Lemma \ref{lem:positivegraded} is more restrictive 
than the Lie algebra $\fL=\fL(A)$ being filtered-formal.

\begin{example}
\label{ex:3step}
Let $\g$ be the $5$-dimensional Lie algebra with basis $e_1,\dots, e_5$ 
and with Lie brackets given by 
$[e_1,e_2]=e_3-e_4/2-e_5$, $[e_1,e_3]=e_4$, $[e_2,e_3]=e_5$, and 
$[e_i,e_j]=0$, otherwise. It is readily verified that $\g$ is filtered-formal, 
although the differential of the $1$-minimal $\cdga$ $A=\bwedge \g^{\vee}$
is not homogeneous on the Hirsch weights. 
\end{example}

If a minimal $\cdga$ is generated in degree $1$ and has positive 
weights, but these weights do not coincide with the Hirsch weights, 
then the dual Lie algebra need not be filtered-formal.  
This phenomenon is illustrated in the next example, 
adapted from \cite{Cornulier14}, \cite{SW-forum}.

\begin{example}
\label{ex:Cornulier}
Let $\g$ be the nilpotent, $5$-dimensional Lie algebra with non-zero 
Lie brackets given by $[e_1,e_3]=e_4$ and $[e_1,e_4]=[e_2,e_3]=e_5$. 
The center of $\g$ is $1$-dimensional, generated by $e_5$, while the 
center of $\gr(\g)$ is $2$-dimensional, generated by $e_2$ and $e_5$.
Therefore, $\g\not\cong \gr(\g)$, and so $\g$ is not filtered-formal. 
The $1$-minimal $\cdga$ $A=\bwedge \g^{\vee}$ does have positive 
weights, given by the index of the chosen basis, but $A$ 
does not admit positive Hirsch weights.  
\end{example}

\section{Lie algebras and filtered formality}
\label{sect:Lie-algebras}
 
\subsection{Graded Lie algebras}
\label{subsec:grlie}
Once again, let us fix a ground field $\k$ of characteristic $0$.  Let $\g$ be 
a Lie algebra over $\k$; that is, a $\k$-vector space $\g$ endowed with an  
alternating bilinear operation, $[\,,\,]\colon \g\times \g\to \g$, that satisfies 
the Jacobi identity.  We say that $\g$ is a {\em graded Lie algebra}\/ if $\g$ 
decomposes as $\g=\bigoplus_{i\ge 1} \g_i$, the Lie bracket is compatible 
with the grading, and the Lie identities are satisfied with the appropriate signs.  
A morphism of graded Lie algebras is a 
$\k$-linear map $\varphi\colon \g\to \h$ which preserves the Lie brackets 
and the degrees; in particular, $\varphi$ induces $\k$-linear maps 
$\varphi_i\colon \g_i\to \h_i$ for all $i\ge 1$.

The most basic example of a graded Lie algebra is constructed as follows.  
Let $V$ a $\k$-vector space. The tensor algebra $T(V)$ has a natural 
Hopf algebra structure, with comultiplication $\Delta$ and counit $\varepsilon$ 
the algebra maps given by $\Delta(v)= v\otimes 1+1\otimes v$ and $\varepsilon(v)= 0$, 
for $v\in V$. The {\em free Lie algebra}\/ on $V$ is the set  of primitive 
elements in the tensor algebra; that is, 
$\Lie(V)=\{x \in T(V) \mid \Delta(x)=x \otimes 1+
1\otimes x\}$, with Lie bracket $[x,y]=x\otimes y-y\otimes x$ 
and grading induced from $T(V)$. 

A Lie algebra $\g$ is said to be \textit{finitely generated}\/ if there 
is an epimorphism $\varphi\colon \Lie(V)\to \g$ for some finite-dimensional 
$\k$-vector space $V$. 
If, moreover, the Lie ideal $\mathfrak{r}=\ker (\varphi)$ is finitely 
generated as a Lie algebra, then $\g$ is called \textit{finitely presented}. 
Now suppose all elements of $V$ are assigned degree $1$ in $T(V)$.  Then 
the inclusion $\iota\colon \Lie(V)\to T(V)$ identifies $\Lie_1(V)$ with $T_1(V)=V$.  
Furthermore, $\iota$ maps $\Lie_2(V)$ to $T_2(V)=V\otimes V$ 
by sending $[v,w]$ to $v\otimes w-w\otimes v$ for each $v,w\in V$;   
we thus may identify $\Lie_2(V)\cong V\wedge V$ by sending 
$[v,w]$ to $v\wedge w$.   

If $\g=\Lie(V)/\fr$, with $V$ a (finite-dimensional) vector space 
concentrated in degree $1$,  
then we say $\g$ is {\em (finitely) generated in degree $1$}.  
If, moreover, the Lie ideal $\mathfrak{r}$ is homogeneous, 
then $\g$ is a graded Lie algebra.  In particular, if $\g$ is 
finitely generated in degree $1$ and the homogeneous ideal 
$\mathfrak{r}$ is generated in degree $2$, then we say 
 $\g$ is a \textit{quadratic Lie algebra}. 

\subsection{Filtrations}
\label{subsec:series}
A\, {\em filtration}\/ $\F$ on a Lie algebra $\g$ is a nested 
sequence of Lie ideals, $\g=\F_1(\g)\supset \F_2(g)\supset \cdots$.  
A well-known such filtration is the {\em derived series}, with terms 
$\F_i(\g)=\g^{(i-1)}$ inductively defined by $\g^{(0)}=\g$ and 
$\g^{(i)}=[\g^{(i-1)}, \g^{(i-1)}]$ for $i\ge 1$. Clearly, 
the derived series is preserved by Lie algebra maps, and 
the quotient Lie algebras $\g/\g^{(i)}$ are solvable. Moreover, 
if $\g$ is a graded Lie algebra, all these solvable quotients 
inherit a graded Lie algebra structure. 

The existence of a filtration $\F$ on a Lie algebra $\g$  
makes $\g$ into a topological vector space, by defining 
a basis of open neighborhoods of an element $x\in\g$ 
to be $\{x+\F_k (\g)\}_{k\in \N}$. The fact that each basis 
neighborhood $\F_k(\g)$ is a Lie subalgebra implies that the Lie 
bracket map $[\,,\,]\colon \g \times \g \to \g$ is continuous; thus, 
$\g$ is, in fact, a topological Lie algebra. 
We say that $\g$ is {\em complete}\/ (respectively, {\em separated}) 
if the underlying topological vector space enjoys those properties. 

Every ideal $\mathfrak{a}$ of $\g$ inherits a filtration,  
given by $\F_k (\mathfrak{a})\coloneqq \F_k (\g)\cap \mathfrak{a}$.  
Likewise, the quotient Lie algebra, $\g/\mathfrak{a}$, has a naturally induced 
filtration with terms $\F_k (\g)/\F_k (\mathfrak{a})$. 
Letting $\overline{\mathfrak{a}}$ be the closure of $\mathfrak{a}$ 
in the filtration topology, we have that $\overline{\mathfrak{a}}$ is a closed ideal 
of $\g$. Moreover, by the continuity of the Lie bracket, 
$\overline{[\bar{\fa},\bar{\fr}]}=\overline{[\fa,\fr]}$. 
Finally, if $\g$ is complete (or separated), then $\g/\overline{\mathfrak{a}}$ 
is also complete (or separated). 

For each $j\ge k$, there is a canonical projection, $\g/\F_j(\g)\to \g/\F_k(\g)$, 
compatible with the projections from $\g$ to its quotient Lie algebras $\g/\F_k(\g)$.  
The {\em completion}\/ of the Lie algebra $\g$ with respect to the filtration 
$\F$ is defined as the limit of this inverse system, 
$\widehat{\g}= \varprojlim\nolimits_{k}\g/\F_k(\g)$. 
Using the fact that $\F_k(\g)$ is an ideal of $\g$,
it is readily seen that $\widehat{\g}$ is a Lie algebra, with Lie bracket 
defined componentwise.  Furthermore, $\widehat{\g}$ has 
a natural inverse limit filtration, $\widehat{\F}$, whose terms 
$\widehat{\F}_k (\widehat{\g})$ are equal to $\widehat{\F_k(\g)}= 
\varprojlim\nolimits_{i\ge k}\F_k (\g)/\F_i (\g)$.
Observe  that $\widehat{\F}_k (\widehat{\g}) = \overline{\F_k (\g)} $, 
and so each term of the filtration $\widehat{\F}$ is a closed 
Lie ideal of $\widehat{\g}$.   Furthermore, the Lie algebra 
$\widehat{\g}$, endowed with this filtration, is both complete and separated. 

Let $\iota\colon \g\to \widehat{\g}$ be the canonical map to the completion.  
Then $\iota$ is a morphism of Lie algebras, preserving the respective 
filtrations.  Clearly, $\ker(\iota)=\bigcap_{k\ge 1} \F_k (\g)$;   
hence, $\iota$ is injective if and only if $\g$ is separated.  
Moreover, $\iota$ is surjective if and only if $\g$ is complete. 

\subsection{Filtered Lie algebras}
\label{subsec:filt lie}
A  {\em filtered Lie algebra}\/ (over the field $\k$) 
is a Lie algebra $\g$ endowed with a decreasing filtration 
$\F=\{\F_k (\g) \}_{k\ge 1}$ by $\k$-vector  subspaces, 
satisfying the condition 
\begin{equation}
\label{eq:mult filt}
[\F_k (\g),\F_{\ell} (\g) ]\subseteq \F_{k+\ell} (\g)
\end{equation}
for all $k, \ell\ge 1$.  This condition implies that each subspace $\F_k(\g)$ 
is a Lie ideal, and so, in particular, $\F$ is a Lie algebra filtration. Let 
\begin{equation}
\label{eq:grfg}
\gr^{\F}(\g)\coloneqq \boplus_{k\ge 1}\F_k(\g)/ \F_{k+1}(\g) 
\end{equation}
be the corresponding associated graded vector space. 
Condition \eqref{eq:mult filt} implies that the Lie bracket map on 
$\g$ descends to a map 
$[\,,\,]\colon \gr^{\F}(\g) \times \gr^{\F}(\g) \to \gr^{\F}(\g)$ 
which makes $\gr^{\F}(\g)$ into a graded Lie algebra, with 
graded pieces given by the decomposition \eqref{eq:grfg}. 

A morphism of filtered Lie algebras is a linear map 
$\phi\colon \g\to \h$ preserving Lie brackets and the 
given filtrations, $\F$ and $\mathcal{G}$. Such a map  
induces morphisms between  nilpotent quotients, 
$\phi_k\colon \g/\F_{k+1}(\g) \to \h/\cG_{k+1} (\h)$, and 
a morphism of associated graded Lie algebras, 
$\gr(\phi)\colon \gr^{\F}(\g)\rightarrow \gr^{\cG}(\h)$. 

If $\g$ is a filtered Lie algebra with a multiplicative filtration 
$\F$ as in \eqref{eq:mult filt}, then its completion, $\widehat{\g}$, 
is again a filtered Lie algebra with the completed multiplicative filtration $\widehat{\F}$. 
By construction, the canonical map to the completion, $\iota\colon \g\to \widehat{\g}$, 
is a morphism of filtered Lie algebras.  It is readily seen that the induced 
morphism, $\gr(\iota)\colon \gr^{\F}(\g) \to \gr^{\widehat{\F}}(\widehat{\g})$, 
is an isomorphism.  Moreover, if $\g$ is both complete 
and separated, then the map $\iota\colon \g\to \widehat{\g}$ itself is 
an isomorphism of filtered Lie algebras. 
More generally, if $\phi\colon \g\to \h$ is a morphism of complete, separated, 
filtered Lie algebras, and $\gr(\phi)$ is an isomorphism, then, as noted in 
\cite[Lemma 2.1]{SW-forum}, $\phi$ is also an isomorphism.

\subsection{The degree completion}
\label{subsec:deg completion}
Every Lie algebra $\g$ comes equipped with a lower central series (LCS) 
filtration, $\{\gamma_{n}(\g)\}_{n\ge 1}$. This filtration is defined inductively 
by $\gamma_1(\g)=\g$ and $\gamma_{n}(\g)=[\gamma_{n-1}(\g),\g]$ for $n\ge 2$.  
This is a multiplicative filtration, and 
if $\{\F_{n}(\g)\}_{n\ge 1}$ is another such filtration, then
$\gamma_{n} (\g)\subseteq \F_{n} (\g)$, for all $n\ge 1$.
Lie algebra morphisms preserve LCS filtrations,  
and the quotient Lie algebras $\g/\gamma_{n} (\g)$ are nilpotent. 
We shall write $\gr(\g)$ for the associated graded Lie algebra 
and $\widehat{\g}$ for the completion of $\g$ with respect to the 
LCS filtration.  Furthermore, we shall take 
$\widehat{\gamma}_{n}=\overline{\gamma}_{n}$ as 
the terms of the canonical filtration on $\widehat{\g}$.

Every graded Lie algebra, $\g=\bigoplus_{i\ge 1} \g_i$, has 
a canonical decreasing filtration induced by the grading, 
$\F_{n} (\g) \coloneqq \bigoplus_{i\ge n} \g_i$. 
Moreover, if $\g$ is generated in degree $1$, then 
this filtration coincides with the LCS filtration. 
In particular, the associated graded Lie algebra with respect to $\F$  
coincides with $\g$. In this case, the completion of $\g$ with respect 
to the lower central series (or, degree) filtration is called the 
{\em degree completion}\/ of $\g$, and is simply denoted by 
$\widehat{\g}$.  It is readily seen that $\widehat{\g}\cong \prod_{i\ge 1} \g_i$.  
Therefore, the morphism $\iota\colon \g\to \widehat{\g}$ is injective, 
and induces an isomorphism between $\g$ and 
$\gr( \widehat{\g} )$.  

\begin{lemma}[\cite{SW-forum}]
\label{lem:presbar}
Suppose $\L$ is a free Lie algebra generated in degree $1$ and 
$\fr$ is a homogeneous ideal. Then the projection $\L\surj \L/\fr$ 
induces an isomorphism 
$\widehat{\L}/\overline{\fr}\isom\widehat{\L/\fr}$.
\end{lemma}

\subsection{Filtered-formality}
\label{subsec:filtered formal}
We now consider in more detail the relationship between a filtered 
Lie algebra $\g$ and the completion of its associated graded Lie 
algebra, $\widehat{\gr}(\g)$, equipped with the inverse limit filtration.  
Note that both Lie algebras share the same associated graded Lie 
algebra, namely, $\gr(\g)$. In general, though, $\g$ may fail to be  
isomorphic to $\widehat{\gr}(\g)$. Of course, this happens  
if $\g$ is not complete or separated, but it may happen 
even in the case when $\g$ is a (finite-dimensional) nilpotent 
Lie algebra.

\begin{definition}[\cite{SW-forum}]
\label{def:filt formal}
A complete, separated, filtered Lie algebra $\g$ is {\em filtered-formal}\/ 
if there is a filtered Lie algebra isomorphism, $\g\cong \widehat{\gr}(\g)$, 
which induces the identity on associated graded Lie algebras.
\end{definition}

If $\g$ is a filtered-formal Lie algebra, there exists a graded 
Lie algebra $\h$ such that $\g$ is isomorphic to 
$\widehat{\h}=\prod_{i\geq 1} \h_i$.
Conversely, if $\g=\widehat{\h}$ is the completion of a 
graded Lie algebra $\h=\bigoplus_{i\ge 1} \h_i$, then $\g$ is filtered-formal. 
Moreover, if $\h$ has homogeneous presentation $\h=\Lie(V)/\mathfrak{r}$, 
with $V$ finitely generated and concentrated in degree $1$, then, 
by Lemma \ref{lem:presbar}, the complete, filtered Lie algebra 
$\g=\prod_{i\ge 1} \h_i$ has presentation 
$\g=\widehat{\Lie}(V)/\overline{\mathfrak{r}}$. 
Some sufficient conditions for filtered formality 
are given in the following proposition.

\begin{proposition}[\cite{SW-forum}]
\label{lem:filtiso}
Let $\g$ be a complete, separated, filtered Lie algebra. Suppose 
one of the following two conditions is satisfied.
\begin{enumerate}
\item \label{filt1}
There is a graded Lie algebra $\h$ and an isomorphism 
$\g\cong \widehat{\h}$ preserving filtrations.
\item \label{filt2}
The  graded Lie algebra $\gr(\g)$ is generated in degree $1$ 
and there is a morphism of filtered Lie algebras,  
$\phi\colon \g \to \widehat{\gr}(\g)$, such that $\gr_1(\phi)$ is an 
isomorphism. 
\end{enumerate}
Then $\g$ is filtered-formal.
\end{proposition} 

As shown in \cite{SW-forum}, filtered-formality enjoys a descent 
property, provided some mild finiteness hypotheses are satisfied.  
As usual, all the ground fields will be of characteristic $0$. First, 
let us record a straightforward lemma, which follows from the 
fact that completion commutes with tensor products.

\begin{lemma}
\label{lem:ffascent}
Let $\g$ be a filtered-formal Lie algebra over a field $\k$.  
If\/ $\k\subset \K$ is a  field extension, then 
the $\K$-Lie algebra $\g \otimes_{\k} \K$ is also filtered-formal. 
\end{lemma}

The next theorem generalizes a result of Cornulier \cite{Cornulier14}; 
its proof is based in part on work of Enriquez \cite{Enriquez}  and 
Maassarani \cite{Maassarani}.

\begin{theorem}[\cite{SW-forum}]
\label{thm:ffdescent}
Let $\g$ be a complete, separated, filtered $\k$-Lie algebra 
such that $\gr(\g)$ is finitely generated in degree $1$, and let 
$\k\subset \K$ be a field extension. Then $\g$ 
is filtered-formal if and only if the $\K$-Lie algebra 
$\g \otimes_{\k} \K$ is  filtered-formal. 
\end{theorem}

\section{Lower central series and Malcev Lie algebras}
\label{sect:malcev-lie}

\subsection{Lower central series}
\label{subsec:lcs}

Let $G$ be a group. Given subgroups $H_1, H_2\le G$, 
their commutator, $[H_1,H_2]$, is the subgroup of $G$ generated by all 
elements of the form $[x_1,x_2]\coloneqq x_1^{}x_2^{}x_1^{-1}x_2^{-1}$ 
with $x_i\in H_i$. The {\em lower central series}\/ (LCS) of the group, 
$\{\gamma_n (G)\}_{n\ge 1}$, is defined inductively by $\gamma_1 (G)=G$ and 
$\gamma_{n+1} (G) = [\gamma_n(G), G]$. This is an $N$-series, in the 
sense of Lazard \cite{Lazard}, that is, $[\gamma_n(G), \gamma_m(G)]\subseteq 
[\gamma_{m+n}(G)]$ for all $m,n\ge 1$. It follows that each subgroup 
$\gamma_n(G)$ is normal in $G$; moreover, each LCS quotient 
$\gamma_n(G)/\gamma_{n+1}(G)$ lies in the center of 
$G/\gamma_{n+1}(G)$, and thus is an abelian group. 
For instance, $\gamma_2(G)=[G,G]$ is the derived (or, commutator) 
subgroup and $G/\gamma_2(G) = G_{\ab}$ is the abelianization of $G$.

If $\gamma_{n} (G)\ne 1$ but $\gamma_{n+1}(G)=1$, then 
$G$ is said to be an {\em $n$-step nilpotent group}; in general, 
though, the LCS filtration does not terminate. For each $n\ge 2$, 
the factor group $G/\gamma_{n}(G)$ is the maximal $(n-1)$-step 
nilpotent quotient of $G$. 

The direct sum of the LCS quotients, $\gr (G)  = \bigoplus_{n\ge 1} \gr_{n} (G)$, 
acquires the structure of a graded Lie algebra over $\Z$, called the 
{\em associated graded Lie algebra}\/ of $G$. The addition in $\gr (G)$ 
is induced from the group multiplication and the Lie bracket is induced 
from the group commutator. For instance, if $G=F_n$ is a finitely 
generated free group of rank $n\ge 1$, then $\gr(F_n)=\Lie(\Z^n)$, 
the free Lie algebra on $n$ generators. 

If $\k$ is a field of characteristic $0$, then 
$\gr (G;\k)  \coloneqq \bigoplus_{n\ge 1} \gr_{n} (G)\otimes_{\Z} \k$ 
is a graded Lie algebra over $\k$. We note that both the 
assignments $G\leadsto \gr(G)$ and $G\leadsto \gr(G;\k)$ 
are functorial. 

\subsection{Malcev completion}
\label{subsec:mg}
A group $G$ is said to be {\em rational}\/ (or, uniquely divisible) 
if the power map $G\to G$, $g\mapsto g^n$ is a bijection, for 
every $n\ge 1$. The rational abelian groups are precisely the 
$\Q$-vector spaces. A natural way to rationalize an 
abelian group $A$ is to map it to $A\otimes_{\Z} \Q$ 
via $a\mapsto a\otimes 1$, with this map being universal 
for homomorphisms of $G$ into uniquely divisible abelian groups.

In work of Malcev \cite{Malcev}, Lazard \cite{Lazard}, and Hilton \cite{Hilton} 
(see also \cite{BK72}, \cite{HMR}, \cite{Iv}), this construction was extended 
to arbitrary nilpotent groups. The {\em Malcev completion}\/ functor is left 
adjoint to the embedding of the category 
of rational nilpotent groups into the category of nilpotent groups.
Thus, if $N$ is a nilpotent group, its Malcev completion (or, rationalization) 
 is a rational nilpotent group, denoted $N\otimes \Q$, that comes endowed 
with a map $\kappa\colon N\to N\otimes \Q$ which is universal 
for homomorphisms of $G$ into uniquely divisible nilpotent groups.
Moreover, the kernel of $\kappa$ is equal to $\Tors(N)$, the torsion 
subgroup of $N$, and the induced map, 
$\kappa^*\colon \Hom(N\otimes \Q,K)\to \Hom(N,K)$, 
is an isomorphism for all rational nilpotent groups $K$. 
Malcev completion is an exact functor, 
which induces isomorphisms $H_*(N,\Q)\cong H_*(N\otimes \Q,\Z)$. 
The quotient $N/\Tors(N)$ is a torsion-free nilpotent group that has 
the same rationalization as $N$. If $N$ is finitely generated, 
then $N\otimes \Q$ is a nilpotent Lie group defined over $\Q$, 
with integral form $N/\Tors(N)$ and whose Lie algebra, 
$\mathfrak{Lie}(N\otimes \Q)$, is nilpotent.

We now turn to an arbitrary group $G$. The succesive nilpotent quotients 
of $G$ assemble into a tower of the form
\vspace*{-2pt}
\begin{equation}
\label{eq:nilp-tower}
\begin{tikzcd}[column sep=20]
\cdots   \ar[r]& G/\gamma _4 (G)\ar[r]& G/\gamma_3 (G)  
\ar[r]& G/\gamma_2(G) \ar[r]&  1.
\end{tikzcd}
\end{equation}
Replacing in this tower each nilpotent quotient 
by its rationalization and taking 
the inverse limit of this directed system, we obtain 
a prounipotent, filtered Lie group over $\Q$,
\begin{equation}
\label{eq:malcev-group}
G_{\Q}\coloneqq \varprojlim\nolimits_{n} (G/\gamma_{n}( G)\otimes \Q),
\end{equation}
which is called the {\em Malcev completion}\/ (or, the {\em prounipotent 
completion}) of the group $G$.  
We denote by $\kappa\colon G\to G_{\Q}$ the canonical homomorphism 
from $G$ to its rational completion and note that the assignment $G\leadsto 
G_{\Q}$ is functorial. 

The pronilpotent Lie algebra 
\vspace*{-2pt}
\begin{equation}
\label{eq:malcevLie}
\m(G):=\varprojlim\nolimits_{n} \mathfrak{Lie}(G/\gamma_{n} (G)\otimes \Q),
\end{equation}
is called the {\em Malcev Lie algebra}\/ of $G$.  This Lie algebra comes 
endowed with the inverse limit filtration, which makes it a complete, 
separated, filtered Lie algebra over $\Q$. As before, the assignment 
$G\leadsto \m(G)$ is functorial. Moreover, if $G$ is finitely generated, 
then $\m(G)$ is a finitely generated Lie algebra.

\subsection{Quillen's construction}
\label{subsec:malcev}

A different approach was taken by Quillen in \cite[Appendix~A]{Qu69}; 
we recall now his construction of the Malcev completion and the 
Malcev Lie algebra, building on the treatment from 
\cite{PS-imrn04}, \cite{PS-formal}, \cite{Massuyeau12}, 
\cite{FHT}, and \cite{SW-forum}. 

A {\em Malcev Lie algebra}\/ is a Lie algebra $\m$ over a field  
of characteristic $0$, endowed with a decreasing, complete 
vector space filtration $\F=\{\F_i\}_{i\ge 1}$ such that $\F_1=\m$ and   
$[\F_i,\F_j]\subset \F_{i+j}$, for all $i, j$, and with the property that the 
associated graded Lie algebra, $\gr(\m)=\bigoplus _{i\ge 1}
\F_i/\F_{i+1}$, is generated in degree~$1$. For example, the 
completion $\widehat{\g}$ of a Lie algebra $\g$ with respect to 
the lower central series filtration $\{\gamma_i (\g)\}_{i\ge 1}$,  
endowed with the canonical completion filtration, is a Malcev Lie algebra. 

Given a group $G$, the group algebra $\Q[G]$ has a natural Hopf 
algebra structure, with comultiplication map $\Delta\colon \Q[G]\to \Q[G]\otimes \Q[G]$  
given by $\Delta(g)=g\otimes g$, and counit the augmentation map 
$\varepsilon\colon \Q[G]\to \Q$ given by $\varepsilon(g)= 1$. 
An element $x\in \Q[G]$ is said to be {\em group-like}\/ if $\Delta(x)=x\otimes x$ 
and {\em primitive}\/ if $\Delta(x)=x\otimes 1+1\otimes x$; under 
the inclusion $G\inj \Q[G]$, the set of all group-like elements 
gets identified with $G$. Let $I=\ker(\varepsilon)$ be the augmentation 
ideal of, and let 
\begin{equation}
\label{eq:completion-groupalgebra}
\widehat{\Q[G]}=\varprojlim_r \Q[G]/I^r
\end{equation}
be the completion of $\Q[G]$ with respect to the filtration by 
the powers of this ideal.  
Define the completed tensor product $\widehat{\Q[G]}\, 
\hat{\otimes}\,  \widehat{\Q[G]}$ as the
completion of $\Q[G]\otimes \Q[G]$ with respect 
to the natural tensor product filtration. 
Applying the $I$-adic completion functor to the map 
$\Delta$ yields a comultiplication map,    
$\widehat{\Delta} \colon \widehat{\Q[G]}\to \widehat{\Q[G]}\, 
\hat{\otimes}\, \widehat{\Q[G]}$, which makes $\widehat{\Q[G]}$ 
into a complete Hopf algebra. As shown by Quillen,  
there is a natural, filtration-preserving isomorphism, 
\begin{equation}
\label{eq:mg-prim}
\m(G)\cong \Prim \big(\widehat{\Q[G]}\big),
\end{equation}
between the Malcev Lie algebra of $G$ and the Lie algebra of primitive 
elements in $\widehat{\Q[G]}$, with Lie bracket given by $[x,y]=xy-yx$.

The set of all primitive elements in $\gr(\Q[G])$ forms a graded 
Lie algebra, which is isomorphic to $\gr(G)\otimes \Q$.
An important connection between the Malcev Lie algebra 
and the associated graded Lie algebra of $G$ was discovered 
by Quillen, who showed in \cite{Qu68} that there is a natural isomorphism 
of graded Lie algebras, 
\begin{equation}
\label{eq:quillen-iso}
\gr(\m(G))\cong \gr(G;\Q).
\end{equation}

The Malcev completion $G_{\Q}$ may be identified 
with the set consisting of all group-like elements in the Hopf algebra 
$\widehat{\Q[G]}$.  This is a group which comes endowed with a complete, 
separated filtration, whose $n$-th term is $G_{\Q} \cap (1+\widehat{I^n})$.  
As explained in \cite{Massuyeau12}, there is a one-to-one, filtration-preserving 
correspondence between primitive elements 
and group-like elements of $\widehat{\Q[G]}$ via the exponential 
and logarithmic maps, 
\begin{equation}
\label{eq:explog}
\begin{tikzcd}[column sep=30]
 G_{\Q}\subset 1+\widehat{I} \ar[rr, bend right=16, "\log" description]&  
  & \widehat{I} \supset\m(G)\ar[ll, bend right=16, "\exp" description].
\end{tikzcd}
\end{equation} 
Passing to associated graded objects and using \eqref{eq:quillen-iso}, 
we find that $\gr( G_{\Q}) \cong \gr(G;\Q)$; in particular, $H_1(G_{\Q})=H_1(G,\Q)$.

\subsection{Multiplicative expansions and Taylor expansions}
\label{subsec:exp}
Let $G$ be a group. 
Given a map $f\colon G\to R$, where $R$ is a ring, we will 
denote by $\bar{f} \colon \Q[G]\to R$ its linear extension to 
the group algebra. 
A \emph{\textup{(}multiplicative\textup{)} expansion}\/
of $G$ is a map 
\begin{equation}
\label{eq:expansion}
\begin{tikzcd}[column sep=18pt]
E\colon G \ar[r]& \widehat{\gr}(\Q[G])
\end{tikzcd}
\end{equation}
such that the linear extension $\bar{E}\colon\Q[G]\to \widehat{\gr}(\Q[G])$  
is a filtration-preserving algebra morphism with the property that $\gr(\bar{E})=\id$. 
Alternatively, a map as in \eqref{eq:expansion} is an expansion 
if it is a (multiplicative) monoid map  and the following 
property holds:  If $f\in I^k\setminus I^{k+1}$, then $\bar{E}(f)$ 
starts with $[f]\in I^k/I^{k+1}$; that is, $\bar{E}(f)=(0,\dots,0,[f],*,*,\dots)$.  

\begin{definition}[\cite{Bar-Natan16}, \cite{SW-ejm}]
\label{def:taylor} 
An expansion $E\colon G \to \widehat{\gr}(\Q[G])$ is 
called a {\em Taylor expansion}\/ if it sends each element of 
$G$ to a group-like element of $\widehat{\gr}(\Q[G])$;  
that is, $\bar\Delta (E(g))=E(g) \hat{\otimes} E(g)$, 
for all $g\in G$.  
\end{definition}

It is shown in \cite{SW-ejm} that a Taylor expansion 
$E\colon G\to \widehat{\gr}(\Q[G])$ induces a 
filtration-pre\-serving isomorphism of complete Hopf algebras,  
$\widehat{E}\colon \widehat{\Q[G]}\to \widehat{\gr}(\Q[G])$, such that 
$\gr(\widehat{E})$ is the identity on $\gr(\Q[G])$. Conversely, 
a filtration-preserving isomorphism of complete Hopf algebras,  
$\phi\colon \widehat{\Q[G]}\to \widehat{\gr}(\Q[G])$, 
induces a Taylor expansion $E\colon G\to \widehat{\gr}(\Q[G])$.
These facts may be summarized as follows. 
 
\begin{theorem}[\cite{SW-ejm}]
\label{thm:TaylorHopf}
The assignment $E\leadsto \widehat{E}$ establishes a 
one-to-one correspondence between Taylor expansions 
$G\to \widehat{\gr}(\Q[G])$ and filtration-preserving isomorphisms 
of complete Hopf algebras $\widehat{\Q[G]}\to \widehat{\gr}(\Q[G])$ for which  
the associated graded morphism is the identity on $\gr(\Q[G])$. 
\end{theorem}

This theorem generalizes a result of Massuyeau, 
from finitely generated free groups to arbitrary finitely generated groups.  
As a corollary, we obtain the following criterion for the existence 
of a Taylor expansion.

\begin{corollary}[\cite{SW-ejm}]
\label{cor:te}
A finitely generated group $G$ has a Taylor expansion if and only if 
there is an isomorphism of filtered Hopf algebras,  
$\widehat{\Q[G]}\cong\widehat{\gr}(\Q[G])$.
\end{corollary}

Now suppose $G$ admits a finite presentation of the form $G=F/R$. 
Starting from a Taylor expansion for the finitely generated free group $F$, 
one may find a presentation for the Malcev Lie algebra 
$\m(G;\k)$, using the approach of Papadima \cite{Papadima95} and 
Massuyeau \cite{Massuyeau12}. This is summarized in the following theorem.

\begin{theorem}[\cite{Massuyeau12}, \cite{Papadima95}]
\label{thm:Massuyeau}
Let $G$ be a group with generators $x_1,\dots,x_n$ and relators $r_1,\dots, r_m$ 
and let $E$ be a Taylor expansion of the free group $F=\langle x_1,\dots,x_n\rangle$.
There exists then a unique filtered Lie algebra isomorphism
\[
\m(G) \cong\widehat{\Lie}(\Q^n)/
\langle\!\langle W\rangle\!\rangle,
\]
where $\langle\!\langle W \rangle\!\rangle$ denotes 
the closed ideal of the completed free Lie algebra $\widehat{\Lie}(\Q^n)$ generated 
by the subset $\{\log(E(r_1)),\dots,\log(E(r_m))\}$. 
\end{theorem}

\subsection{Filtered formality}
\label{subsec:tff}

Following \cite{SW-forum}, we say that a group $G$ is {\em filtered formal}\/ if 
its Malcev Lie algebra is filtered formal, that is, $\m(G)$ is isomorphic (as a 
filtered Lie algebra) to the degree completion of its associated graded Lie 
algebra, $\gr(\m(G))$. In view of \eqref{eq:quillen-iso}, this condition is 
equivalent to $\m(G)\cong \widehat{\gr}(G;\Q)$. It follows from 
Lemma \ref{lem:presbar} that $G$ is filtered formal if and only if 
$\m(G)$ admits a homogeneous presentation. 

For instance, if $G=F_n$, then $\m(F_n)\cong \widehat{\Lie}(\Q^n)$, and so 
$F_n$ is filtered formal. Moreover, if $G$ is a torsion-free, $2$-step 
nilpotent group for which $G_{\ab}$ is torsion-free 
(e.g., if $G=F_n/\gamma_3(F_n)$ with $n\ge 2$), 
then $G$ is filtered-formal. On the other hand, there are torsion-free, $3$-step 
nilpotent groups that are not filtered formal; see \cite{SW-forum}.

As the next theorem shows, the Taylor expansions of a finitely generated 
group $G$ are closely related to the isomorphisms between the Malcev Lie 
algebra and the LCS completion of the associated 
graded Lie algebra of $G$.

\begin{theorem}[\cite{SW-ejm}]
\label{thm:expansionFiltered}
There is a one-to-one correspondence between Taylor expansions 
$G\to \widehat{\gr}(\Q[G])$ and filtration-preserving Lie algebra isomorphisms 
$\m(G)\to \widehat{\gr}(G;\Q)$ inducing the identity on $\gr(G;\Q)$.  
\end{theorem}

Using this theorem, we obtain an alternate interpretation of filtered-formality.

\begin{corollary}[\cite{SW-ejm}]
\label{cor:TFF}
A finitely generated group $G$ is filtered-formal if and only if $G$ 
has a Taylor expansion.
\end{corollary}

\subsection{The RTFN property and Taylor expansions}
\label{subsec:rtfn-taylor}

A group $G$ is said to be {\em residually torsion-free nilpotent}\/ 
(for short, RTFN) if for any $g\in G$, $g\neq 1$,
there exists a torsion-free nilpotent group $Q$ 
and an epimorphism $\psi\colon G\to Q$ such that $\psi(g)\neq 1$. 
The property of being residually torsion-free nilpotent 
is inherited by subgroups and is preserved under 
direct products and free products.

The RTFN property may be expressed in terms of the {\em rational 
lower central series}\/ of $G$, whose terms are given by 
\begin{equation}
\label{eq:sqrt-filtration}
\gamma^{\rat}_{n}(G)=\{g\in G\mid \text{$g^m \in \gamma_{n}(G)$, 
for some $m\in \N$} \}.
\end{equation}
The group $G$ is RTFN if and only if the intersection of its rational 
lower central series, $\gamma^{\rat}_{\omega}(G)\coloneqq 
\bigcap_{n\ge 1} \gamma^{\rat}_{n}(G)$, is the trivial subgroup.
We refer to \cite{Su-lcs} for alternate definitions and more 
properties of this $N$-series.

As is well known, a group $G$ is residually 
torsion-free nilpotent if and only if the group-algebra 
$\Q[G]$ is residually nilpotent, that is, $\bigcap_{n\ge 1}I^n=\{0\}$, 
where $I$ is the augmentation ideal. 
Therefore, if $G$ is finitely generated, the RTFN condition is 
equivalent to the injectivity of the canonical map to the prounipotent 
completion, $\kappa\colon G\to G_{\Q}$, where recall $G_{\Q}$ is 
the set of group-like elements in $\widehat{\Q[G]}$.

If $G$ is residually nilpotent and $\gr_{n} (G)$ is torsion-free for 
all $n\ge 1$, then $G$ is residually torsion-free nilpotent. 
Residually torsion-free nilpotent implies residually nilpotent, which
in turn, implies residually finite. 
Examples of residually torsion-free nilpotent groups include 
torsion-free nilpotent groups, free groups, and surface groups. 

\begin{proposition}[\cite{SW-ejm}]
\label{prop:fftaylor}
A finitely generated group $G$ has an injective Taylor expansion 
if and only if $G$ is residually torsion-free nilpotent and filtered-formal. 
\end{proposition}

\section{Holonomy Lie algebras}
\label{sect:holo}

\subsection{The holonomy Lie algebra of a $\cdga$}
\label{subsec:holo-dga}

Let $A=(A^{*},d)$  be a $1$-finite $\k$-$\cdga$, that is, a 
$\cdga$ over a field $\k$ of characteristic $0$ with $A^0=\k$ 
and $\dim_{\k} A^1<\infty$. Writing $A_i=\Hom (A^i, \k)$ for 
the $\k$-duals of the graded pieces, we let 
$\mu^{\vee} \colon A_2 \to A_1\wedge A_1$ 
be the $\k$-dual of the multiplication map 
$\mu \colon A^1\wedge A^1\to A^2$, and we 
let $d^{\vee} \colon A_2\to A_1$ be the dual 
of the differential $d \colon A^1\to A^2$. We shall denote by 
$\Lie(A_1)$ the free Lie algebra on the $\k$-vector space $A_1$, 
and we will identify $\Lie_1(A_1)=A_1$ and  
$\Lie_2(A_1)=A_1\wedge A_1$.

\begin{definition}[\cite{MPPS}]
\label{def:holo cdga}
The {\em holonomy Lie algebra}\/ of a $1$-finite $\cdga$ $A=(A^{*},d)$ 
is the quotient of $\Lie(A_1)$ by the ideal generated by the image of the map 
$\partial_A=d^{\vee} + \mu^{\vee}$,
\begin{equation}
\label{eq:holo}
\h(A) = \Lie(A_1) / \langle\im (\partial_A)\rangle. 
\end{equation}
\end{definition}

Clearly, this construction is functorial.  Indeed, let $\varphi\colon A\to B$ is a 
morphism of $\cdga$s as above, and write 
$\varphi_i=(\varphi^i)^{\vee}\colon B_i\to A_i$. 
Then the induced map,  
$\Lie(\varphi_1)\colon \Lie(B_1)\to \Lie(A_1)$, factors through a 
morphism of Lie algebras, $\h(\varphi)\colon \h(B)\to \h(A)$.  
Observe that the Lie algebra $\h(A)$ depends only on the sub-$\cdga$ 
$\k \cdot 1 \oplus A^1 \oplus (d(A^1)+\mu(A^1\wedge A^1))$
of the truncation $A^{\le 2}$. Therefore, $\h(A)$ is finitely presented. 

In general, though, the ideal generated by $\im(\partial_A)$ 
is not homogeneous, and so the Lie algebra $\h(A)$ does not 
inherit a grading from $\Lie(A_1)$. 

\begin{example}
\label{ex:heis-lie}
Let $A=\bwedge(a_1,a_2,a_3)$ be the exterior algebra on generators $a_i$ 
in degree $1$, endowed with the differential $d$ given by $d{a_1}=d{a_2}=0$ 
and $d{a_3}=a_1 \wedge a_2$.  Identify $\Lie(A_1)$ with the free Lie algebra 
on dual generators $x_1,x_2,x_3$. Then the ideal $\langle\im(\partial_A)\rangle$ is 
generated by $x_3+[x_1,x_2]$, $[x_1,x_3]$, and $[x_2,x_3]$, and thus 
is not homogeneous. 
\end{example}

In the above example,  $\h(A)$ still admits the structure 
of a graded Lie algebra, with $x_1$ and $x_2$ in degree $1$, and $x_3$ in degree $2$.
Nevertheless, using a construction from \cite{SW-forum}, we may define a minimal, finite 
$\cdga$ $A$ for which $\h(A)$ does not admit any grading compatible with the 
lower central series filtration.

\begin{example}
\label{ex:noncarnot}
Let $A=\bwedge(a_1,\dots, a_5)$, with $\abs{a_i}=1$ and differential $d$ given by 
$d{a_4}=a_1 \wedge a_3$, $d{a_5}=a_1 \wedge a_4+a_2 \wedge a_3$, 
and $d{a_i}=0$, otherwise. Then, as shown in \cite[Example 10.5]{SW-forum}, 
$\h(A)$ is not isomorphic to $\gr(\h(A))$, its associated 
graded Lie algebra with respect to the LCS filtration.
\end{example}

\subsection{The holonomy Lie algebra of a $\cga$}
\label{subsec:holo-ga}
Now suppose $d=0$, so that $A$ is a graded, graded-commutative,  
$1$-finite $\k$-algebra. Then 
$\h(A) = \Lie(A_1) / \langle\im (\mu^{\vee})\rangle$ is the classical holonomy 
Lie algebra introduced by K.T.~Chen in \cite{Chen73} and further studied 
in \cite{Kohno}, \cite{MP}, \cite{PS-imrn04}, \cite{SW-jpaa}, and \cite{SW-forum}. 
Clearly, $\h(A)$ inherits a natural grading from the free Lie algebra 
$\Lie(A_1)$, which is compatible with the Lie bracket.  
Consequently, $\h(A)$ is a finitely-presented, graded 
Lie algebra, with generators in degree $1$ and relations in 
degree $2$; in other words, $\h(A)$ is a quadratic Lie algebra.  

A  graded, $1$-finite $\k$-algebra $A$ may be realized as the quotient 
$T(V)/I$, where $T(V)$ is the tensor algebra on a finite-dimensional 
$\k$-vector space $V$ by a homogeneous, two-sided ideal $I$. 
The algebra $A$ is said to be {\em quadratic}\/ if $A^1=V$ 
and the ideal $I$ is generated in degree $2$, i.e., $I=\langle I^2 \rangle$, 
where $I^2=I\cap (V\otimes V)$.  

Given a quadratic algebra $A=T(V)/I$, we identify 
$V^{\vee}\otimes V^{\vee} \cong (V\otimes V)^{\vee}$, and 
define the {\em quadratic dual}\/ of $A$ to be the algebra 
$A^{!}=T(V^{\vee})/I^{\perp}$,
where $I^{\perp}$ is the ideal of $T(V^{\vee})$ 
generated by the vector subspace $(I^2)^{\perp}\coloneqq 
\{\alpha\in V^{\vee}\otimes V^{\vee} \mid \alpha(I^2)=0\}$. 
Clearly, $A^{!}$ is also a quadratic algebra, and $(A^{!})^{!}=A$.    
For any graded algebra of the form $A=T(V)/I$, we may 
define its quadrature closure as $\qA=T(V)/\langle I^2\rangle$. 

For a Lie algebra $\g$, we let $U(\g)$ be its universal enveloping 
algebra. This is the filtered, associative algebra obtained as the 
quotient of the tensor algebra $T(\g)$ by the (two-sided) ideal 
generated by all elements of the form 
$a\otimes b-b\otimes a-[a,b]$ with $a, b\in \g$. 

\begin{proposition}[\cite{PY99}, \cite{SW-forum}]
\label{prop:Papadima-Y}
Let $A$ be a commutative graded $\k$-algebra such that $A^0=\k$ 
and $\dim_{\k} A^1<\infty$. Then $U(\h(A))$ is a quadratic algebra, 
and $U(\h(A))=(\qA)^!$.
\end{proposition}

Now suppose $\g$ is a finitely generated graded Lie algebra generated 
in degree $1$. Then, as shown in \cite{SW-forum}, there is a unique, 
functorially defined quadratic Lie algebra, $\qg$, such that $U(\qg) = \q U(\g)$. 
Therefore, by Proposition \ref{prop:Papadima-Y}, we have that 
$\h((\q U(\g))^{!})=\qg$.

Work of L\"{o}fwall \cite{Lofwall} yields another interpretation of the 
universal enveloping algebra of the holonomy Lie algebra.  

\begin{proposition}[\cite{Lofwall}]
\label{prop:yoneda}
Let $[\Ext^1_A(\k,\k)]:=\bigoplus_{i\ge 0} \Ext^i_A(\k,\k)_i$ be the linear strand 
in the Yoneda algebra of $A$.  Then $U(\h(A))\cong [\Ext^1_A(\k,\k)]$. 
\end{proposition}

Applying the Poincar\'{e}--Birkhoff--Witt theorem, we infer that the 
graded ranks of  $\h(A)$ are given by 
\begin{equation}
\label{eq:pbw}
\prod_{n\geq 1}(1-t^n)^{\dim_{\k} \h_n(A)} = \sum_{i\ge 0} b_{i,i}(A) t^i,
\end{equation}
where $b_{i,i}(A)=\dim_{\k} \Ext^i_A(\k,\k)_i$.

\subsection{The completion of the holonomy Lie algebra of a $\cga$}
\label{subsec:holohat}

Let $A$ be a connected $\k$-$\cga$. A $1$-minimal model $\M_1(A)$ for $A$ may 
be constructed in a ``formal'' way, following the approach outlined by Carlson 
and Toledo \cite{CT} (see also \cite{SW-forum}).  For the construction of the 
full, bigraded minimal model of a $\cga$ we refer to Halperin and Stasheff \cite{HS79}.
 
As in Section \ref{subsec:dual-min}, start with the $\cdga$s 
$\M(1)=(\bigwedge(V_1),0)$, 
where $V_1=A^1$, and $\M(2)=(\bigwedge(V_1\oplus V_2),d)$,
where $V_2=\ker \big(\mu\colon A^1\wedge A^1\to A^2\big)$ and 
$d\colon V_2\inj V_1\wedge V_1$ is the inclusion map.  Now 
define inductively a $\cdga$ $\M(i)$ as the Hirsch extension 
$\M(i-1)\otimes \bigwedge (V_{i})$, where 
the $\k$-vector space $V_{i}$ fits into the short exact sequence 
\begin{equation}
\label{eq:ses}
\begin{tikzcd}[column sep=18pt]
0 \ar[r]& V_{i} \ar[r]&  H^2(\M(i-1)) \ar[r]& \im(\mu) \ar[r]& 0,
\end{tikzcd}
\end{equation}
while the differential $d$ includes $V_{i}$ into $V_1\wedge V_{i-1}\subset \M(i-1)$. 
Setting $\M_1(A)$ equal to $\bigcup_{i\ge 1} \M(i)$, 
the subalgebras $\{\M(i)\}_{i\ge 1}$ constitute the canonical filtration 
\eqref{eq:filtration-minimal} of $\M_1(A)$ and the differential 
$d$ preserves the Hirsch weights on $\M_1(A)$.  For these reasons, 
we say that $\M_1(A)$ is the {\em canonical}\/ $1$-minimal model of $A$. 

The next theorem relates $\fL(\M_1(A))$, the Lie algebra associated 
to $\M_1(A)$ under the adjoint correspondence from Section \ref{subsec:dual-min}
to the degree completion of $\h(A)$, the holonomy Lie 
algebra of $A$.  A generalization will be given in Theorem 
\ref{thm:nat1model}.

\begin{theorem}[\cite{Mo}, \cite{MP}, \cite{SW-forum}]
\label{thm:model-holonomy}
If $A$ is a $1$-finite $\cga$, then $\fL(\M_1(A))$ and $\widehat{\h(A)}$ 
are isomorphic as complete, filtered Lie algebras. 
\end{theorem}

\begin{corollary}
\label{cor:holomin}
If $A$ is a $1$-finite $\cga$ and $\M_1(A)=\bwedge \big(\boplus _{i\ge 1} V_i \big)$ 
is the canonical $1$-minimal of $A$, then  $\dim_{\k} \h_i(A) = \dim V_i$ for all 
$i\ge 1$. 
\end{corollary}

\subsection{Holonomy and flat connections}
\label{subsec:holflat}

Given a $\k$-$\cdga$ $(A,d)$ and a Lie algebra $\g$, we 
let $\F(A, \g)$ be the set of $\g$-valued\, {\em flat connections}\/ on $A$,    
that is, the set of all elements $\omega \in A^1\otimes \g$ satisfying the
Maurer--Cartan equation,
\begin{equation}
\label{eq:mc}
d\omega + \tfrac{1}{2}[\omega, \omega]=0  .
\end{equation}

Suppose now that $A$ is $1$-finite. As shown in \cite{MPPS}, 
the natural isomorphism $A^1\otimes \g \isom \Hom (A_1, \g)$ 
induces a natural identification,
\begin{equation}
\label{eq:flathol}
\begin{tikzcd}
\F(A, \g) \ar[r, "\cong"] & \Hom_{\Lie} (\h (A), \g)  .
\end{tikzcd}
\end{equation}

Assuming further that $\g$ is finite-dimensional, we let 
$\cC(\g)=\big(\bigwedge \g^{\vee}, d)$ be the Cheval\-ley--Eilenberg 
complex of $\g$. This is the $\cdga$ whose underlying graded 
algebra is the exterior algebra on the dual $\k$-vector space 
$\g^{\vee}$, and whose differential  is the extension by the 
graded Leibnitz rule of the dual of the signed Lie bracket, 
$d=-\beta^*$, on the algebra generators, 
see e.g.~\cite{HS79}, \cite{FHT}.  There is then a natural isomorphism 
$A^1\otimes \g \isom \Hom (\g^{\vee}, A^1)$, which, 
by \cite[Lemma 3.4]{DP-ccm}, induces a natural identification,
\begin{equation}
\label{eq:flatcochains}
\begin{tikzcd}
\F(A, \g) \ar[r, "\cong"] &\Hom_{\cdga} (\cC (\g), A) .
\end{tikzcd}
\end{equation}

Now let $\wC$ be the functor which associates to a finitely generated 
Lie algebra $\h$ the direct limit of $\cdga$s 
\begin{equation}
\label{eq:chat}
\wC (\h)= \varinjlim_n \cC (\h/\gamma_n(\h)).
\end{equation}
This functor sends finite-dimensional central 
Lie extensions to Hirsch extensions of $\cdga$s.
It follows that $\wC (\h)$ is a $1$-minimal $\cdga$.

Now let $(A,d)$ be a $1$-finite $\cdga$, with holonomy Lie algebra $\h=\h(A)$. 
By \eqref{eq:flathol}, the identity map of $\h$ may be identified with 
the `canonical' flat connection, 
\begin{equation}
\label{eq:can}
\omega = \sum_i x_i^* \otimes x_i \in \F (A, \h(A)),
\end{equation}
where $\{ x_i \}$ is a basis for $A_1$ and $\{ x_i^* \}$ 
is the dual basis for $A^1$. This gives rise to a compatible 
family of flat connections, $\{ \omega_n \in \F(A, \h/\gamma_n(\h)) \}_{n\ge 1}$.
Using the correspondence (\ref{eq:flatcochains}), we obtain a compatible 
family of $\cdga$ maps, $f_n \colon \cC (\h/\gamma_n (\h)) \to A$. 
Passing to the limit, we arrive at a natural $\cdga$ map,
$f\colon \wC (\h(A)) \to  A$. 
The next theorem recovers (in a self-contained way) results 
from \cite{Bez}, \cite{BH}, and \cite{BMPP}.  

\begin{theorem}[\cite{PS-jlms}]
\label{thm:nat1model}
If $A$ is a $1$-finite $\cdga$, then 
the classifying map $f\colon \wC (\h(A)) \to A$ 
is a $1$-minimal model map for $A$.
\end{theorem}

Consequently, we have an isomorphism $\M_1(A)\cong \wC (\h(A))$.

\subsection{The holonomy Lie algebra of a group}
\label{subsec:holo lie group}
A construction due to K.-T.~Chen \cite{Chen73} and further developed in 
the works mentioned in Section \ref{subsec:holo-ga} assigns to every finitely 
generated group $G$ its holonomy Lie algebra, $\h(G;\k)$, which 
is defined as the holonomy Lie algebra of the cohomology algebra 
of $G$ with coefficients in a field $\k$ of characteristic $0$,
\begin{equation}
\label{eq:hololie-group}
\h(G;\k)\coloneqq \h\big(H^{*}(G,\k)\big).
\end{equation}
By the discussion from Section \ref{subsec:holo-ga}, we have that 
$\h(G;\k)=\Lie(H_1(G,\k))/\langle \mu_G^{\vee} \rangle$, where 
$\mu_G\colon H^1(G,\k)\wedge H^1(G,\k)\to H^2(G,\k)$ is the 
cup-product map in group cohomology and $\mu_G^{\vee}$ is 
its $\k$-dual. It is readily seen that the assignment 
$G\leadsto \h(G;\k)$ is functorial.

The Lie algebra $\h(G;\k)$ is a finitely presented, 
quadratic Lie algebra that depends only on the cup-product 
map $\mu_G$. Moreover, 
as noted in \cite{SW-jpaa}, the projection map 
$G\surj G/\gamma_n(G)$ induces isomorphisms 
$\h(G;\k)\isom \h(G/\gamma_n(G);\k)$ for all $n\ge 3$. 
Consequently, the holonomy Lie algebra of $G$ depends 
only on its second nilpotent quotient, $G/\gamma_3 (G)$.

An important feature of the holonomy Lie algebra
is its relationship to the associated graded Lie algebra, 
as detailed in the next theorem. 

\begin{theorem}[\cite{MP}, \cite{PS-imrn04}, \cite{SW-jpaa}]
 \label{thm:holoepi}
There exists a natural epimorphism of graded $\k$-Lie algebras, 
$\Phi\colon \h(G; {\k}) \surj \gr(G;\k)$, 
which induces isomorphisms in degrees $1$ and $2$.  
\end{theorem}

Following \cite{SW-jpaa}, \cite{SW-forum}, we say that a finitely generated 
group $G$ is {\em graded formal}\/ if the map 
$\Phi \colon \h(G;\k) \surj \gr(G; \k)$ is an isomorphism. 
This condition is equivalent to $\gr(G; \k)$ being a quadratic Lie algebra. 
As shown in \cite{SW-forum}, if $K\le G$ is a retract 
of a graded formal group $G$, then $K$ is also graded formal. 

The next result shows how to find a presentation for $\h(G;\k)$, 
given a presentation for $\gr(G;\k)$.
 
\begin{proposition}[\cite{SW-forum}]
 \label{prop:closure}
Let $V=H_1(G;\k)$. 
Suppose the associated graded Lie algebra $\g=\gr(G;\k)$ 
has presentation $\Lie(V)/\fr$. 
Then the holonomy Lie algebra $\h(G;\k)$ has presentation 
$\Lie(V)/\langle \fr_2\rangle $, where $\fr_2=\fr\cap \Lie_2(V)$.  
Furthermore, if $A=U(\g)$, then $\h(G;\k)=\h\big((\qA)^{!}\big)$. 
\end{proposition}

\section{Algebraic models for spaces}
\label{sect:algmod}

\subsection{Rational homotopy equivalences}
\label{subsec:rht}
We start with a definition that goes back to the work of Quillen \cite{Qu69}, 
Bousfield--Guggen\-heim \cite{BG76}, and Halperin--Stasheff \cite{HS79}. 
A continuous map between two topological spaces, $f\colon X\to Y$,  
is said to be a {\em rational quasi-isomorphism}\/ if the induced map 
in rational cohomology, $f^*\colon H^{*}(Y,\Q)\to H^{*}(X,\Q)$, 
is an isomorphism.  A {\em rational homotopy equivalence}\/ 
between $X$ and $Y$ is a sequence of continuous maps 
(going either way) connecting the two spaces via rational 
quasi-isomorphisms.  We say that $X$ and $Y$ are {\em 
rationally homotopy equivalent}\/ (or, have the same rational 
homotopy type) if such a zig-zag of rational quasi-isomorphisms 
exists, in which case we write $X\simeq_{\Q} Y$.  The purpose 
of rational homotopy theory, then, is to classify topological spaces 
up to this equivalence relation. 

One of the motivations of Sullivan’s work in this field was  
the idea that the rational homotopy type of a simply connected 
manifold, together with suitable characteristic class and integral data 
determines the diffeomorphism type up to finite ambiguity. For instance, 
he showed in \cite[Theorem 13.1]{Sullivan77} that closed, simply connected, 
smooth manifolds can be classified up to finite ambiguity in terms of 
their rational homotopy type, rational Pontrjagin classes, bounds on torsion, 
and certain integral lattice invariants. This important result was 
subsequently refined by Kreck and Triantafillou \cite{KT} (under 
some partial formality assumptions) and Crowley and Nordstr\"{o}m \cite{CN} 
(under some higher connectivity assumptions). 

\subsection{Sullivan algebras of piecewise polynomial differential forms}
\label{subsec:algmodels}
Let $(C^{*}(X,\k),d)$ be the singular cochain complex of a space $X$, 
with coefficients in a field $\k$ of characteristic $0$. This is, in fact, 
a differential graded algebra, with multiplication given by the cup-product. 
By definition, the cohomology of this $\k$-$\dga$ is the cohomology algebra 
$H^*(X,\k)$; this is a $\cga$, although the cochain algebra itself is not a 
$\cdga$ (except in some very special situations).  More generally, we say 
that a $\k$-$\dga$ $(A,d_A)$ is a $\dga$ model for $X$ if it is weakly equivalent 
(through $\dga$s) to $(C^{*}(X,\k),d)$. 

In his seminal paper \cite{Sullivan77},  Sullivan associated in 
a functorial way to every space $X$ a rational, {\em commutative}\/ 
$\dga$, denoted by $(\apl(X),d)$. When $X$ is a simplicial complex, 
the elements of this $\cdga$ may be viewed as compatible collections 
of forms on the simplices of $X$, which are sums with rational coefficients 
of monomials in the barycentric coordinates.  
Integration defines a chain map from $\apl(X)$ to $C^*(X,\Q)$ which 
induces an isomorphism in cohomology. In fact, Sullivan's algebra 
$(\apl(X),d)$ is weakly equivalent (through $\dga$s) with the cochain 
algebra $(C^*(X,\Q),d)$; moreover, under the resulting identification 
of graded algebras, $H^{*}(\apl(X)) \cong H^{*}(X,\Q)$, the induced 
homomorphisms in cohomology correspond, see \cite[Corollary 10.10]{FHT}. 
  
We say that a $\k$-$\cdga$ $(A,d_A)$ is a {\em model}\/
over $\k$ for the space $X$ if $A$ is weakly equivalent (through $\cdga$s) 
to $\apl(X)\otimes_{\Q} \k$; in particular, $H^*(A)\cong H^*(X,\k)$. 
In view of Theorem \ref{thm:cprw}, 
we may also allow the weak equivalence to go through $\dga$s 
in this definition. For instance, if $X$ is a smooth manifold, then 
the de Rham algebra $\Omega^{*}_{\dR}(X)$ of smooth forms on 
$X$ is a model of $X$ over $\R$, and if $X$ is a simplicial 
complex, then a rational model for $X$ is $A_{{\rm s}}(X)$, the 
algebra of piecewise polynomial $\Q$-forms on the simplices of $X$. 
We refer to \cite{FHT, FHT2, FOT, Tanre} for more details.

By the functoriality of the Sullivan algebra, a rational 
quasi-isomorphism $f\colon X\to Y$ induces a quasi-isomorphism 
$\apl(f)\colon \apl(Y)\to \apl(X)$; therefore, if $X \simeq_{\Q} Y$, then $\apl(X)\simeq \apl(Y)$. 
Consequently, the weak isomorphism type of $\apl(X)$ depends only on the rational 
homotopy type of $X$. As another consequence, the  existence of a finite model for 
a space $X$ is an invariant of rational homotopy type, and thus, of homotopy type.

\begin{remark}
\label{rem:non-descent}
In \cite{Sullivan77}, Sullivan showed that there exist smooth 
manifolds whose rational models are not weakly isomorphic, but which 
become weakly isomorphic when tensored with $\R$. Such failure 
of descent from real homotopy type to rational homotopy type 
may even occur with models endowed with $0$-differentials.  
\end{remark}

\subsection{Sullivan minimal models}
\label{subsec:sullivan-models}
A {\em minimal model}\/ for a connected space $X$, denoted $\M(X)$, is a 
minimal model for the Sullivan algebra $\apl(X)$.  By Theorem \ref{thm:mm}, 
this a minimal $\cdga$, which always exists and is unique up to isomorphism.  
Sullivan's minimal model comes equipped with a $\cdga$ map, 
$\rho \colon \M(X)\to \apl(X)$, which is a quasi-isomorphism. 
Moreover, if $A\simeq \apl(X)$ is a connected rational $\cdga$ 
model for $X$, then there 
is a quasi-isomorphism $\M(X) \to A$ which corresponds to $\rho$ 
via the chosen weak equivalence between $A$ and $\apl(X)$. 
By a previous remark, the isomorphism type of $\M(X)$ is uniquely 
defined by the rational homotopy type of $X$. It is an open question 
whether there exist spaces with weakly equivalent cochain algebras 
but non-isomorphic minimal models, see \cite{FH17}.

All these notions have partial analogs. Fix an integer $q\ge 1$. 
A {\em $q$-model}\/ over $\k$ for a space $X$ is a $\k$-$\cdga$ $(A,d)$ 
which is $q$-equivalent to $\apl(X)\otimes_{\Q} \k$; in particular, 
$H^{i}(A)\cong H^{i}(X,\k)$, for all $i\le q$.
A {\em $q$-minimal model}\/ for $X$, denoted $\M_q(X)$, is a 
$q$-minimal model $\apl(X)$; this $\cdga$ comes equipped with 
a $q$-quasi-isomorphism, 
\begin{equation}
\label{eq:q-model}
\begin{tikzcd}[column sep=20pt]
\rho_q \colon \M_q(X)\ar[r] &\apl(X).
\end{tikzcd}
\end{equation}

A map $f\colon X\to Y$ is said to be a {\em $q$-rational homotopy 
equivalence}\/ if the induced map in rational cohomology, 
$f^*\colon H^{*}(Y,\Q)\to H^{*}(X,\Q)$, is an isomorphism 
in degrees up to $q$ and a monomorphism in degree $q+1$.  
Clearly, such a map induces a $q$-equivalence,  
$\apl(f)\colon \apl(Y)\to \apl(X)$. 

In this context, a basic question was raised in \cite{PS-jlms}: 
When does a $q$-finite space $X$ admit a $q$-finite $q$-model $A$?
It follows from the above considerations that the  
existence of a $q$-finite $q$-model for a space $X$ is an invariant of 
$q$-rational homotopy type, and thus, of $q$-homotopy type.

\subsection{Rational completion}
\label{subsec:BK}

In their foundational monograph \cite{BK72}, Bousfield and Kan associated 
to every space $X$ its {\em rational completion}, $\Q_{\infty}X$. This is 
a rational space (i.e., its homology groups in positive degrees 
are $\Q$-vector spaces) which comes equipped with a structure 
map, $k_X\colon X\to \Q_{\infty}X$, with the following property: 
Given a map $f\colon X\to Y$, 
there is an induced map, $\Q_{\infty}f\colon \Q_{\infty}X\to \Q_{\infty}Y$, 
such that $\Q_{\infty}f\circ k_X=k_Y\circ f$. Moreover, the map $f$ 
is a rational homology equivalence if and only if the map 
$\Q_{\infty}f$ is a weak homotopy equivalence. 

A space $X$ is called {\em $\Q$-good}\/ if the structure map 
$k_X\colon X\to \Q_{\infty}X$ is a rational quasi-isomorphism. 
It has been known for a long time that not all spaces enjoy this property. 
Recently, Ivanov and Mikhailov \cite{IM} gave the first examples 
of finite CW-complexes that are $\Q$-bad: if $X=\bigvee^{n} S^1$ is a wedge 
of $n\ge 2$ circles, then $H_2( \Q_{\infty}X, \Q)$ is non-zero (in fact, it is 
uncountable), although of course $H_2(X,\Q)=0$.

The main connection between Sullivan's minimal model $\M(X)$ and  
Bousfield and Kan's rational completion $\Q_{\infty}X$ is provided by the 
following theorem of Bousfield and Guggenheim \cite{BG}.  

\begin{theorem}[\cite{BG}]
\label{thm:bg-bk-sv}
Let $X$ be a connected space with finite Betti numbers, and let 
$\M(X)=\big( \bwedge V, d)$ be a minimal model for $X$ over $\Q$. 
Then $\pi_n(\Q_{\infty}X)\cong (V^n)^{\vee}$, for all $n\ge 2$.
\end{theorem}

A connected space $X$ is a said to be {\em rationally aspherical}\/ (or, a 
rational $K(\pi,1)$ space) if its rational completion is aspherical, i.e.,
$\pi_n(\Q_{\infty}X)=0$ for all $n\ge 2$. As an application 
of the above theorem, we have the following immediate corollary. 

\begin{corollary}[\cite{Falk}, \cite{PY99}]
\label{cor:nilp-bk-mm}
A connected space $X$ is rationally aspherical if and only if $\M(X)\cong \M_1(X)$.
\end{corollary}

\subsection{Nilpotent spaces}
\label{subsec:nilp-spaces}

For simply-connected spaces and, more generally, for nilpotent spaces, 
rational homotopy theory takes a more concrete and approachable form. 
A path-connected space $X$ is said to be {\em nilpotent}\/ if the fundamental group 
$G=\pi_1(X)$ is nilpotent and acts nilpotently on the homotopy groups 
$\pi_n (X)$ for all $n>1$. For instance, all tori are nilpotent, but the Klein bottle 
is not; moreover, a real projective space $\RP^{n}$ is nilpotent if and only if $n$ is odd.

If $X$ is a nilpotent space, then, as shown in \cite{BK72}, $X$ is $\Q$-good. 
Moreover, $\pi_1(\Q_{\infty}X)$ is isomorphic to $\pi_1(X)\otimes \Q$---the 
Malcev completion of the nilpotent group $\pi_1(X)$---%
while $\pi_n(\Q_{\infty}X)\cong \pi_n(X)\otimes \Q$ for $n\ge 2$, all in a functorial way. 
In this context, we also have the following rational analog of Whitehead's theorem 
(see also \cite{RWZ}). 

\begin{theorem}[\cite{BK72}]
\label{thm:wh-bk}
A pointed map $f\colon  X\to Y$ between two nilpotent spaces is a 
rational homotopy equivalence if and only if it induces isomorphisms 
$f_{*}\colon \pi_n(X)\otimes \Q \to \pi_n(Y)\otimes \Q$ for all $n\ge 1$.
\end{theorem}

Assume now that $X$ is a nilpotent CW-complex with finite Betti numbers. 
Sullivan proved in \cite{Sullivan77} that the minimal model (over $\Q$) 
of such a space is of the form 
$\M (X)= (\bigwedge V, d)$, where $V$ is a graded $\Q$-vector space 
of finite type.  Here are a few standard examples.

\begin{example}
\label{ex:models-nilp}
An odd-dimensional sphere has minimal model 
$\M(S^{2n+1})=\big(\bwedge (a) ,0\big)$, with $\abs{a}=2n+1$. 
On the other hand, an even-dimensional sphere has minimal model 
$\M(S^{2n})=\left(\bwedge (a,b) ,da=b^2\right)$, with $da=0$, $db=a^2$, 
and $\abs{a}=2n$.  Finally, an Eilenberg--MacLane space 
of type $\K(\Z,n)$ has minimal model $\big(\bwedge (a) ,0\big)$, 
with $\abs{a}=n$. 
\end{example}

If $H^{>n}(X)=0$ for some $n>0$, we can say a bit more. Pick a vector 
space decomposition, $\M^n (X)= Z^{n}(\M (X)) \oplus C^{n}$.  Then  
the direct sum $J= \M^{\ge n+1}(X) \oplus C^n$ is an acyclic
differential graded ideal of $\M (X)$. By construction, 
$\apl (X)$ is weakly isomorphic to the $\cdga$ 
$\M (X)/J$, which is finite-dimensional as a vector space.  We 
summarize this discussion, as follows.

\begin{theorem}[\cite{Sullivan77}]
\label{thm:nilpsp}
Let $X$ be a nilpotent CW-complex. 
\begin{enumerate}
\item \label{nilp1} 
If all the Betti numbers of $X$ are finite, 
then $X$ admits a $q$-finite $q$-model, for all $q$.
\item  \label{nilp2}
If $\dim H_{*}(X, \Q)<\infty$, then $X$ admits a model which is 
finite-dimensional over $\Q$.
\end{enumerate}
\end{theorem}

The main application of Sullivan's theory of minimal models to the rational 
homotopy of nilpotent spaces is given by the following theorem. 

\begin{theorem}[\cite{Sullivan77}]
\label{thm:nilp-homotopy}
Let $X$ be a connected, nilpotent CW-complex with finite 
Betti numbers, and let $\M(X)=\big( \bwedge V, d)$ be a 
minimal model for $X$ over $\Q$. Then 
$\pi_n(X)\otimes \Q\cong (V^n)^{\vee}$, for all $n\ge 2$.
\end{theorem}

An alternative proof of this foundational result was given 
by Lehmann in \cite{Le}.  A generalization was given by 
Bock \cite{Bock}, who relaxed the hypothesis that 
$\pi_1(X)$ be nilpotent, thereby proving a statement 
first mentioned by Halperin in \cite{Hal83}. 

\begin{theorem}[\cite{Bock}]
\label{thm:nilp-homotopy-bis}
Let $X$ be a path-connected, triangulable space 
whose universal covering exists. Suppose $\pi_1(X)$ has a 
rationally aspherical classifying space and $\pi_n(X)$ is 
a finitely generated nilpotent $\pi_1(X)$-module, for each $n\ge 2$. 
If $\M(X)=\big( \bwedge V, d)$ is a minimal model for $X$ over 
$\Q$, then $\pi_n(X)\otimes \Q\cong (V^n)^{\vee}$, for all $n\ge 2$.
\end{theorem}

Consider now the rational Hurewicz homomorphisms,   
$\operatorname{hur}_k \colon \pi_{k}(X)\otimes \Q\to H_{k}(X,\Q)$. 
If $X$ is $n$-connected for some $n\ge 1$, 
the above theorem implies that 
 $\operatorname{hur}_k$ is an isomorphism for $k\le 2n$,  
while $\ker(\operatorname{hur}_k)$ is the $\Q$-span of 
 the Whitehead products for $2n+1\le k\le 3n+1$, see \cite{FH17}. 
For generalizations of Theorem \ref{thm:nilp-homotopy} 
to rationally nilpotent spaces we refer to \cite{FHT2}. 

\subsection{Models for polyhedral products}
\label{subsec:pp}

We illustrate the general theory with a class of spaces particularly amenable 
to study via rational homotopy methods. These spaces, variously known as 
polyhedral products, (generalized) moment-angle complexes, or (generalized) 
Davis--Januszkiewicz spaces, are constructed as follows (see for instance 
\cite{DeS07}, \cite{BBCG} and references therein). 

Let $K$ be a finite simplicial complex on vertex set $[n]=\{1,\dots,n\}$, 
and let $(\uX,\uX')$ be a sequence $(X_1,X_1'), 
\dots , (X_n,X_n')$ of pairs of spaces. The polyhedral product 
$\ZZ_K(\uX,\uX')$ is then 
the subspace of the cartesian product $\prod_{i=1}^n X_i$ 
obtained as the union of all subspaces of the form 
$\ZZ_{\sigma}(\uX,\uX')=\prod_{i=1}^n Y_i$, 
where $\sigma$ runs through the simplices of $K$ and 
$Y_i=X_i$ if $i\in \sigma$ and $Y_i=X_i'$ if $i\notin \sigma$.

Assume now that all spaces $X_i$, $X_i'$ are nilpotent 
CW-complexes of finite type. In \cite{FT09}, F\'elix and Tanr\'e describe 
the rational homotopy type of the corresponding polyhedral product, as follows. 
Let $A_i$ and $A_i'$ be connected, finite-type rational models for 
$X_i$ and $X_i'$, so that there are quasi-isomorphisms 
$\M(X_i)\to A_i$ and  $\M(X_i')\to A_i'$. 
Suppose there are surjective morphisms 
$\varphi_i \colon A_i\surj A_i'$ modeling the inclusions $X_i' \inj X_i$. 
For each simplex $\sigma$ on $[n]$, let $I_{\sigma}=\prod_{i=1}^n E_i$, 
with $E_i=\ker(\varphi_i)$ if $i\in \sigma$ and 
$E_i=A_i$ if $i\notin \sigma$.

\begin{theorem}[\cite{FT09}]
\label{thm:ft-zk}
With asumptions as above, 
the polyhedral product $\ZZ_K(\uX,\uX')$ has a connected, 
finite-type $\cdga$ model of the form $A(K)=(\bigotimes_{i=1}^n A_i)/I(K)$, 
where $I(K)$ is the ideal $\sum_{\sigma\notin K} I_{\sigma}$. 
Moreover, if $L\subset K$ is a subcomplex, then the 
inclusion $\ZZ_L(\uX,\uX')\inj \ZZ_K(\uX,\uX')$
is modelled by the projection $A(K)\surj A(L)$.
\end{theorem}

Taking homology, this theorem recovers a result from \cite{BBCG}: 
the cohomology algebra $H^*(\ZZ_K(\uX,\uX'),\Q)$ is isomorphic 
to the quotient $(\bigotimes_{i=1}^n H^*(X_i,\Q)/J(K)$, where $J(K)$ 
is the Stanley--Reisner ideal generated by all the monomials 
$x_{j_1}\cdots x_{j_k}$ with $x_{i}\in  H^*(X_i,\Q)$ for which 
the simplex $\sigma=(j_1,\dots, j_k)$ is not in $K$.

\subsection{Configuration spaces}
\label{subsec:config}
A construction due to Fadell and Neuwirth associates 
to a space $X$ and a positive integer $n$ the space of 
ordered configurations of $n$ points in $X$, 
\begin{equation}
\label{eq:conf gamma}
\Conf(X,n) = \{ (x_1, \dots , x_n) \in X^{\times n} 
\mid x_i \ne x_j \text{ for } i\ne j\}.
\end{equation}
The most basic example is the configuration space 
of $n$ ordered points in $\C$; this is a classifying 
space for $P_n$, the pure braid group on $n$ strings, 
whose cohomology ring was computed by Arnol'd in the 
the late 1960s.  

The $E_2$-term of the Leray spectral sequence for 
the inclusion $\Conf(X,n)\inj X^{\times n}$ was described in 
concrete terms by Cohen and Taylor \cite{CT78}. 
If $X$ is a smooth, complex projective variety 
of dimension $m$, then $\Conf(X,n)$ is a smooth, 
quasi-projective variety; moreover, as shown by Totaro in \cite{To96}, 
the Cohen--Taylor spectral sequence collapses at the 
$E_{m+1}$-term, and the $E_m$-term is a $\cdga$ model 
for the configuration space $\Conf(X,n)$. Other rational 
models for configuration spaces of smooth projective varieties 
were constructed by Fulton--MacPherson \cite{FMac} and 
K\v{r}\'{\i}\v{z} \cite{Kriz}. 

Now let $M$ be a closed, simply-connected smooth manifold. 
Under the assumption that $b_2(M)=0$, Lambrechts and 
Stanley \cite{LS-aif} showed how to construct a rational model 
for $\Conf(M,2)$ out of a model for $M$; as a consequence, 
the rational homotopy type of $\Conf(M,2)$ depends only on that of $M$.
For configuration spaces of $n$ points, Lambrechts and 
Stanley \cite{LS-asens} used Theorem \ref{thm:LS} to 
associate to every rational model $A$ for $M$ a $\Q$-$\cdga$ 
$G_A(n)$, which they conjectured to be a rational model for 
$\Conf(M,n)$.  In \cite{Idrissi}, Idrissi proved that 
$G_A(n)\otimes_{\Q} \R$ is a real model for the configuration space;  
thus, the real homotopy type of $M$ determines the real homotopy 
type of $\Conf(M,n)$, for all $n$.

\subsection{Rationalization}
\label{subsec:rationalization}
To every space $X$, Sullivan  \cite{Sullivan74}, \cite{Sullivan77}, \cite{Sullivan05} 
associated in a functorial way its {\em rationalization}, denoted $X_{\Q}$; 
we refer to \cite{BG76}, \cite{FHT2}, \cite{FH17}, \cite{RWZ}, and \cite{Iv} for 
more details on this construction. The rationalization of $X$ may be 
viewed as a geometric realization 
of the Sullivan minimal model, $\M(X)$, for the $\cdga$ $\apl(X)$. 
The space $X_{\Q}$  comes equipped with a structure map, $h\colon X\to X_{\Q}$, 
which realizes the morphism $\rho\colon \M(X)\to \apl(X)$. 

Now suppose $X$ is a connected, pointed CW-complex which is 
a nilpotent space; then, as shown in \cite{FHT2}, the space $X_{\Q}$ 
is again nilpotent and the map $h$ is a rational homotopy equivalence. 
Moreover, if $H^*(X,\Q)$ is of finite type, then the maps $h_*\colon 
\pi_n(X)\otimes \Q \to \pi_n(X_{\Q})$ are isomorphisms, for all $n\ge 2$. 
The nilpotency condition is crucial here. Indeed, if $X=\RP^2$, then 
$\pi_1(X)=\Z_2$ is nilpotent but does not act nilpotently on $\pi_2(X)=\Z$; 
we also have that $X_{\Q}\simeq \{*\}$, and so the map $h_*\colon 
\pi_2(X)\otimes \Q \to \pi_2(X_{\Q})$ is the zero map.

In general, the Bousfield--Kan completion and the Sullivan  
rationalization do not agree, even for nilpotent spaces. 
Nevertheless, if $X$ is nilpotent and $H^*(X,\Q)$ is of 
finite type, then $\Q_{\infty}X = X_{\Q}$, see \cite{BG76}.

When $X$ is a CW-complex, a more concrete way to construct the 
rationalization $X_{\Q}$ is via Sullivan's infinite telescopes, 
introduced in \cite{Sullivan74}.  
For instance, if $n$ is odd, then $S^n_{\Q}\simeq K(\Q,n)$.

The constructions from Section \ref{sect:malcev-lie} are related 
to the rationalizations of spaces, as follows. 
Let $X$ be a path-connected space with fundamental group $\pi_1(X)=G$. 
Then $\fM(G;\Q)=\pi_1(X_{\Q})$, the fundamental group of the rationalization of $X$. 

\subsection{Equivariant algebraic models}
\label{subsec:equiv}

The study of the rational equivariant homotopy type of a space subject 
to the action of a finite group goes back to the work of Triantafillou \cite{Tr} 
on equivariant minimal models. We summarize here some recent work 
from \cite{PS-jlms} on this subject.

Let $\Phi$ be a finite group. The category $\Phi$-$\cdga$ (over $\k$) 
has objects $\cdga$s $A$ endowed with a compatible $\Phi$-action, 
while the morphisms are $\Phi$-equivariant $\cdga$ maps  
$A\to B$. Given a $\Phi$-$\cdga$ $A$, we let $A^{\Phi}$ be the 
sub-$\cdga$ of elements fixed by $\Phi$; there is then a canonical 
$\cdga$ map $A^{\Phi}\to A$. By definition, a $q$-equivalence $A \simeq_q B$ 
in $\Phi$-$\cdga$ ($1\le q\le \infty$) is a zigzag of $\Phi$-equivariant 
$q$-equivalences in $\cdga$. It is readily seen that 
$A \simeq_q B$ in $\Phi$-$\cdga$ implies that $A^{\Phi} \simeq_q B^{\Phi}$ 
in $\cdga$.

Now suppose $\Phi$ acts freely on a space $Y$, 
and let $X=Y/\Phi$ be the orbit space.  
As is well-known, every CW-complex $X$ has the homotopy type 
of a simplicial complex $K$; moreover, if $X$ has finite $q$-skeleton, 
so does $K$. Fix such a triangulation of $X$, and lift it to the cover $Y$.   
The corresponding simplicial Sullivan algebras are then related by the 
equality $A_{\rm s}(X)=A_{\rm s}(Y)^{\Phi}$. Therefore, we have the following result.

\begin{proposition}[\cite{PS-jlms}]
\label{prop:equiv}
Let $X$ be a CW-complex, and let $Y\to X$ be a finite regular  
cover, with group of deck transformations $\Phi$.  Let $A$ 
be a $\Phi$-$\cdga$ over $\k$.
\begin{enumerate}
\item  \label{phi1}
Suppose $\apl(Y)\otimes_{\Q}\k \simeq_q A$ in 
$\Phi$-$\cdga$, for some $1\le q\le \infty$. Then 
$\apl(X)\otimes_{\Q}\k \simeq_q A^{\Phi}$ in $\cdga$. 
\item  \label{phi2}
If, moreover, $A$ is $q$-finite, then $A^{\Phi}$ is 
$q$-finite.
\end{enumerate}
\end{proposition}

As a consequence, if $Y$ admits an equivariant $q$-finite $q$-model,   
then $X$ admits a $q$-finite $q$-model.  The hypothesis from 
part \eqref{phi1} in the above proposition cannot be completely dropped. 
Nevertheless, we have the following conjecture regarding algebraic 
models for the orbit space $X=Y/\Phi$ constructed from $\Phi$-equivariant 
models for $Y$.

\begin{conjecture}[\cite{PS-jlms}]
\label{conj:equiv}
Let $X$ be a connected $CW$-complex, and let $Y \to X$ be a finite,  
regular cover with deck group $\Phi$. Suppose that $Y$ has finite Betti 
numbers. Let $A$ be a $\Phi$-$\cdga$, and assume that there is a zig-zag of 
quasi-isomorphisms connecting $\apl (Y)\otimes_{\Q}\k$ to $A$ in $\cdga$,
such that the induced isomorphism between $H^{*}(Y,\k)$ and $H^{*}(A)$ is 
$\Phi$-equivariant. Then $A^{\Phi}$ is a model for $X$.
\end{conjecture}

\subsection{On the Betti numbers of minimal models}
\label{subsec:infobs}
We conclude this section with an obstruction to the existence of a $q$-finite 
$\cdga$ model $A$ for a space $X$, an obstruction expressed 
in terms of Betti numbers of the $q$-minimal model $\M_q(X)$ 
associated to $X$.

\begin{theorem}[\cite{PS-jlms}]
\label{thm:infobsq}
Let $X$ be a connected CW-space, and 
assume that one of the following conditions is satisfied.
\begin{enumerate}
\item  \label{c1}
 $X$ is $(q+1)$-finite. 
 
\item  \label{c2}
$\apl (X)\otimes_{\Q} \k\simeq_q A$, where $A$ a $q$-finite 
$\cdga$ over $\k$; or,
\end{enumerate}
Then $b_i(\M_q(X))< \infty$, for all $i\le q+1$.
\end{theorem}

\begin{proof}
Recall from \eqref{eq:q-model} that we have a $q$-quasi-isomorphism 
$\M_q(X)\to \apl(X)$. In case \ref{c1}, the claim follows at once. 
In case \ref{c2}, the discussion in Section \ref{subsec:min-mod} 
shows that $\M_q(X)$ is also a $q$-minimal model for $A$; 
thus, the claim follows from Proposition \ref{prop:betti-minimal}.
\end{proof}

\section{Algebraic models for groups}
\label{sect:algmod-groups}

\subsection{Malcev Lie algebras and $1$-minimal models}
\label{subsec:malcev-min}
Let $G$ be a group, and let $\M_1(G)$ be its $1$-minimal model, 
as described in Section \ref{sect:malcev-lie}.  By definition, this is 
a minimal $\cdga$ over $\Q$, generated in degree $1$. 
If $G=\pi_1(X)$ is the fundamental group of a path-connected 
space $X$, then any classifying map $X\to K(G,1)$ induces an isomorphism 
between the corresponding $1$-minimal models, $\M_1(X)\cong \M_1(G)$.
Consequently, the existence of a $1$-finite $1$-model for a path-connected 
space $X$ is equivalent to the existence a $1$-finite $1$-model for its fundamental 
group, $G=\pi_1(X)$. 

Assume now that $G$ is a finitely generated group. 
There is then a natural duality between the Malcev Lie 
algebra $\m (G)$, endowed with the inverse limit filtration 
given by \eqref{eq:malcevLie} and the $1$-minimal model $\M_1(G)$, 
endowed with the increasing filtration from \eqref{eq:filtration-minimal}. 
Recall that the latter filtration, $\{\M(i)\}_{i\ge 0}$, starts with $\M(0)=\Q$. 
Since $G$ is finitely generated, the vector space $V_1\coloneqq H^1(G,\Q)$ 
is finite-dimensional. Each sub-$\cdga$ $\M(i)$ is then a Hirsch extension 
of the form $\M(i-1)\otimes \bwedge V_{i}$, where 
$V_i =\ker\big(H^2(\M(i-1))\to H^2(\M_1(G))\big)$ is again finite-dimensional. 
Let $\fL(G)=\varprojlim_i \fL_i(G)$ be the 
pronilpotent Lie algebra functorially associated to the $1$-minimal 
model $\M_1(G)$ in the manner described in 
Section \ref{subsec:dual-min}. We then have the following 
basic correspondence between the aforementioned Lie algebras.

\begin{theorem}[\cite{Sullivan77}, \cite{CP}, \cite{GM13}]
\label{thm:sullivan}
There is a natural isomorphism between the towers of nilpotent 
Lie algebras $\{\m(G/\gamma_i(G))\}_{i\ge 0}$ and $\{\fL_{i}(G)\}_{i\ge 0}$, 
which gives rise to a functorial isomorphism of complete, filtered Lie 
algebras, $\m(G)\cong \fL(G)$.
\end{theorem}  

The functorial isomorphism $\m(G) \cong \fL(G)$, together with the 
dualization correspondence $\M_1(G) \leftrightsquigarrow \fL(G)$ 
define adjoint functors between the category of Malcev Lie algebras 
of finitely generated groups and 
the category of $1$-minimal models of finitely generated groups.
Using this isomorphism and the one from \eqref{eq:quillen-iso}, 
we may identify $\gr_n(G;\Q)$ with $(V_n)^{\vee}$ for all $n\ge 1$.

\subsection{Groups with $1$-finite $1$-models}
\label{subsec:1-finite-gps}

The next theorem provides an effective way of computing the 
Malcev Lie algebra of a group $G$, under a certain finiteness assumption. 

\begin{theorem}[\cite{PS-jlms}]
\label{thm:malholo}
Let $G$ be a finitely generated group that admits a 
$1$-finite $1$-model $A$. Then the Malcev Lie algebra 
$\m (G)$ is isomorphic to the $\lcs$ completion of the holonomy Lie 
algebra $\h (A)$.
\end{theorem}
 
\begin{proof} By our hypothesis and by the uniqueness 
of $1$-minimal models, we have an isomorphism $\M_1(G)\cong \M_1(A)$. 
By construction, the Lie algebra $\m (G)$ is filtered isomorphic to the 
inverse limit of a tower of central extensions of finite-dimensional 
nilpotent Lie algebras. By Theorem \ref{thm:sullivan}, the terms 
$\m(G/\gamma_i(G))$ of this tower are obtained by dualizing the 
canonical filtration of $\M_1(G)$. 

On the other hand, by 
Theorem \ref{thm:nat1model}, the $\cdga$ $\M_1(A)$ is isomorphic 
to $\wC (\h(A))$, the completion of the Chevalley--Eilenberg cochain functor 
applied to $\h(A)$. Furthermore, it is shown in \cite[Corollary 5.7]{PS-jlms} that
the dual of the canonical filtration of $\wC (\h(A))$ is a tower of central 
extensions of finite-dimensional Lie algebras, 
whose terms are the nilpotent quotients $\h(A)/\gamma_i(\h(A))$. 
Putting all these facts together yields the desired isomorphism, 
$\m(G)\cong\widehat{\h}(A)$. 
\end{proof}

As an application, we have the following result, 
which gives a characterization of groups $G$ having a $1$-finite 
$1$-model in terms of their Malcev Lie algebras.

\begin{theorem}[\cite{PS-jlms}]
\label{thm:psobs}
A finitely generated group $G$ admits a $1$-finite $1$-model if and 
only if the Malcev Lie algebra $\m(G)$ is the lower central series 
completion of a finitely presented Lie algebra over $\Q$.
\end{theorem}

The above condition means that $\m(G)=\widehat{L}$, for some finitely presented 
Lie algebra $L$ over $\Q$, where $\widehat{L}=\varprojlim_{n} L/\gamma_n (L)$. 
By Theorem \ref{thm:malholo}, if $A$ is a $1$-finite $1$-model for $G$, 
we may take $L$ to be the holonomy Lie algebra $\h(A)$. 

Finally, here is a finiteness obstruction for finitely generated groups, which 
follows at once from Theorem \ref{thm:infobsq}.

\begin{corollary}[\cite{PS-jlms}]
\label{cor:infobs1}
Let $G$ be a finitely generated group. Assume that either $G$ is finitely presented 
or $G$ admits a $1$-finite $1$-model. Then $b_2(\M_1(G))< \infty$.
\end{corollary}

\subsection{Filtered formal groups}
\label{subsec:filtered-formal} 

Recall from Section \ref{subsec:tff} that a finitely generated group $G$ 
is said to be filtered formal if its Malcev Lie algebra $\m(G)$ is isomorphic to 
the degree completion of its associated graded Lie algebra. 
The next result connects certain finiteness properties of algebraic 
objects associated to such a group $G$. 

\begin{proposition}[\cite{PS-jlms}]
\label{prop:filtf}
Let $G$ be a finitely generated, filtered formal group, so that 
$\m (G) \cong \widehat{L}$, where $L=\L/J$ is a graded Lie algebra 
over $\Q$ generated in degree $1$ and $J$ is an ideal included in $\L^{\ge 2}$.  
If $b_2(\M_1(G))< \infty$, then $\dim_{\Q} (J/[\L, J] )< \infty$.
\end{proposition}

Here is another characterization of filtered-formality, this time in 
terms of minimal models. 

\begin{theorem}[\cite{SW-forum}]
\label{thm:filtmin}
A finitely generated group $G$ is filtered-formal over $\Q$ if and only 
if the canonical $1$-minimal model $\mathcal{M}_1(G)$ is filtered-isomorphic 
to a $1$-minimal model $\mathcal{M}$ with positive Hirsch weights.
\end{theorem}

The notion of filtered formality over an arbitrary field $\k$ of characteristic $0$ 
is defined analogously. It follows from Theorem \ref{thm:ffdescent} that $G$ 
is filtered-formal over $\k$ if and only it is filtered-formal over $\Q$. Another 
notable property of filtered formality is that it descends to maximal solvable 
quotients. The next theorem develops a theme started in \cite{PS-imrn04}.

\begin{theorem}[\cite{SW-forum}]
\label{thm:chenlieiso}
Let $G$ be a finitely generated group. 
For each $i\ge 2$, the quotient map $q_i\colon G\surj G/G^{(i)}$ induces a 
natural epimorphism of graded $\k$-Lie algebras, 
\begin{equation}
\label{eq:psi-map}
\begin{tikzcd}[column sep=22pt]
\Psi^{(i)}\colon \gr(G;\k)/\gr(G;\k)^{(i)} \ar[r, two heads]&  \gr(G/G^{(i)};\k).
\end{tikzcd}
\end{equation}
Moreover, if $G$ is filtered-formal, then $\Psi^{(i)}$ is an isomorphism
and the solvable quotient $G/G^{(i)}$ is filtered-formal.
\end{theorem}

Taking $G=F_n$, it follows that each solvable quotient $F_n/F_n^{(i)}$ 
is a  filtered formal group, with associated graded Lie algebra 
equal to $\L_n/\L_n^{(i)}$, where $\L_n=\Lie(\Q^n)$ 
denotes the free $\Q$-Lie algebra on $n$ generators.

\subsection{Non-finiteness properties of certain metabelian groups}
\label{subsec:meta}

As an application of these techniques, 
we may construct a large class of metabelian groups that 
do not have good finiteness properties, either at the level of 
presentation complexes, or at the level of $1$-models. 

A finitely generated group $G$ is said to be {\em very large}\/ 
if it has a quotient a free group $F_n$ of rank $n$ greater or equal to $2$. 
The group $G$ is merely {\em large}\/ if it has a finite-index subgroup 
which is very large.

\begin{theorem}[\cite{PS-jlms}]
\label{thm:meta}
Let $G$ be a metabelian group of the form $G=\pi/\pi''$, where  
$\pi$ is very large.  Then $G$ is not finitely presentable and $G$ 
does not admit a $1$-finite $1$-model.
\end{theorem}

\begin{proof}
By assumption, there is an  
epimorphism $\varphi\colon \pi\surj F_n$, for some $n\ge 2$.  
Since the group $F_n$ is free, the map $\varphi$ admits a 
splitting, and thus, the induced homomorphism on maximal 
metabelian quotients, $\bar\varphi\colon \pi/\pi'' \surj F_n/F_n''$, 
also has a splitting.  By the homotopy functoriality 
of the $1$-minimal model construction from Theorem \ref{thm:nat1model}, 
the map $\bar\varphi$ induces a $\cdga$ map, 
$\bar\varphi^*\colon \M_1(F_n/F_n'') \to \M_1(\pi/\pi'')$, 
which is a split injection up to homotopy.

Suppose now that $\pi/\pi''$ admits a finite presentation, or 
a $1$-finite $1$-model. It then follows from Corollary \ref{cor:infobs1} 
that $b_2(\M_1(\pi/\pi''))< \infty$.  Since the map 
$\bar\varphi^*$ is split injective (up to homotopy), 
and since homology is a homotopy functor, we 
infer that $b_2(\M_1(F_n/F_n''))< \infty$. 
Hence, since $F_n/F_n''$ is filtered formal 
and $\L_n'' \subset \L^{\ge 2}$,  
Proposition \ref{prop:filtf} implies that the $\Q$-vector 
space $\L_n''/[\L_n, \L_n'']$ is finite-dimensional. 
On the other hand, a computation with Hall--Reutenauer 
bases done in \cite[Proposition 3.2]{PS-jlms} shows that 
$\dim_{\Q} ( \L_n''/[\L_n, \L_n''] ) =\infty$. This is a contradiction, 
and the proof is complete.
\end{proof}

\section{Formality of spaces, maps, and groups}
\label{sect:formal}

\subsection{Formal spaces}
\label{subsec:formal-spaces}

A space $X$ is said to be {\em formal}\/ (over a field $\k$ 
of characteristic $0$) if Sullivan's algebra $\apl(X)\otimes_{\Q} \k$ is formal, 
that is, it is weakly equivalent to the cohomology algebra $H^{*}(X,\k)$, 
equipped with the zero differential,
\begin{equation}
\label{eq:formal-space}
\apl(X)\otimes_{\Q} \k \simeq (H^{*}(X,\k),0).
\end{equation}
If $X$ is formal (over $\Q$), its rationalization $X_{\Q}$ depends only on 
$H^{*}(X,\Q)$.

The formality property behaves well with respect to field extensions of 
the form $\Q\subset \k$.  Indeed, Halperin and Stasheff's 
Corollary \ref{cor:Kformal} implies that a connected space $X$ with finite Betti 
numbers is formal over $\Q$ if and only if $X$ is formal over $\k$.  
This result was first stated and proved by Sullivan \cite{Sullivan77}, 
using different techniques, while an independent proof was 
given by Neisendorfer and Miller \cite{NM78} in the simply-connected case.

Formality is preserved under several standard operations on spaces. 
For instance, if $X$ and $Y$ are formal, then so is the product 
$X\times Y$ and the wedge $X\vee Y$; moreover, a retract of 
a formal space is formal; see \cite{FHT}, \cite{FOT} for details. 
In general, a finite cover of a formal space need not be formal;  
nevertheless, Conjecture \ref{conj:equiv} holds in the formal case, and  
leads to the following result. 

\begin{proposition}[\cite{PS-jlms}]
\label{prop:equivformal}
Suppose $\Phi$ is a finite group acting simplicially on a formal 
simplicial complex $Y$ with finite Betti numbers.  Then 
the orbit space $X=Y/\Phi$ is again formal.
\end{proposition}

The following result of Kreck and Triantafillou \cite{KT} fits into Sullivan's 
``determined up to finite ambiguit'' philosophy.

\begin{theorem}[\cite{KT}]
Let $H$ be a finitely generated graded commutative ring over $\Z$.
Then there are only finitely many homotopy types of simply 
connected, formal, finite $CW$-complexes 
with integral cohomology isomorphic to $H$.
\end{theorem}

\subsection{Formality criteria}
\label{subsec:formality-cr}

For nilpotent spaces, Sullivan gave a formality criterion in terms 
of lifting automorphisms of the cohomology algebra to the minimal model.

\begin{theorem}[\cite{Sullivan77}]
\label{thm:sullivan-formality}
Let $X$ be a nilpotent CW-complex with finite Betti numbers. Then 
$X$ is formal if and only if every automorphism of $H^*(X,\Q)$ can 
be realized by an automorphism of $\M(X)$.
\end{theorem}

Roughly speaking, the more highly connected a space is, the more likely 
it is to be formal.  This was made precise by Stasheff in \cite{Sta83}, as follows.  

\begin{theorem}[\cite{Sta83}]
\label{thm:stasheff}
Let $X$ be a $k$-connected CW-complex of 
dimension $n$; if $n \le 3k+1$, then $X$ is formal.
\end{theorem}

This is the best possible bound:  attaching a cell $e^{3k+2}$ to the 
wedge $S^{k+1}\vee S^{k+1}$ via the iterated Whitehead 
product $[\iota_1,[\iota_1,\iota_2]]$ yields a non-formal 
CW-complex. 

A powerful formality criterion was given by Sullivan in \cite{Sullivan77}.

\begin{theorem}[\cite{Sullivan77}]
\label{thm:regular}
If $H^*(X,\k)$ is the quotient of a free $\cga$ by an ideal generated 
by a regular sequence, then $X$ is a formal space. Consequently, 
if $H^*(M,\k)$ is a free $\cga$, then $X$ is formal.
\end{theorem}

This result provides a large supply of formal spaces, such as: 
rational cohomology spheres and tori; 
compact connected Lie groups $G$, as well as their 
classifying spaces, $BG$; 
homogeneous spaces of the form $G/K$, with 
$\rank G=\rank K$;  and 
Eilenberg--MacLane classifying spaces $K(G, n)$ for 
discrete groups $G$, provided $n\ge 2$.
In particular, if $X$ is the complement of a knotted sphere 
in $S^{n}$, $n\ge 3$, then $X$ is a formal space. 

On the other hand, not all homogeneous spaces are formal. 
For instance, the quotient spaces $\SU(pq)/(\SU(p)\times \SU(q))$ 
for $p,q\ge 3$; $\SO(n^2-1)/\SU(n)$ for $n\ge 3$; 
$\Sp(5)/\SU(5)$; and $\SO(78)/E_6$ are known to be non-formal.
Furthermore, $K(G, 1)$ spaces need not be formal. For instance, 
Hasegawa \cite{Has} showed that a classifying space for a 
torsion-free, finitely generated nilpotent group $G$ is formal 
if and only if $G$ is abelian. We refer to \cite{FOT} for more 
on these topics.

A connected space $X$ is said to be {\em intrinsically formal}\/ if any 
connected space whose rational cohomology algebra is isomorphic to 
$H^*(X,\Q)$ has the same rational homotopy type as $X$; in other 
words, if there is a unique rational homotopy type whose 
rational cohomology algebra is isomorphic to that of $X$. 

\begin{theorem}[\cite{Baues}, \cite{HS79}]
\label{thm:wedge-spheres}
Let $X$ be a connected space whose minimal model $\M(X)$ is of finite type. 
If $b_{2k}(X)=0$ for all $k\ge 1$, then $X$ is intrinsically formal and has 
the rational homotopy type of a wedge of odd spheres.
\end{theorem}

Although the spaces in the above theorem are intrinsically formal, they are 
typically not hyperformal. For instance, the space $X=S^{2k_1-1}\vee S^{2k_2-1}$ 
fits into this framework, but the cohomology algebra 
$H^*(X,\k)$ is isomorphic to $\bwedge (x_1,x_2)/(x_1x_2)$, 
with $\abs{x_i}=2k_i-1$, which is not hyperformal if $k_1\ne k_2$, 
since in that case $\{x_1x_2\}$ is not a regular sequence.

\subsection{Formality properties of closed manifolds}
\label{subsec:formal-mfd}

As shown by Miller \cite{Mi79}, the dimension bound from Theorem \ref{thm:stasheff}
can be relaxed for closed manifolds, by using Poincar\'{e} duality.  

\begin{theorem}[\cite{Mi79}]
\label{thm:miller}
Let $M$ be a closed, $k$-connected manifold ($k\ge 1$) 
of dimension $n \le 4k+2$. Then $M$ is formal.  
\end{theorem}

In particular, all simply-connected closed
manifolds of dimension at most $6$ are formal.  
Again, this is best possible: as shown by 
Fern\'andez and Mu\~noz in \cite{FM04}, there exist  
closed, simply-connected, non-formal manifolds in each  
dimension $n\ge 7$.  On the other hand, if $M$ is a closed, 
orientable, $k$-connected $n$-manifold with $b_{k+1}(M)=1$, 
then the bound insuring formality can be improved to $n \le 4k+4$, 
see Cavalcanti \cite{Ca}.  

Formality also behaves well with respect to standard operations 
on manifolds.  For instance, Stasheff \cite{Sta83} proved the 
following: If $M$ is a closed, simply-connected manifold such 
that the punctured manifold $M\setminus \set{*}$ is formal, 
then $M$ is formal.  Moreover, if $M$ and $N$ are closed, 
orientable, formal manifolds, so is their connected sum, $M\# N$;
see \cite{FHT}.

It has been shown by Cavalcanti \cite{Ca}, and, in stronger form, 
by Crowley and Nordstr\"{o}m in \cite{CN}, that a certain type of 
Hard Lefschetz property insures the intrinsic formality of highly 
connected manifolds. 

\begin{theorem}[\cite{CN}]
\label{thm:intrinsic-cn}
Let $M$ be an $(n-1)$-connected manifold of dimension $4n-1$. 
Suppose $b_n(M)\le 3$ and there is a cohomology class $u\in H^{2n-1}(M,\Q)$ 
such that the map $H^n(M,\Q)\to H^{3n-1}(M,\Q)$, $v\mapsto u v$ 
is an isomorphism. Then $M$ is intrinsically formal.
\end{theorem}

In the same paper, Crowley and Nordstr\"{o}m construct infinitely many 
simply-con\-nected, non-formal manifolds all of whose Massey products 
vanish (the smallest dimension of such a manifold is $7$). We summarize 
their results, as follows.

\begin{theorem}[\cite{CN}]
\label{thm:nf-Massey-vanish}
For each $k\ge 1$, there is a non-formal, $(2k-1)$-connected 
manifold of dimension $8k-1$ and a $(2k)$-connected manifold 
of dimension $8k+3$ such that all Massey products in the rational 
cohomology rings of these manifolds vanish. 
\end{theorem}

In \cite{DGMS}, Deligne, Griffiths, Morgan, and Sullivan showed that 
every compact K\"{a}hler manifold $M$ is formal. On the other hand, 
symplectic manifolds need not be formal: the simplest example is 
the Kodaira--Thurston manifold, which is the product of the circle with 
the $3$-dimensional Heisenberg nilmanifold (see 
Example \ref{ex:heis-massey} below). This led Lupton 
and Oprea  \cite{Lupton-Oprea} to raise the question
whether compact, simply-connected symplectic manifolds 
are formal. The question was answered in the negative by 
Babenko and Taimanov \cite{BT}, \cite{BT-2}, who used McDuff's symplectic 
blow-ups to construct non-formal, simply-connected symplectic 
manifolds in all even dimensions greater than $8$; an $8$-dimensional 
example was subsequently constructed by Fern\'andez and Mu\~noz \cite{FM08}. 
We refer to \cite{FGM}, \cite{TO}, \cite{RT}, \cite{LS-gt}, and \cite{FOT} for 
more on this subject. 

\subsection{Formal maps}
\label{subsec:formal-maps}
A continuous map $f\colon X\to Y$ is said to be {\em formal}\/ 
(over $\Q$) if the induced morphism between Sullivan models,  
$\apl(f)\colon \apl(Y)\to \apl(X)$, is formal, in the sense of 
Definition \ref{def:formal-map}. By the discussion from Section 
\ref{subsec:minmod-formal}, this condition is equivalent to the 
existence of a diagram of the form 
\begin{equation}
\label{eq:formal=map}
\begin{tikzcd}[column sep=24pt, row sep=22pt]
\apl(Y)  \ar{d}{\apl(f)}
&
\M(Y) \ar[swap, pos=.4]{l}{\rho_Y} 
\ar{d}{\M(f)}
\ar{r}{\psi_Y} 
& (H^{*}(Y,\Q), 0)\phantom{,} \ar{d}{f^*}
\\
 \apl(X)
&\M(X)\ar[swap, pos=.4]{l}{\rho_X} \ar{r}{\psi_X}
& (H^{*}(X,\Q), 0) ,
\end{tikzcd}
\end{equation}
which commutes up to homotopy and 
in which the horizontal arrows are quasi-isomor\-phisms. 
When $f$ is formal, the surjectivity of $f^*$ implies that of $\M(f)$. 

One may define in a similar fashion formality of 
maps over an arbitrary field $\k$ of characteristic $0$. 
As shown by Vigu\'{e}-Poirrier in \cite{VP79}, a map 
$f\colon X\to Y$ between two nilpotent CW-complexes 
of finite type is formal over $\k$ if and only if it is 
formal over $\Q$. Moreover, as shown by F\'{e}lix 
and Tanr\'{e} \cite{FT88}, the cofiber of such a map 
is a formal space.

\begin{example}
\label{ex:holo-formal}
Suppose $f\colon M\to N$ is a holomorphic map between two 
compact K\"{a}hler manifolds. Then, as shown in \cite{DGMS}, 
$f$ is a formal map over $\R$.
\end{example}

In general, though, a map between two formal spaces need 
not be formal. A simple example is provided by the Hopf map 
$f\colon S^3\to S^2$, for which $f^*\colon \widetilde{H}^*(S^2,\Q)\to 
\widetilde{H}^*(S^3,\Q)$ is the zero map, yet the induced morphism 
$\M(f)\colon \M(S^2)\to \M(S^3)$ is non-trivial.

The next result, due to Arkowitz \cite{Ark}, delineates a 
class of formal spaces $X$ and $Y$ for which every map 
$f\colon X\to Y$ is formal.

\begin{theorem}[\cite{Ark}]
\label{thm:arkowitz}
Let $X$ and $Y$ be simply connected, formal, rational spaces, 
and let $[X, Y]_{\mathrm{f}}$ be the set of homotopy classes 
of formal maps from $X$ to $Y$.
\begin{enumerate}
\item  \label{ak1}
The map $[X,Y]_{\mathrm{f}} \to \Hom(H^*(Y,\Q), H^*(X,\Q))$, 
$f\mapsto f^*$ is a bijection.
\item \label{ak2}
Further assume that $X$ and $Y$ are of finite type, $b_i(X)=0$ 
for $i\ge 2n+1$, and $Y$ is $n$-connected. Then every map 
$f\colon X\to Y$ is formal.
\end{enumerate}
\end{theorem}

\subsection{Partial formality}
\label{subsec:partial}
Let $q$ be a non-negative integer.  A space $X$ is said to be {\em $q$-formal}\/ 
(over a field $\k$ of characteristic $0$) if its Sullivan algebra is $q$-formal, 
that is, $(\apl(X)\otimes_{\Q} \k,d)\simeq_q (H^*(X,\k),0)$. 
Clearly, if $X$ is formal, then $X$ is $q$-formal for all $q\ge 0$.  
Under some additional hypothesis, this implication may be reversed.

\begin{theorem}[\cite{Mc10}]
\label{thm:partial-to-formal}
Let $X$ be a space such that $H^{i}(X,\k)=0$ for all $i\ge q+2$. 
Then $X$ is $q$-formal if and only if $X$ is formal.
\end{theorem}

In particular, the notions of formality and $q$-formality 
coincide for $(q+1)$-dimensional CW-complexes. 

\begin{example}
\label{ex:curves}
Let $V$ be a complex algebraic hypersurface in $\CP^{n}$, 
with complement $X=\CP^{n}\setminus V$. 
Work of Morgan \cite{Mo} shows that $X$ is $1$-formal, though 
not formal, in general. By Morse theory, $X$ has the homotopy 
type of a finite CW-complex of dimension $n$.  Thus, if $n=2$ 
(that is, $V$ is a plane curve), Theorem \ref{thm:partial-to-formal} 
implies that $X$ is formal.
\end{example}

\begin{example}
\label{ex:heis-massey}
Let $M=G_{\R}/G_{\Z}$ be the $3$-dimensional Heisenberg 
nilmanifold, where $G_{\R}$ is the group of real, 
unipotent $3\times 3$ matrices, and $G_{\Z}=\pi_1(M)$ 
is the subgroup of integral matrices in $G_{\R}$.  
This manifold has as a model the $\cdga$ $(A,d)$, where 
$A=\bwedge (a_1,a_2,b)$ with generators in degree $1$, 
and differential given by $da_i=0$ and $db=a_1a_2$. 
As noted in Example \ref{ex:gen-heisenberg}, this $\cdga$ 
is not $1$-formal. Alternatively, the triple Massey product 
$\langle [a_1],[a_1],[a_2] \rangle= \{[a_1b]\}$ 
is non-vanishing, with trivial indeterminacy.  
Therefore, $M$ is not $1$-formal. 
\end{example}

Partial formality enjoys a descent property analogous to that for full formality. 
Indeed, Theorem \ref{thm:i-formalField} has the following immediate corollary.

\begin{corollary}[\cite{SW-forum}]
\label{cor:pfext}
Let $X$ be a connected space such that $b_i(X)<\infty$ for $i\le q+1$. 
Then $X$ is $q$-formal over $\Q$ if and only if $X$ is $q$-formal over $\k$.
\end{corollary}

We may also consider a partial formality notion for maps. 
A continuous map $f\colon X\to Y$ is said to be {\em $q$-formal} if 
the morphism $\apl(f)\colon \apl(Y)\to \apl(X)$ is $q$-equivalent to 
the induced homomorphism in cohomology, $f^*\colon H^{*}(Y,\Q)\to H^{*}(X,\Q)$.

\subsection{Koszul algebras and formality}
\label{subsec:koszul-formal}

Let $A$ be a connected, locally finite $\k$-$\cga$. 
The trivial $A$-module $\k$ has a free, graded $A$-resolution of the form 
\begin{equation}
\label{eq:minres}
\begin{tikzcd}[column sep=20pt]
\cdots \ar[r, "\varphi_3"]& A^{n_2}  \ar[r, "\varphi_2"]&
A^{n_1}  \ar[r, "\varphi_1"]&A \ar[r]& \k \ar[r]& 0.
\end{tikzcd}
\end{equation}
Such a resolution is {\em minimal}\/ if all the nonzero 
entries of the matrices $\varphi_i$ have positive degrees.  
The algebra $A$ is said to be a {\em Koszul algebra}\/ if 
the minimal $A$-resolution of $\k$ is linear, or, equivalently, 
$\Tor_{i}^A(\k,\k)_{j}=0$ for all $i\ne j$. 
A necessary condition is that $A$ be expressed as the 
quotient $A=E/I$ of an exterior algebra on generators in 
degree~$1$ by an ideal $I$ generated in degree~$2$.  
A sufficient condition is that the ideal $I$ has a quadratic 
Gr\"{o}bner basis.
If $A$ is a Koszul algebra, then the quadratic dual $A^{!}$ is 
also a Koszul algebra and the following ``Koszul duality'' 
formula holds: 
\begin{equation}
\label{eq:kdual}
\Hilb(A,t)\cdot \Hilb(A^{!},-t)=1.
\end{equation}

The following theorem of Papadima and Yuzvinsky \cite{PY99}  
relates certain properties of the minimal model 
of a space $X$ to the Koszulness of its cohomology algebra.

\begin{theorem}[\cite{PY99}]
\label{thm:py}
Let $X$ be a connected space with finite Betti numbers.  
\begin{enumerate}
\item \label{py1}
If $\M(X)\cong \M_1(X)$, then the cohomology algebra 
$H^{*}(X,\Q)$ is a Koszul algebra.
\item  \label{py2}
If $X$ is formal and $H^{*}(X,\Q)$ is a Koszul algebra, then 
$\M(X)\cong \M_1(X)$.
\end{enumerate}
\end{theorem}

Consequently, if $X$ is formal, then $X$ is rationally aspherical 
if and only if $H^{*}(X,\Q)$ is a Koszul algebra. When $X$ is also 
a nilpotent space, Berglund \cite{Berglund} recovers this equivalence 
(without assuming the cohomology algebra is generated in degree $1$) 
and finds several alternative conditions yielding the same 
class of spaces, which he calls {\em Koszul spaces}. 

As an application of Theorem \ref{thm:py}, we have the 
following formality criterion.

\begin{corollary}[\cite{PS-mathann}]
\label{cor:ps-koszul}
Let $X$ be a connected, finite-type CW-complex, 
and suppose that $H^*(X,\Q)$ is a Koszul algebra.  
Then $X$ is $1$-formal if and only if $X$ is formal. 
\end{corollary}

\begin{example}
\label{ex:fiber-type}
Let $\A$ be an arrangement of linear hyperplanes in $\C^{n}$, 
with complement $X=\C^{n} \setminus \bigcup_{H\in \A} H$. 
Work of Arnol'd and Brieskorn from the 1960s shows that 
$X$ is formal. Now suppose $\A$ is a fiber-type arrangement, 
or, equivalently, if its intersection lattice, $L(\A)$, is supersolvable. 
Then $X$ is aspherical and $H^{*}(X,\Q)$ is a Koszul algebra. 
Theorem \ref{thm:py} implies that $X$ is also rationally aspherical 
(this is a result first proved by Falk \cite{Falk} by other methods). 
It is an open question whether the converse is true: If $X$ is rationally 
aspherical, is $\A$ necessarily of fiber-type? Put differently:  If $H^{*}(X,\Q)$ is 
a Koszul algebra, is $L(\A)$ necessarily supersolvable?
\end{example}

\subsection{The $1$-formality property for groups}
\label{subsec:1-formal}
A finitely generated group $G$ is said to be {\em $1$-formal}\/ 
(over a field $\k$ of characteristic $0$) 
if there is a classifying space $K(G,1)$ which is $1$-formal (over $\k$). 
In view of the discussion from Section \ref{subsec:sullivan-models}, 
we see that a connected CW-complex $X$ is $1$-formal if and 
only if its fundamental group, $G=\pi_1(X)$, is $1$-formal. 

Over $\Q$, the $1$-formality property of a group $G$ depends 
only on its Malcev Lie algebra, $\m(G)$, or its rationalization, $G_{\Q}$. 
This is a consequence of the following well-known theorem, proved for 
instance in \cite{CT}, \cite{MP}, \cite{SW-forum}.

\begin{theorem}
\label{thm:1f-group}
A finitely generated group $G$ is $1$-formal 
if and only if $\m(G)$ is isomorphic to the degree 
completion of a finitely generated, quadratic Lie algebra.
\end{theorem}

Let $\h(G)=\h(G;\Q)$ be the holonomy Lie algebra of $G$, as 
described in Section \ref{subsec:holo lie group}. As shown 
in \cite{PS-imrn04}, the $1$-formality of $G$ is equivalent 
to $\m(G)\cong \widehat{\h}(G)$. 

\begin{example}
\label{ex:free}
Let $F_n$ be the free group of rank $n\ge 1$.  We then have 
$H_1(F_n,\Q)=\Q^n$ and $H_2(F_n,\Q)=0$; hence, $\mu_{G}=0$ 
and so $\h(F_n)=\Lie(\Q^n)$, the free Lie algebra of rank $n$.  
On the other hand,  $\m(F_n)=\widehat{\Lie}(\Q^n)$, by 
Theorem \ref{thm:Massuyeau}. Therefore, $F_n$  is $1$-formal. 
\end{example}

\begin{example}
\label{ex:surface}
Let $\Sigma_g$ be the Riemann surface of genus $g\ge 1$.   
The group $G=\pi_1(\Sigma_g)$ is generated 
by $x_1,y_1,\dots,x_g,y_g$, subject to the single relation 
$[x_1,y_1]\cdots [x_g,y_g]=1$. It is readily checked that $\h(G,\k)$ 
is the quotient of the free Lie algebra on $x_1,y_1,\dots,x_g,y_g$  
by the ideal generated by $[x_1,y_1]+ \cdots + [x_g,y_g]$. 
A further computation using Theorem \ref{thm:Massuyeau} 
shows that $\m(G)\cong \widehat{\h}(G)$; thus, $G$ is $1$-formal. 
\end{example}
 
The $1$-formality property is preserved under finite free 
products and direct products of finitely generated groups.  
The following lemma (which follows at once from the discussion in 
Section \ref{subsec:sullivan-models}) provides a useful $1$-formality 
criterion. 

\begin{lemma}
\label{lem:1fg}
Let $G$ a finitely generated group. Suppose there is a $1$-formal 
group $K$ and a homomorphism $\varphi\colon G\to K$ such that 
$\varphi^*\colon H^1(K,\Q) \to H^1(G,\Q)$ is an isomorphism and 
$\varphi^*\colon H^1(K,\Q) \to H^1(G,\Q)$ is injective. Then $G$ 
is also $1$-formal.
\end{lemma}

\begin{example}
\label{ex:low-betti}
If $G$ is a finitely generated group with $b_1(G)$ equal to $0$ or $1$, 
then $G$ is $1$-formal. Indeed, the claim is true for $K_0=\{1\}$ (trivially) 
and for $K_1=\Z$ (by Example \ref{ex:free}). Moreover, if $b_1(G)=i\in \{0,1\}$, 
we may define a homomorphism $\varphi\colon G\to K_i$ 
satisfying the assumptions of Lemma \ref{lem:1fg}. Therefore, 
the claim holds for $G$, too.
\end{example}

Here is another interpretation of the $1$-formality notion. We say that 
a finitely generated group $G$ is {\em graded-formal}\/ (over 
$\k$) if the associated graded 
Lie algebra $\gr(G;\k)$ is quadratic.  It follows from 
Theorem \ref{thm:holoepi} that $G$ is graded-formal 
precisely when the canonical surjection 
$\Phi\colon \gr(G;\k)\surj  \h(G;\k)$ is an isomorphism. 
As in Section  \ref{subsec:tff}, we say that $G$ is {\em filtered-formal}\/ 
over $\k$ if $\m(G)\otimes \k \cong \widehat{\gr}(G;\k)$. Putting 
things together, we obtain the following result.

\begin{proposition}[\cite{SW-forum}]
\label{prop:qwformal}
A finitely generated group $G$ is $1$-formal (over $\k$) if and only if $G$ is 
graded-formal and filtered-formal (over $\k$).
\end{proposition}

As a corollary, we deduce that $1$-formality enjoys a descent property.

\begin{corollary}[\cite{SW-forum}]
\label{cor:formal-kq}
A finitely generated group $G$ is $1$-formal over $\k$ if and only if $G$ is 
$1$-formal over $\Q$.
\end{corollary}

Indeed, it is easily seen that graded-formality is independent of the choice 
of a field $\k$ of characteristic $0$. By Theorem \ref{thm:ffdescent}, 
the same is true for filtered-formality, and the conclusion follows 
from Proposition \ref{prop:qwformal}.  When $G$ is finitely presented, 
we have that $b_2(G)<\infty$, and the result also follows from 
Corollary \ref{cor:pfext}.

As we saw in Example \ref{ex:free}, the free group $F_n$ 
has vanishing cup-product map $\mu_{F_n}$ and is 
$1$-formal.  Here is a partial converse.

\begin{proposition}[\cite{DPS-duke}]
\label{prop:commrel}
Let $G$ be a group admitting a finite presentation with 
only commutator relators. If $G$ is $1$-formal and 
$\mu_G=0$, then $G$ is a free group. 
\end{proposition}

\begin{example}
\label{ex:heis-bis}
Let $G=G_{\Z}$ be the Heisenberg group from 
Example \ref{ex:heis-massey}. Then $G$ is isomorphic to 
$F_2/\gamma_3(F_2)$, and so it has a finite 
presentation with only commutator relators; moreover, 
$\mu_G=0$, yet  $G$ is not a free group, since it is 
$2$-step nilpotent. Therefore, we conclude once again 
that $G$ is not $1$-formal. 
\end{example}

\subsection{Polyhedral products and right-angled Artin groups}
\label{subsec:zk-raag}
We conclude this section with a discussion of the formality 
properties of polyhedral product spaces and some related 
groups. Given a finite simplicial complex $K$,  it is a 
subtle question to decide whether the polyhedral products 
$\ZZ_{K}(\uX,\uX')$ from Section \ref{subsec:pp} are formal, 
even when all the spaces $X_i$ and the subspaces $X_i'$ 
are formal. Theorem \ref{thm:ft-zk} (together with a previous 
remark) yields a sufficient condition for this to happen.

\begin{corollary}[\cite{FT09}]
\label{cor:formal-zk}
Let $X_i^{}$, $X_i'$ be nilpotent, finite-type CW-complexes.
Assume that the inclusion maps $X_i'\inj X_i$ are formal 
and induce epimorphisms in cohomology.  
Then all polyhedral products $\ZZ_{K}(\uX,\uX')$ are formal.
\end{corollary}

We specialize now to the case when $X_i=X$ and $X_i'=X'$ 
for all $i$, and write $\ZZ_{K}(X,X')$ for the corresponding 
polyhedral product.
If $X$ is nilpotent and formal, then the inclusion 
$* \to X$ satisfies the hypothesis of 
Corollary \ref{cor:formal-zk}, and thus  $\ZZ_K(X,*)$
is formal---a result first proved in \cite{NR}. In particular, 
the Davis--Januszkiewicz spaces 
$DJ_K= \ZZ_K(\CP^{\infty},*)$ and the 
toric complexes $T_K=\ZZ_K(S^1,*)$ are all formal.

Letting $\Gamma$ be the $1$-skeleton of $K$, 
it is readily seen that the fundamental group of 
$T_K$ is the right-angled Artin group $G_{\Gamma}$ 
associated to the graph $\Gamma$. Consequently, 
all right-angled Artin groups are $1$-formal---a result 
also proved in \cite{PS-mathann}, using 
Theorems \ref{thm:Massuyeau} and \ref{thm:1f-group}. 
Moreover, if the flag complex of $\Gamma$ is 
simply-connected, then, as shown in \cite{PS-jlms07}, 
the Bestvina--Brady group associated to $\Gamma$ is 
also $1$-formal.

Finally, let us consider the moment-angle complexes 
$\ZZ_{K}=\ZZ_{K}(D^2,S^1)$. In this situation, 
Corollary \ref{cor:formal-zk} no longer applies, since 
the inclusion-induced homomorphism $H^1(D^2,\Q) \to H^1(S^1,\Q)$ 
is not surjective. In fact, there are infinite families 
of simplicial complexes $K$ for which $H^*(\ZZ_{K},\Q)$ 
has non-vanishing Massey products, and thus $\ZZ_K$ 
is non-formal, see \cite{Baskakov}, \cite{DeS07}, \cite{GL21}. 
If $K$ is an $n$-vertex triangulation of $S^{m}$, 
then $\ZZ_K$ is a closed manifold of dimension $n+m+1$. 
Asymptotically, almost all triangulations $K$ of $S^2$ 
yield non-formal moment-angle manifolds $\ZZ_K$, 
see \cite{DeS07}.

\section{Alexander invariants and resonance varieties}
\label{sect:alex-res}

\subsection{A generalized Koszul complex}
\label{subsec:univ aomoto}
Given a finite-dimensional $\k$-vector space $V$, 
we define the corresponding {\em canonical element}\/ to be tensor 
$\omega_{V} \in V^{\vee} \otimes V$ 
which corresponds to the identity automorphism 
of $V^{\vee}$ under the tensor-hom adjunction 
(recall that $\otimes=\otimes_{\k}$). 
In concrete terms, if we pick a basis $\{e_1,\dots ,e_n\}$ for $V^{\vee}$ 
and let $\{x_1,\dots ,x_n\}$ be the dual basis for $V$, 
then $\omega_V=\sum_{j=1}^n e_j\otimes x_j$.

Now let $A=(A^{*},d)$ be a connected $\k$-$\cdga$, 
and assume that the $\k$-vector space $H^1(A)$ 
is finite-dimensional.
Since $d(1)=0$ and $A^0=\k$, the differential 
$d\colon A^0\to A^1$ vanishes; thus, we may 
identify $H^1(A)$ with $Z^1(A)$.
Setting $H_1(A)=(H^1(A))^{\vee}$, 
we let $\omega_A\coloneqq \omega_{H_1(A)} \in 
H^1(A)\otimes H_1(A)$ 
be the corresponding canonical element.

Let $S=\Sym(H_1(A))$ be the symmetric algebra on $H_1(A)$. 
The tensor product $A\otimes_{\k} S$ is both a free $S$-module and 
a bigraded $\k$-algebra, with product $(a\otimes s)(a'\otimes s')=aa'\otimes ss'$. 
It is also a $\k$-$\cdga$, with differential $d \otimes \id_S$. 
Left-multiplication by $\omega_A$, viewed as an element of $Z^1(A)\otimes H_1(A)$, 
defines an endomorphism of $A\otimes S$ of bidegree $(1,1)$. 
We define an $S$-linear map, $\delta_A\colon A\otimes_{\k} S\to A\otimes S$, by 
\begin{equation}
\label{eq:delta-a}
\delta_A = \omega_A  + d \otimes \id_S .
\end{equation}
It is readily verified that $\delta_A^2=0$, and so the next result follows.

\begin{proposition}[\cite{Su-cdga}]
\label{prop:LA}
Let $(A^*,d)$ be a connected $\k$-$\cdga$ with 
$\dim_\k H^1(A)<\infty$. There is then 
a cochain complex of free $S$-modules,
\begin{equation}
\label{eq:koszul-dga}
\begin{tikzcd}[column sep =20pt]
 \cdots \ar[r] 
& A^{i}\otimes S \ar[r, "\delta^{i}_A"] 
&[4pt] A^{i+1} \otimes S \ar[r, "\delta^{i+1}_A"] 
&[4pt] A^{i+2} \otimes S \ar[r] 
& \cdots ,
\end{tikzcd}
\end{equation}
with differentials given by \eqref{eq:delta-a}, such that 
$(A^*\otimes S,\delta_A)$ is again a $\k$-$\cdga$.
\end{proposition}

If we fix a $\k$-basis $\{ e_1,\dots, e_n \}$ for 
$H^1(A)$ and let $\{ x_1,\dots, x_n \}$ be the dual basis 
for $H_1(A)$, the ring $S=\Sym(H_1(A))$ may 
be identified with the polynomial ring $\k[x_1,\dots, x_n]$,
viewed as the coordinate ring of the affine space $H^1(A)$. 
The differentials in \eqref{eq:delta-a} are then the $S$-linear 
maps given by
\begin{equation}
\label{eq:diff}
\delta^{i}_A(a \otimes s)= 
\sum_{j=1}^{n} e_j a \otimes s x_j + d(a) \otimes s
\end{equation}
for all $a\in A^{i}$ and $s\in S$. 
If the $\cdga$ $A$ has zero differential, each map $\delta^{i}_{A}$ 
is given by a matrix whose entries are linear forms in the 
variables $x_1,\dots ,x_n$; in general, though, the entries of 
$\delta^i_A$ may also have non-zero constant terms.

\subsection{The Alexander invariants of a $\cdga$}
\label{subsec:alexinv}

The $S$-dual of the cochain complex \eqref{eq:koszul-dga} is 
the chain complex of free $S$-modules,
\begin{equation}
\label{eq:ls dual}
\begin{tikzcd}[column sep=22pt]
(A_*\otimes S,\delta^A): \ 
\cdots \ar[r]
& A_2 \otimes S   \ar[r, "\delta_2^A"] 
& A_1 \otimes S   \ar[r, "\delta_1^A"]
& A_0 \otimes S=S ,
\end{tikzcd}
\end{equation}
where the maps $\delta_i^A$ are the $S$-linear duals of the 
maps $\delta^i_A$. By analogy with classical the topological setting, 
we define the {\em Alexander invariants}\/ of a $\cdga$ $(A^*,d)$ as 
the homology $S$-modules of this chain complex, 
\begin{equation}
\label{eq:ai-def}
\B_i(A)\coloneqq H_i (A_*\otimes S) .
\end{equation}

If $d=0$, then the differentials in \eqref{eq:ls dual} are homogeneous (of degree $1$), 
and so the $S$-modules $\B_i(A)$ inherit a natural grading. 
For instance, if $E=\bwedge V$ is the exterior algebra on a finite-dimensional 
$\k$-vector space $V$, with differential $d=0$, then $\fB_i(E)=0$ for all $i\ge 1$. 
In general, though, the Alexander invariants $\B_i(A)$ do not have a natural grading. 

An explicit finite presentation for the first Alexander invariant, $\B(A)\coloneqq \B_1(A)$, 
was given in \cite[Theorem~6.2]{PS-imrn04} in the the case when $d=0$. This 
presentation is generalized in \cite{Su-cdga}, as follows.  

Let $(A,d)$ be a connected $\k$-$\cdga$ with $A^1$ finite-dimensional. 
Set $E=\bwedge H^1(A)$ and identify $E^1=H^1(A)$ with 
$Z^1=\ker\big(d\colon A^1\to A^2\big)$. Let $U^1$ be its 
complementary $\k$-vector subspace, so that $A^1=E^1\oplus U^1$,  
and write $A_i=(A^i)^{\vee}$ and so forth for the $\k$-dual vector spaces. 
Then $U_1$ may be identified with the image of the $\k$-dual of the 
differential, $d^{\vee}\colon A_2\to A_1$, and we have a direct 
sum decomposition, $A_1=E_1\oplus U_1$. Let 
$\pi_U\colon A_1\to U_1$ be the projection onto the 
second summand.

\begin{theorem}[\cite{Su-cdga}]
\label{thm:Bpres}
The Alexander invariant of $A$, viewed as a module over the symmetric 
algebra $S=\Sym(E_1)$, has presentation 
\begin{equation}
\label{eq:Bmod-pres}
\begin{tikzcd}[column sep=18pt]
\big(E_3\oplus A_2\big) \otimes S
\ar[r, "
\sbm{\delta_3^{E}  \amp 0 \\  
\mu_E^{\vee} \otimes\id_S\amp d_A^{\vee} \otimes\id_S+ \beta_A^{\vee}}
"]&[68pt] 
\big( E_2\oplus U_1\big)  \otimes S
\ar[r]& \B(A)  \ar[r]& 0 ,
\end{tikzcd}
\end{equation}
where $\beta_A^{\vee}=(\pi_U\otimes \id_S) \circ (\omega_A - \mu_E\circ \omega_E)^{\vee}$.
\end{theorem}

Finally, let $I$ be the maximal ideal at $0$ of the polynomial ring $S$. 
The powers of this ideal define a descending filtration, $\{I^n\B(A)\}_{n\ge 0}$, 
on the Alexander invariant of $A$. Let $\gr(\B(A))$ be the associated graded 
$S$-module with respect to this filtration. 

\begin{proposition}[\cite{Su-cdga}]
\label{prop:torb}
For each $k\ge 1$, there is an isomorphism of $\k$-vector spaces,
\[
\gr_k(\fB(A))^{\vee}\cong \Tor ^{\cE}_{k-1} (A,\k)_{k} , 
\]
where on the right side $A$ is viewed as a graded module over the exterior 
algebra $\cE=\bigwedge A^1$.  
\end{proposition}

\subsection{Resonance varieties}
\label{subsec:resvar}
Let $(A,d)$ be a connected $\cdga$.  As noted previously, $H^1(A)=Z^1(A)$. 
For every $\omega \in H^1(A)$, the operator $d_{\omega}:= d+ \omega \, \cdot$ 
is a differential on $A$. The {\em resonance varieties}\/ $\RR^i_k(A)$
are defined, for all $i,k\ge 0$, as the infinitesimal jump loci
\begin{equation}
\label{eq:defres}
\RR^i_k(A) =\{ \omega \in H^1(A) \mid 
\dim H^i(A, d_{\omega} )\ge k\} .
\end{equation}
When the $\cdga$ $A$ is $q$-finite, for some $q\ge 1$, these sets 
are Zariski closed subsets of the affine space $H^1(A)$, for all 
$i\le q$ and $k\ge 0$. 

Clearly, $H^i(A, \delta_{0})=H^i(A)$; thus, the point 
$\zero\in H^1(A)$ belongs to the variety $\RR^i_1(A)$ 
if and only if $H^i(A)\ne 0$.  Moreover, $\RR^0_1(A)=\{\zero\}$.  
When the differential of $A$ is zero, the resonance varieties 
$\RR^i_k(A)$ are homogeneous subsets of $H^1(A)=A^1$. 
In general, though, the resonance varieties of a $\cdga$ 
are not homogeneous, as we shall see in Example \ref{ex:nonhomog}.

The following lemma follows quickly from the definitions. 

\begin{lemma}[\cite{MPPS}] 
\label{lem:functoriality}
Let $\varphi\colon A\to A'$ be a $\cdga$ morphism, 
and assume $\varphi$ is an isomorphism up to degree $q$, 
and a monomorphism in degree $q+1$, for some $q\ge 0$.  
Then the induced isomorphism in cohomology, 
$\varphi^*\colon H^1(A')\to H^1(A)$, identifies 
$\RR^i_k(A)$ with $\RR^i_k(A')$ for each $i\le q$,
and sends $\RR^{q+1}_k(A)$ into $\RR^{q+1}_k(A')$, 
for all $k\ge 0$. 
\end{lemma}

Consequently, if $A$ and $A'$ are isomorphic $\cdga$s, then their 
resonance varieties are ambiently isomorphic.  As we shall see  
(also in Example \ref{ex:nonhomog}), the conclusions of 
Lemma \ref{lem:functoriality} do not always hold if we only 
assume that $\varphi\colon A\to A'$ is a $q$-quasi-isomorphism.

An alternative interpretation of the degree $1$ resonance varieties is 
given by the following theorem.

\begin{theorem}[\cite{Su-cdga}]
\label{theorem:r1a}
Let $A$ be a connected $\cdga$ with $0<\dim A^1<\infty$.  Then, 
for all $k\ge 1$,
\begin{equation}
\label{eq:res-inv}
\RR^1_k(A) = 
\bV \big( \!\Ann\big(\bwedge^k \big(\fB(A)\big)\big)\big),
\end{equation}
at least away from $\zero\in H^1(A)$.
\end{theorem}

The next example (adapted from \cite{MPPS} and \cite{Su-tcone}) illustrates several 
of the points mentioned above.

\begin{example}
\label{ex:nonhomog}
Let $A$ be the exterior algebra over $\C$ on generators $a,b$ in 
degree $1$, equipped with the differential given by $d{a}=0$ 
and $d{b}=b\cdot a$. Then $H^1(A)=\C$, generated by $a$. 
Setting $S=\C[x]$, the chain complex \eqref{eq:ls dual} takes the form 
\begin{equation}
\label{eq:toy}
\begin{tikzcd}[column sep=48]
S \ar[r, "\delta_2=\sbm{ 0 \\ x-1}"]
&S^2 \ar[r, "\delta_1=\sbm{x \: \amp 0}"]
&S .
\end{tikzcd}
\end{equation} 
It is readily seen that the Alexander invariant $\B(A)=H_1(A_{*}\otimes S)$
is isomorphic to $S/(x-1)$. Its support is equal to $\{1\}$, yet the resonance 
variety $\RR^1_1(A)$ is equal to $\{0,1\}$; both are non-homogeneous 
subvarieties of $\C$.  Finally, let $A'$ be the sub-$\cdga$ generated by $a$. Clearly, 
the inclusion map, $\iota\colon A'\inj A$, induces an isomorphism 
in cohomology. Nevertheless, $\RR^1_1(A')=\{0\}$, and so the resonance 
varieties of $A$ and $A'$ differ, although $A$ and $A'$ are quasi-isomorphic. 
\end{example}

\subsection{Resonance of tensor products and Hirsch extensions}
\label{subsec:tensor-hirsch}
The resonance varieties behave well with respect to some 
natural operations on $\cdga$s. The next result details the behavior 
of the depth-$1$ resonance varieties with respect to tensor products.

\begin{proposition}[\cite{PS-plms10}, \cite{PS-springer}]
\label{prop:prod-res}
Let $(A,d)$ and $(A',d')$ be two connected, finite-type $\cdga$s. Then, 
for all $q\ge 0$, 
\begin{equation}
\label{eq:rprod}
\RR^i_1(A\otimes A') =\bigcup_{p+q=i} \RR^p_1(A)\times \RR^q_1(A').
\end{equation}
\end{proposition}

A proof of this statement is given in \cite[Proposition~13.1]{PS-plms10} 
under the assumption that both differentials, $d$ and $d'$, vanish (see also 
\cite[Proposition~2]{PS-springer}).  The same approach works in this 
wider generality.

We conclude this section with a result that shows how the 
resonance varieties behave under a certain type 
of Hirsch extensions.

\begin{proposition}[\cite{PS-imrn19}]
\label{prop:hirsch-res}
Let $B$ be a connected, finite-type $\cdga$. Fix an element $e\in B^2$ with $de=0$, and 
let $A=(B\otimes_{e} \bwedge (t), d)$ be the corresponding Hirsch extension. 
\begin{enumerate}
\item  \label{e0}
If $[e]=0$, then $\RR^i_1(A) =  \RR^{i-1}_1(B) \cup \RR^i_1(B)$, for all $i$. 
\item   \label{e1}
If $[e]\ne 0$, then
\begin{enumerate}
\item[(a)] \label{r1}
$\RR^i_k (A)\subseteq \RR^{i-1}_1 (B) \cup \RR^i_k (B)$, for all $i$ and $k$.
\item[(b)] \label{r2}
$\RR^1_k (A)=\RR^1_k (B)$, for all $k$.
\end{enumerate}
\end{enumerate}
\end{proposition}

\section{Cohomology jump loci and finiteness properties}
\label{sect:cjl}

\subsection{Characteristic varieties}
\label{subsec:charvar}

Given a space $X$, the jump loci for cohomology with coefficients in rank $1$ 
complex local systems on $X$ are powerful homotopy-type invariants, 
defined as follows.   

We will assume that $X$ is path-connected and its fundamental 
group, $G=\pi_1(X)$, is finitely generated. Let 
$\T_G\coloneqq \Hom(G, \C^{*})$ be the  group of $\C$-valued 
multiplicative characters on $G$. This is an abelian, complex 
algebraic group, whose identity $\one$ corresponds to the 
trivial representation.  The group $\T_G$ may be identified with 
the cohomology group $\Char(X)\coloneqq  H^1(X,\C^*)$. 
Its identity component, $\T^0_G$, is isomorphic to the complex 
algebraic torus $(\C^*)^{b_1(X)}$; the other connected components 
of $\T_G$ are copies of this torus, indexed by the torsion subgroup 
of the finitely generated abelian group $G_{\ab}=H_1(X,\Z)$.  

The {\em characteristic varieties}\/ of $X$ in degree $i\ge 0$ and 
depth $k\ge 0$ are the sets 
\begin{equation}
\label{eq:defjump}
\VV^i_k(X) =\{ \rho \in \Char(X) \mid 
\dim H^i(X, \C_{\rho} )\ge k\} ,
\end{equation}
where $\C_{\rho}$ is the rank $1$ local system on $X$ associated to 
a representation $\rho\colon G\to \C^{*}$. In other words, $\C_{\rho}$ 
is the vector space $\C$ viewed as a module over the group algebra $\C[G]$ 
via the action $g\cdot a=\rho(g) a$, for $g\in G$ and $a\in \C$. 

When the space $X$ is $q$-finite, for some $q\ge 1$, the sets 
$\VV^i_k(X)$ are Zariski closed subsets of the character group, 
for all $i\le q$ and $k\ge 0$, see \cite{PS-mrl}. It is readily seen that the 
the sets $\VV^1_k(X)$ depend only on the group $G=\pi_1(X)$. 

Now let $G$ be a finitely generated group, and set 
$\VV^i_k(G):=\VV^i_k(K(G,1))$.  It is known that the sets 
$\VV^1_k(G)$ with $k\ge 0$ depend only on 
the maximal metabelian quotient $G/G''$ 
(see e.g.~\cite{DPS-imrn}); 
more precisely, $\VV^1_k(G)=\VV^1_k(G/G'')$.

The characteristic varieties have several useful naturality properties.  
For instance, suppose $\varphi\colon G\surj Q$ is an epimorphism. 
Then the induced morphism on character groups, $\varphi^*\colon 
\T_Q\to \T_G$, is injective and sends $\VV^1_k(Q)$ 
into $\VV^1_k(G)$  for all  $k\ge 0$.  
Likewise, suppose that $H<G$ is a finite-index subgroup.  Then 
the inclusion $\alpha\colon H\to G$ induces a 
morphism $\alpha^*\colon \T_G\to \T_H$ with finite 
kernel, which sends $\VV^i_k(G)$ to $\VV^i_k(H)$ for all  $i,k\ge 0$.  

For the free groups $F_n$ of rank $n\ge 2$, we have that 
$\VV^1_k(F_n)=(\C^{*})^n$ for $k\le n-1$ and $\VV^1_n(F_n)=\{1\}$. 
In general, though, the jump loci of a group can be arbitrarily complicated.  

\begin{example}
\label{ex:syz}
Let $f\in \Z[t_1^{\pm 1},\dots , t_n^{\pm 1}]$ 
be an integral Laurent polynomial with $f(1)=0$.   
Then, as shown in \cite{SYZ}, there is a finitely presented group $G$ 
with $G_{\ab}=\Z^n$ such that $\VV^1_1(G)$ coincides with 
the variety $\mathbf{V}(f):=\{t \in (\C^{*})^n \mid f(t)=0\}$.  
\end{example}

\subsection{Algebraic models and cohomology jump loci}
\label{subsec:mod-cjl}

Work of Dimca and Papadima \cite{DP-ccm}, 
generalizing previous work of Dimca, Papadima, and Suciu \cite{DPS-duke},
establishes a tight connection between the geometry of the characteristic 
varieties of a space and that of resonance varieties of 
a model for it, around the origins of the respective ambient spaces, 
provided certain finiteness conditions hold.

Let $X$ be a path-connected space with $b_1(X)<\infty$, 
and consider the analytic map $\exp\colon H^1(X,\C)\to H^1(X,\C^*)$ 
induced by the coefficient homomorphism $\C\to \C^*$, $z\mapsto e^z$.
Let $(A,d)$ be a $\cdga$ model for $X$, defined over $\C$.  
Upon identifying $H^1(A)\cong H^1(X,\C)$, we obtain an analytic map 
$H^1(A)\to H^1(X,\C^*)$, which takes $\zero$ to $\one$.   

\begin{theorem}[\cite{DP-ccm}]
\label{thm:thmb}
Let $X$ be a $q$-finite space, and suppose $X$ admits a $q$-finite, 
$q$-model $A$, for some $q\ge 1$. Then, the aforementioned map,  
$H^1(A)\to H^1(X,\C^*)$, induces a local analytic isomorphism,  
$H^1(A)_{(\zero)} \to H^1(X,\C^*)_{(\one)}$, 
which identifies the germ at $\zero$ of $\RR^i_k(A)$ 
with the germ at $\one$ of $\VV^i_k(X)$, for all $i\le q$ and all $k\ge 0$. 
\end{theorem}

The work of Budur and Wang \cite{BW20} builds on this theorem, 
providing a structural result on the geometry of the characteristic varieties 
of spaces satisfying the hypothesis of the above theorem.   
Putting together Theorem \ref{thm:thmb} and Corollary \ref{cor:ax-bis} 
yields their result, in the slightly stronger form given in \cite{PS-jlms}.

\begin{theorem}[\cite{BW20}]
\label{thm:bw}
Suppose $X$ is a $q$-finite space which admits a $q$-finite $q$-model. 
Then all the irreducible components of 
$\VV^i_k(X)$ passing through $\one$ are algebraic subtori  of 
$\Char(X)$, for all $i\le q$ and $k\ge 0$.
\end{theorem}

\subsection{Finiteness obstructions}
\label{subsec:infvsBW}

The above theorem may be used to give examples 
of finite CW-complexes which do not have $1$-finite $1$-models.

\begin{example}
\label{ex:fpresgrp}
Let $f$ be an integral Laurent polynomial in $n\ge 2$ variables, and 
assume its zero set in $(\C^{*})^n$ contains the origin $1$, is irreducible 
but is not an algebraic subtorus; for instance, take $f(t)=\sum_{i=1}^n t_i -n$.  
Letting $G$ be a finitely presented group with $\VV^1_1(G)=\mathbf{V}(f)$ 
as in Example \ref{ex:syz}, we deduce from Theorem \ref{thm:bw} 
that the finite presentation complex of $G$ admits no $1$-finite $1$-model. 
\end{example}

On the other hand, as the next example shows, the existence of a 
$1$-finite $1$-model for a finitely generated group does not 
necessarily imply that the group is finitely presented.

\begin{example}
\label{ex:bbgrps}
Let $Y$ be a finite, connected CW-complex which is non-simply connected yet 
has $b_1(Y)=0$, and let $G$ be the Bestvina--Brady group associated to a flag 
triangulation of $Y$. It is proved in \cite[\S 10]{PS-adv09} that $G$ is finitely generated
and $1$-formal, but not finitely presented. 
\end{example}

As the next family of examples illustrates, the infinitesimal finiteness 
obstruction from Theorem \ref{thm:infobsq} may be stronger than 
the one from Theorem \ref{thm:bw}, even when $q=1$. 

\begin{example}
\label{ex:subtler}
Consider the free metabelian group $G= F_n/F_n''$ with $n\ge 2$. 
The free group $F_n=\pi_1(\bigvee^n S^1)$ admits a formal, finite CW-complex 
as classifying space; thus, Theorem \ref{thm:bw} applies to $F_n$.
It follows that the characteristic varieties $\VV^i_k(G)\cong \VV^i_k(F_n)$ satisfy 
the conditions from Theorem \ref{thm:bw} for $i\le 1$ and $k\ge 0$. 
On the other hand, as we saw in the proof of Theorem \ref{thm:meta}, 
we have that $b_2(\M_1(G))= \infty$, and so the group $G$ admits 
no $1$-finite $1$-model. 
\end{example}

\subsection{Tangent cones}
\label{subsec:exp tc}

Before proceeding, we review two constructions that provide 
approximations to a subvariety $W$ of a complex algebraic 
torus $(\C^*)^n$. The first one is the classical tangent cone, 
while the second one is the exponential tangent cone, 
a construction introduced in \cite{DPS-duke} and further 
studied in \cite{Su-imrn}, \cite{DP-ccm}, and \cite{SYZ}.

Let $I$ be an ideal in the Laurent polynomial ring   
$\C[t_1^{\pm 1},\dots , t_n^{\pm 1}]$ such that $W=V(I)$.   
Picking a finite generating set for $I$, and multiplying 
these generators with suitable monomials if necessary, 
we see that $W$ may also be defined by the ideal $I\cap R$ 
in the polynomial ring $R=\C[t_1,\dots,t_n]$.  
Let $J$ be the ideal  in the polynomial ring 
$S=\C[x_1,\dots, x_n]$ generated by the polynomials 
$g(x_1,\dots, x_n)=f(x_1+1, \dots , x_n+1)$, 
for all $f\in I\cap R$. 

The {\em tangent cone}\/ of $W$ at $\one \in (\C^*)^n$ is the algebraic 
subset $\TC_{\one}(W)\subseteq \C^n$ defined by the ideal 
$\init(J)\subset S$ generated by the initial forms of 
all non-zero elements from $J$.  The set 
$\TC_{\one}(W)$ is a homogeneous subvariety of $\C^n$, 
which depends only on the analytic germ of $W$ at 
the identity.  In particular, $\TC_{\one}(W)\ne \emptyset$ 
if and only if $\one\in W$.  

Let $\exp\colon \C^n \to (\C^*)^n$ be the exponential map, 
given in coordinates by $x_i\mapsto e^{x_i}$.  
The {\em exponential tangent cone}\/ at $\one$ 
to a subvariety $W\subseteq (\C^*)^n$ is the set  
\begin{equation}
\label{eq:tau1}
\tau_{\one}(W)= \{ x\in \C^n \mid \exp(\lambda x)\in W,\ 
\text{for all $\lambda \in \C$} \}.
\end{equation}
It is readily seen that $\tau_{\one}$ commutes with finite unions and 
arbitrary intersections. Furthermore, $\tau_{\one}(W)$ only depends 
on $W_{(\one)}$, the analytic germ of $W$ at the identity; in particular, 
$\tau_{\one}(W)\ne \emptyset$ if and only if $\one\in W$.  The main 
property of this construction is encapsulated in the following lemma.

\begin{lemma}[\cite{DPS-duke}, \cite{Su-imrn}, \cite{SYZ}] 
\label{lem:exp-tcone}
The exponential tangent cone $\tau_{\one}(W)$  of a subvariety 
$W\subseteq (\C^*)^n$ is a finite union of rationally defined linear 
subspaces of the affine space $\C^n$.  
\end{lemma}
 
For instance, if $W$ is an algebraic subtorus of $(\C^{*})^n$,  
then $\tau_{\one}(W)$ equals $\TC_{\one}(W)$, and both 
coincide with $T_{\one}(W)$, 
the tangent space to $W$ at the identity $\one$. 
More generally, there is always an inclusion between the two 
types of tangent cones associated to an  algebraic subset 
$W\subseteq (\C^{*})^n$, namely, 
\begin{equation}
\label{eq:ttinc}
\tau_{\one}(W)\subseteq \TC_{\one}(W). 
\end{equation}

As we shall see, though, this inclusion is far from being an equality 
for arbitrary $W$. For instance, the tangent cone $\TC_{\one}(W)$ may 
be a non-linear, irreducible subvariety of $\C^n$, or $\TC_{\one}(W)$ 
may be a linear space containing the exponential tangent cone 
$\tau_{\one}(W)$ as a union of proper linear subspaces.

\subsection{The Exponential Ax--Lindemann theorem}
\label{subsec:ax}

In \cite{BW20}, Budur and Wang establish the following version 
of a classical result, due to Ax and Lindemann.

\begin{theorem}[\cite{BW20}]
\label{thm:ax}
Let $V\subseteq \C^n$ and $W\subseteq (\C^*)^n$ be 
irreducible algebraic subvarieties.
\begin{enumerate}
\item \label{bw1}
Suppose $\dim V=\dim W$ and $\exp(V)\subseteq W$.  Then $V$ is a translate 
of a linear subspace, and $W$ is a translate of an algebraic subtorus. 
\item \label{bw2}
Suppose the exponential map $\exp\colon \C^n \to (\C^*)^n$ induces a 
local analytic isomorphism  $V_{(\zero)} \to W_{(\one)}$. Then  
$W_{(\one)}$ is the germ of an algebraic subtorus. 
\end{enumerate}
\end{theorem}

A standard dimension argument shows the following: 
if $W$ and $W'$ are irreducible algebraic subvarieties of 
$(\C^*)^n$ which contain $\one$ and whose germs at 
$\one$ are locally analytically isomorphic, then $W\cong W'$.  
Using this fact, we obtain the following corollary to part \eqref{bw2} of 
the above theorem.

\begin{corollary}
\label{cor:ax-bis}
Let $V\subseteq \C^n$ and $W\subseteq (\C^*)^n$ be irreducible 
algebraic subvarieties.  Suppose the exponential map 
$\exp\colon \C^n \to (\C^*)^n$ induces a local analytic 
isomorphism $V_{(\zero)} \cong W_{(\one)}$. Then $W$ is 
an algebraic subtorus and $V$ is a rationally defined 
linear subspace.
\end{corollary}

\subsection{Tangent cones and jump loci}
\label{subsec:tcone thm}

Let $X$ be a $q$-finite space. Its cohomology algebra, 
$H^{*}(X,\C)$, is then $q$-finite; that is, $b_i(X)<\infty$ for $i\le q$. 
Thus, the resonance varieties $\RR^i_k(X):=\RR^i_k(H^{*}(X,\C))$ 
are homogeneous algebraic subsets of the affine space $H^{1}(X,\C)$, 
for all $i\le q$ and $k\ge 0$.  

The following basic relationship between the characteristic 
and resonance varieties was established by Libgober in \cite{Li02} in 
the case when $X$ is a finite CW-complex and $i$ is arbitrary; a similar 
proof works in the generality that we work in here.

\begin{theorem}[\cite{Li02}]
\label{thm:lib}
Suppose $X$ is a $q$-finite space.  Then, for all $i\le q$ and $k\ge 0$, 
\begin{equation}
\label{eq:tc lib}
\TC_{\one}(\VV^i_k(X))\subseteq \RR^i_k(X).
\end{equation}
\end{theorem}

Putting together these inclusions with those from \eqref{eq:ttinc}, 
we obtain the following corollary. 
\begin{corollary}
\label{cor:tcone inc}
Suppose $X$ is a $q$-finite space.  Then, for all $i\le q$ and $k\ge 0$, 
\begin{equation}
\label{eq:tc inc}
\tau_{\one}(\VV^i_k(X))\subseteq  
\TC_{\one}(\VV^i_k(X))\subseteq \RR^i_k(X).
\end{equation}
\end{corollary}

A particular case of this corollary is worth mentioning separately.

\begin{corollary}
\label{cor:tcone gps}
Let $G$ be a finitely generated group.  Then, for all $k\ge 0$, 
\[
\tau_{\one}(\VV^1_k(G))\subseteq  
\TC_{\one}(\VV^1_k(G))\subseteq \RR^1_k(G).
\]
\end{corollary}

Using now Theorems \ref{thm:thmb} and \ref{thm:bw}, we 
obtain the following ``Tangent Cone formula.''

\begin{theorem}
\label{thm:tcone-fm}
Suppose $X$ is a $q$-finite space which admits a $q$-finite $q$-model $A$.  
Then, for all $i\le q$ and 
$k\ge 0$,
\begin{equation}
\label{eq:tc-fm}
\tau_{\one}(\VV^i_k(X))=\TC_{\one}(\VV^i_k(X))=\RR^i_k(A).
\end{equation}
\end{theorem}

This theorem, together with Theorem \ref{thm:psobs}, yields the 
following corollary.

\begin{corollary}
\label{cor:malcone}
Suppose $G$ is a finitely generated group whose Malcev Lie algebra   
is the LCS completion of a finitely presented Lie algebra. Then 
$\tau_{\one}(\VV^1_k(G))=\TC_{\one}(\VV^1_k(G))$, for all $k\ge 0$.
\end{corollary}

In other words, if the first half of the Tangent Cone formula fails in degree $1$, 
i.e., if $\tau_{\one}(\VV^1_k(G))\subsetneqq \TC_{\one}(\VV^1_k(G))$ for 
some $k> 0$, then $\m(G)\not\cong \widehat{L}$, for any finitely presented 
Lie algebra $L$.  This will happen automatically if the variety 
$\TC_{\one}(\VV^1_k(G))$ has an irreducible component which is not 
a rationally defined linear subspace of $H^1(G,\C)$.

\subsection{Formality and cohomology jump loci}
\label{subsec:formal-cjl}

The main connection between the formality property 
of a space and the geometry of its cohomology jump 
loci is provided by the next result. This result, which 
was first proved in degree $i=1$ in \cite{DPS-duke}, 
and in arbitrary degree in \cite{DP-ccm}, is now an 
immediate consequence of Theorem~\ref{thm:tcone-fm}.
 
\begin{corollary}
\label{cor:tcone}
Let $X$ be a $q$-finite, $q$-formal space. Then, for all $i\le q$ and 
$k\ge 0$,
\begin{equation}
\label{eq:tc}
\tau_{\one}(\VV^i_k(X))=\TC_{\one}(\VV^i_k(X))=\RR^i_k(X).
\end{equation}
\end{corollary}

In particular, if $G$ is a finitely generated, $1$-formal group, then, for all $k\ge 0$, 
\begin{equation}
\label{eq:tc-pi}
\tau_{\one}(\VV^1_k(G))= 
\TC_{\one}(\VV^1_k(G))=\RR^1_k(G).
\end{equation}

As an application of Corollary \ref{cor:tcone}, we have the following characterization 
of the irreducible components of the cohomology jump loci in the formal setting.

\begin{corollary}
\label{cor:rational}
Suppose $X$ is a $q$-finite, $q$-formal space.  Then, for 
all $i\le q$ and $k\ge 0$, the following hold.
\begin{enumerate}
\item \label{tc1}
All irreducible components of the resonance varieties  
$\RR^i_k(X)$ are rationally defined linear subspaces 
of $H^1(X,\C)$. 
\item \label{tc2}
All irreducible components of the characteristic varieties  
$\VV^i_k(X)$ which contain the origin are algebraic 
subtori of $\Char(X)^{0}$, of the form $\exp(L)$, where 
$L$ runs through the linear subspaces comprising $\RR^i_k(X)$.
\end{enumerate}
\end{corollary}

\section{Algebraic models for smooth quasi-projective varieties}
\label{sect:models-qp}

\subsection{Compactifications and formality}
\label{subsec:qproj}

A complex projective variety is a subset of a complex projective 
space $\CP^n$, defined as the zero-locus of a homogeneous prime ideal in 
$\C[z_0,\dots, z_n]$.  A Zariski open subvariety of a projective 
variety is called a quasi-projective variety.  We will only consider here 
projective and quasi-projective varieties which are connected and smooth. 

If $M$ is a smooth, projective variety---or, more generally, a compact 
K\"ahler manifold---then the Hodge decomposition on the cohomology 
ring $H^*(M,\C)$ imposes stringent constraints on the topological properties 
of $M$.  For instance, in the famous paper of Deligne, Griffiths, Morgan, 
and Sullivan \cite{DGMS} it is shown that every such manifold is formal. 
  
Each smooth, quasi-projective variety $X$ admits a good compactification.  
That is to say, there is a smooth, complex projective variety $\overline{X}$  and a 
normal-crossings divisor $D$ such that $X=\overline{X}\setminus D$. 
By a well-known theorem of Deligne, each cohomology group of 
$X$ admits a mixed Hodge structure.   This additional structure puts 
definite constraints on the algebraic topology of such manifolds.   

For instance, if $X$ admits a smooth compactification 
$\overline{X}$ with $b_1(\overline{X})=0$, the weight 
$1$ filtration on $H^1(X,\C)$ vanishes; in turn, by work 
of Morgan \cite{Mo}, this implies the $1$-formality of $X$.   
Thus, as noted by Kohno in \cite{Kohno}, if $X$ is the complement 
of a hypersurface in $\CP^n$, then $\pi_1(X)$ is $1$-formal.  

In general, though, smooth quasi-projective varieties need not 
be $1$-formal.  Moreover, even when they are $1$-formal, 
they still can be non-formal. 

\begin{example}
\label{ex:qp1f}
Let $E^{\times n}$ be the product of $n$ copies of an elliptic curve $E$.  
The closed form $\frac{1}{2} \sqrt{-1} \sum_{i=1}^n dz_i \wedge d\bar{z}_i$ 
defines an integral cohomology class $\omega\in H^{1,1}(E^{\times n},\Z)$.  
By the Lefschetz theorem on $(1,1)$-classes, 
$\omega$ can be realized as the first Chern class of an algebraic 
line bundle over $E^{\times n}$.  Let $X_n$ be the complement 
of the zero-section of this bundle. Then $X_n$ is a smooth, 
quasi-projective variety which is not formal.  In fact, $X_n$ 
deform-retracts onto the $(2n+1)$-dimensional Heisenberg 
nilmanifold $\Heis_n$ from Example \ref{ex:heis-sasaki}, and 
so $X_n$ is $(n-1)$-formal but not $n$-formal. 
\end{example}

\subsection{Algebraic models}
\label{subsec:gysin}

As before, let $X$ be a connected, smooth, complex quasi-projective 
variety, and  choose a smooth compactification  $\overline{X}$ such that 
the complement is a finite union, $D=\bigcup_{j\in J} D_j$, 
of smooth divisors with normal crossings.
There is then a rationally defined $\cdga$, 
$A=A(\overline{X}, D)$, called the {\em Gysin model}\/ (or, the 
{\em Morgan model}) of the compactification, constructed as follows.
As a $\C$-vector space, $A^i$ is the direct sum 
of all subspaces 
\begin{equation}
\label{eq:gysin}
A^{p,q}= \bigoplus_{\abs{S} =q} 
H^p \Big(\bigcap_{k\in S} D_k, \C\Big)(-q)
\end{equation}
with $p+q=i$, where $(-q)$ denotes the Tate twist. 
Furthermore, the multiplication in $A$ 
is induced by the cup-product in $\overline{X}$, and has the property that 
$A^{p,q} \cdot A^{p',q'} \subseteq A^{p+p',q+q'}$, while the differential, 
$d \colon A^{p,q} \to A^{p+2, q-1}$, is constructed from the 
Gysin maps arising from intersections of divisors. 
The $\cdga$ just constructed depends on the compactification 
$\overline{X}$; for simplicity, though, we will denote it by $A(X)$ 
when the compactification is understood. 

An important particular case 
is when our variety $X$ has dimension $1$. That is to say, let $\Sigma$ 
be a connected, possibly non-compact, smooth algebraic curve.  
Then $\Sigma$ admits a canonical compactification,  
$\overline{\Sigma}$, and thus, a canonical Gysin model, 
$A (\Sigma)$. We illustrate the construction of this model 
in a simple situation, using the very explicit description 
given by Bibby in \cite{Bi16} for complements of elliptic 
arrangements.

\begin{example}
\label{ex:elliptic gysin}
Let $\Sigma=E^{*}$ be a once-punctured 
elliptic curve.  Then $\overline{\Sigma}=E$, and the Gysin 
model $A (\Sigma)$ is the algebra $A=\bigwedge(a,b,e)/(ae, be)$ 
on generators $a,b$ in bidegree $(1,0)$ and 
generator $e$ in bidegree $(0,1)$, with differential 
$d\colon A\to A$ given by $d{a}=d{b}=0$ and $d{e}=ab$. 
\end{example}

The above construction is functorial, in the following sense: 
If $f\colon X\to Y$ is a morphism of quasi-projective manifolds 
which extends to a regular map $\bar{f}\colon \overline{X}\to \overline{Y}$ 
between the respective good compactifications, then there is an induced 
$\cdga$ morphism $f^{!}\colon A(Y)\to A(X)$ which respects the bigradings. 

Morgan showed in \cite{Mo} that the Sullivan model $\apl (X)$ 
is connected to the Gysin model $A(X)$ by a chain  
of quasi-isomorphisms preserving $\Q$-structures. 
Moreover, setting the weight of $A^{p,q}$ equal to $p+2q$ defines a 
positive-weight decomposition on $(A^{*}, d)$. 

In \cite{Dp15}, Dupont constructed a Gysin-type model for certain types 
of quasi-projective varieties, where the normal-crossings divisors 
hypothesis on the compactification can be relaxed. More precisely, 
let $\A$ be an arrangement of smooth hypersurfaces in a smooth, 
$n$-dimensional complex projective variety $\overline{X}$, 
and suppose $\A$ locally looks like an arrangement of 
hyperplanes in $\C^n$.  There is then a $\cdga$ model 
for the complement, $X=\overline{X}\setminus \bigcup_{L\in \A} L$, 
which builds on the combinatorial definition of the 
Orlik--Solomon algebra of a hyperplane arrangement.

Finally, let $\A$ be an arrangement of complex linear 
subspaces in $\C^n$.  Using a blow-up construction, De~Concini 
and Procesi gave in \cite{DP95} a `wonderful' $\cdga$ model for 
the complement of such an arrangement. Based on a simplication 
of this model due to Yuzvinsky \cite{Yu}, Feichtner and Yuzvinsky 
showed in \cite{FeYu} the following:  If the intersection poset of $\A$ 
is a geometric lattice, then the complement of $\A$ is a formal space. 
In general, though, the complement of a complex subspace arrangement 
need not be formal.  For instance, the polyhedral product constructions 
of \cite{Baskakov}, \cite{DeS07}, \cite{GL21} mentioned in 
Section \ref{subsec:zk-raag} yield coordinate subspace 
arrangements whose complements admit non-trivial Massey 
products over the rationals. 

\subsection{Characteristic varieties}
\label{subsec:cv qproj} 

The structure of the jump loci for cohomology in rank $1$ local systems 
on smooth, complex projective and quasi-projective varieties (and, more 
generally, on K\"{a}hler and quasi-K\"{a}hler manifolds) was determined 
through the work of Beauville \cite{Be92},  Green and 
Lazarsfeld \cite{GL87}, Simpson \cite{Si93}, and Arapura \cite{Ar}.  
The definitive structural result in the quasi-projective setting 
was obtained by Budur and Wang in \cite{BW-asens}, 
building on the work of Dimca and Papadima \cite{DP-ccm}. 

\begin{theorem}[\cite{BW-asens}]
\label{thm:bw-asens}
Let $X$ be a smooth quasi-projective variety.  Then each 
characteristic variety $\VV^i_k(X)$ is a finite union of torsion-translated 
subtori of $\Char(X)$. 
\end{theorem}

Work of Arapura \cite{Ar} explains how the  
non-translated subtori occurring in the above decomposition of 
$\VV^1_1(X)$ arise.  Let us say that a holomorphic map $f \colon X \to \Sigma$ 
is {\em admissible}\/ if $f$ is surjective, has connected generic fiber, 
and the target $\Sigma$ is a connected, smooth complex curve 
with negative Euler characteristic.  Up to reparametri\-zation at the target, 
the variety $X$ admits only finitely many admissible maps; let 
$\mathcal{E}_X$ be the set of equivalence classes of such maps. 

If $f \colon X \to \Sigma$ is an admissible map, 
it is readily verified that $\VV^1_1(\Sigma)=\Char(\Sigma)$. 
Thus, the image of the induced morphism between character groups, 
$f^*\colon \Char(\Sigma)\to \Char(X)$, is an algebraic subtorus 
of  $\Char(X)$.

\begin{theorem}[\cite{Ar}]
\label{thm:arapura}
The correspondence $f \mapsto f^*(\Char(\Sigma))$ defines 
a bijection between the set $\mathcal{E}_X$ of equivalence classes 
of admissible maps from $X$ to curves and the set of 
positive-dimensional, irreducible components of $\VV^1(X)$ 
containing $\one$.
\end{theorem}

The positive-dimensional, irreducible components of $\VV^1_1(X)$ 
which do not pass through $\one$ can be similarly described, by replacing 
the admissible maps with certain ``orbifold fibrations,'' whereby 
multiple fibers are allowed. 

\subsection{Resonance varieties}
\label{subsec:res qproj} 

We now turn to the resonance varieties associated with a 
quasi-projective manifold, and how they relate to the 
characteristic varieties.  The Tangent Cone theorem 
takes a very special form in this setting. 

\begin{theorem}
\label{thm:tcone qp}
Let $X$ be a smooth, quasi-projective variety, and let $A(X)$ be a 
Gysin model for $X$.  Then, for each $i\ge 0$ and $k\ge 0$, 
\begin{equation}
\label{eq:tcone qp}
\tau_{\one}(\VV^i_k(X)) =  \TC_{\one}(\VV^i_k(X)) = \RR^i_k(A(X)) \subseteq  \RR^i_k(X).
\end{equation}
Moreover, if $X$ is $q$-formal, the last inclusion is an equality, 
for all $i\le q$.
\end{theorem}

In particular, the resonance varieties $\RR^i_k(A(X))$ are finite unions of 
rationally defined linear subspaces of $H^1(X,\C)$.  On the other hand, 
the varieties $\RR^i_k(X)$ can be much more complicated; for instance, 
they may have non-linear irreducible components. If $X$ is $q$-formal, 
though, Theorem \ref{eq:tcone qp} guarantees this cannot happen, as 
long as $i\le q$. 

\subsection{Resonance in degree $1$}
\label{subsec:res deg1} 

Once again, let $X$ be a smooth, quasi-projective variety, and 
let $A(X)$ be the Gysin model associated with a good compactification 
$\overline{X}$.  The degree $1$ resonance varieties $\RR^1_1(A(X))$, 
and, to some extent, $\RR^1_1(X)$, admit a much more precise 
description than those in higher degrees.  

As in the setup from Theorem \ref{thm:arapura}, let $\mathcal{E}_X$ 
be the set of equivalence classes of admissible maps from $X$ to 
curves, and let $f\colon X\to \Sigma$ be such map.  Recall 
from Section \ref{subsec:gysin} that the curve $\Sigma$ 
admits a canonical Gysin model, $A (\Sigma)$.  
As noted in  \cite{DP-ccm}, the induced 
$\cdga$ morphism, $f^{!} \colon A(\Sigma) \to A (X)$, 
is injective. Let $f^*\colon H^1(A(\Sigma)) \to H^1(A (X))$ 
be the induced homomorphism in cohomology.

\begin{theorem}[\cite{DP-ccm, MPPS}]
\label{thm:r1 gysin}
For a smooth, quasi-projective variety $X$, 
the decomposition of $\RR^1_1(A(X))$ into (linear) irreducible components 
is given by
\begin{equation}
\label{eq:pencils}
\RR^1_1(A(X))= \bigcup_{f\in \mathcal{E}_X} f^*(H^1(A(\Sigma))).
\end{equation}
\end{theorem}

If $X$ admits no admissible maps, that is, if $\mathcal{E}_X=\emptyset$, 
formula \eqref{eq:pencils} should be understood to mean $\RR^1_1(A(X))= \{\zero\}$ 
if $b_1(X)>0$ and $\RR^1_1(A(X))= \emptyset$ if $b_1(X)=0$.

\begin{example}
\label{ex:heis-complex}
Let $X=X_1$  be the complex, smooth quasi-projective surface 
constructed in Example \ref{ex:qp1f}.  Clearly, this manifold is a 
$\C^*$-bundle over $E=S^1\times S^1$ which deform-retracts 
onto the $3$-dimensional Heisenberg nilmanifold $M=G_{\R}/G_{\Z}$ 
from Example \ref{ex:heis-massey}.  Hence, $\VV^1_1(X)=\{\one\}$, and so 
$\tau_{\one}(\VV^1_1(X))=\TC_{\one}(\VV^1_1(X))=\{\zero\}$.  
On the other hand, $\RR^1_1(X)=\C^2$, and so $X$ is 
not $1$-formal.
\end{example}

Under a $1$-formality assumption, the usual resonance varieties $\RR^1_1(X)$ 
admit a similar description. 

\begin{theorem}[\cite{DPS-duke}]
\label{thm:res kahler} 
Let $X$ be a smooth, quasi-projective variety, and suppose $X$ is $1$-formal. 
Then the decomposition into irreducible components of the first resonance 
variety is given by
\begin{equation}
\label{eq:r1dec}
\RR^1_1(X)= \bigcup_{f\in \mathcal{E}_X} f^*(H^1(\Sigma,\C)), 
\end{equation}
with the same convention as before when $\mathcal{E}_X = \emptyset$.  
Moreover, all the (rationally defined) linear subspaces in this  
decomposition have dimension at least $2$, and any two 
distinct ones intersect only at $\zero$. 
\end{theorem}

If $X$ is compact, then the formality assumption 
in the above theorem is automatically satisfied, due to \cite{DGMS}.  
Furthermore, the conclusion of the theorem can also 
be sharpened in this case: each (non-trivial) irreducible component 
of $\RR^1_1(X)$ is even-dimensional, of dimension at least $4$.   
In general, though, the resonance varieties of a quasi-projective 
manifold can have non-linear components.  

\begin{example}[\cite{DPS-duke}]
\label{ex:res-config}
Let $X=\Conf(E,n)$ be the configuration space of $n$ 
points on an elliptic curve $E$.  Letting $\{ a, b\}$ be the 
standard basis of $H^1(E,\C)=\C^2$, we may identify 
$H^*( E^{\times n}, \C)$ with $\bigwedge(a_1, b_1,\dots , a_n, b_n)$ 
and find a presentation for $H^{\le 2}(X,\C)$ from Totaro's spectral sequence 
\cite{To96}. A computation then gives
\begin{equation} 
\label{eq:resg1}
\RR^1_1(\Conf(E,n))=\left\{ (x,y) \in \C^n\times \C^n \left|
\begin{array}{l}
\sum_{i=1}^n x_i=\sum_{i=1}^n y_i=0 ,\\[2pt]
x_i y_j-x_j y_i=0,  \text{ for $1\le i<j< n$}
\end{array}
\right\}. \right.
\end{equation}
If $n\ge 3$, this variety is irreducible and non-linear (in fact, it is a 
rational normal scroll), from which we conclude that the configuration 
space $\Conf(E,n)$ is not $1$-formal.
\end{example}

\subsection{Large quasi-projective groups}
\label{subsec:qplarge}

Recall that a quasi-projective variety is a Zariski open subset of 
a projective variety.  We will say that a space $X$ is a 
{\em quasi-projective manifold}\/ if it is a connected, 
smooth, complex quasi-projective variety. Every such 
manifold has the homotopy type of a finite CW-complex.

A group $G$ is said to be {\em quasi-projective}\/ if it can be 
realized as the fundamental group of a quasi-projective manifold. 
Clearly, every such a group admits a finite presentation. 
 We now turn to the question of deciding whether 
a quasi-projective group is large.  It turns out that a complete answer 
to this question can be given in terms of ``admissible'' maps to curves.

A map $f\colon X \to C$ from a quasi-projective manifold $X$ to a 
smooth complex curve $C$ is said to be {\em admissible}\/ if it is 
regular, surjective, and has connected generic fiber. 
It is easy to see that the homomorphism on 
fundamental groups induced by such a map, 
$f_{\sharp}\colon \pi_1(X)\to \pi_1(C)$, is surjective. 
We denote by $\cE(X)$ the family of admissible 
maps to curves with negative Euler characteristic, 
modulo automorphisms of the target.  

Deep work of Arapura \cite{Ar} characterizes   
those positive-dimensional, irreducible components of the characteristic 
variety $\VV^1_1(X)$ which contain the origin of the character 
group $\Char(X)$: all such components are connected, 
affine subtori, which arise by pullback of the character torus 
$\Char(C)$ along the homomorphism 
$f_{\sharp}\colon \pi_1(X)\to \pi_1(C)$ induced 
by some map $f\in \cE(X)$. 

Suppose now that $C$ is a smooth complex curve with $\chi(C)<0$.  It is 
readily seen that the fundamental group $G=\pi_1(C)$ surjects onto 
a free, non-abelian group, and so $G$ is very large. More 
generally, we have the following characterization of large, 
quasi-projective groups.

\begin{proposition}[\cite{PS-jlms}]
\label{prop:qproj}
Let $X$ be a smooth quasi-projective variety.  Then:
\begin{enumerate}
\item  \label{lg1}
$\pi_1(X)$ is large if and only if there is a finite cover $Y\to X$ such that 
$\mathcal{E}(Y)\ne \emptyset$. 
\item \label{lg2}
$\pi_1(X)$ is very large if and only if $\mathcal{E}(X)\ne \emptyset$. 
\end{enumerate}
\end{proposition}

Consequently, if $b_1(X)>0$, 
then $\mathcal{E}(X)\ne \emptyset$ if and only if the analytic germ 
at $\one$ of $\VV^1_1(X)$ is not equal to $\{ \one\}$. 

\subsection{Resonance and largeness}
\label{subsec:reslarge}
To conclude this section, we rephrase the last condition in terms of resonance varieties.
As shown by Morgan \cite{Mo}, every quasi-projective manifold $X$ 
admits a finite-dimensional model $A(\oX,D)$; such a `Gysin' model depends on 
a smooth compactification $\oX$ for which the complement 
$D=\oX\setminus X$ is a normal crossings divisor.  
Let $A$ be a Gysin model for $X$, or any one of the more general Orlik--Solomon 
models constructed by Dupont in \cite{Dp16}. In either case, let us note that all 
resonance varieties of $A$ have {\em positive weights}, i.e., they are invariant 
with respect to a $\C^{*}$-action on $H^1(A)$ with positive weights.

\begin{proposition}[\cite{PS-jlms}]
\label{prop:qprojtest}
Let $X$ be a smooth, quasi-projective variety with $b_1(X)>0$ and let $A$ be an 
Orlik--Solomon model for $X$.  Then  $\pi_1(X)$ is very large if and only if 
$\RR^1_1(A)\ne \{ \zero\}$.
\end{proposition}

\begin{example}
\label{ex:partial}
Let $\Sigma_g$ be a compact, connected Riemann surface 
of genus $g$, and let $X=F_{\Gamma}(\Sigma_g)$ be the partial 
configuration space associated to a finite simple graph $\Gamma$.  
More concretely, if $n$ is the number of vertices of $\Gamma$, then 
$F_{\Gamma}(\Sigma_g)$ is the complement in 
$\Sigma_g^n$ of the union of the diagonals $z_i=z_j$, indexed by 
the edges of $\Gamma$. 
No convenient presentation is available for the fundamental group 
$G_{\Gamma, g}\coloneqq \pi_1(F_{\Gamma}(\Sigma_g))$.  
On the other hand, the Orlik--Solomon model $A$ for 
$F_{\Gamma}(\Sigma_g)$ is much more approachable. Computing the resonance 
variety $\RR^1_1(A)$ leads to a complete, explicit description of 
$\mathcal{E}(F_{\Gamma}(\Sigma_g))$. Such a description is given in 
\cite{BMPP}, for all $g\ge 0$ and for all finite graphs $\Gamma$,  
generalizing a result from \cite{BH}, valid only for chordal graphs. 
In particular, $\mathcal{E}(F_{\Gamma}(\Sigma_g))= \emptyset$, that is, 
$G_{\Gamma,g}$ is not very large, if and only if 
either $g=1$ and $\Gamma$ has no edges, or $g=0$ 
and  $\Gamma$ contains no complete subgraph on $4$ vertices. 
\end{example}

\section{Algebraic models for  Lie group actions}
\label{sect:models-act}

\subsection{Almost free actions and Hirsch extensions}
\label{subsec:s1act}

Let $K$ be a compact, connected, real Lie group. 
Consider the universal principal $K$-bundle, $K\to EK \to BK$, 
with contractible total space $EK$ and with base 
space the classifying space $BK=EK/K$.  
By a classical result of Hopf, the cohomology ring of $K$ (with 
coefficients in a field $\k$ of characteristic $0$) is isomorphic to the 
cohomology ring of a finite product of odd-dimensional spheres.  That is,
$H^{*}(K,\k) \cong \bwedge P^{*}$, 
where $P^{*}$ is an oddly-graded, finite-dimensional vector space, 
with homogeneous basis $\{ t_{\alpha} \in P^{m_{\alpha}} \}$, 
for some odd integers $m_1,\dots ,m_r$, where $r=\rank(K)$.  

Now let $M$ be a compact, connected, differentiable manifold 
on which the compact, connected Lie group $K$ acts smoothly. 
Both $M$ and the orbit space $N=M/K$ have the homotopy type 
of finite $CW$-complexes. We consider the diagonal action 
of $K$ on the product $EK\times M$, and form the Borel construction, 
$M_K=(EK\times M)/K$. Let $\proj \colon M_K \to N$ be the map induced 
by the projection $\proj_2 \colon EK \times M \to M$. 

The $K$-action on $M$ is said to be {\em almost free}\/ if all its 
isotropy groups are finite. When this assumption is met, the work of 
Allday and Halperin \cite{AH} provides a very useful 
Hirsch extension model for the manifold $M$.  

\begin{theorem}[\cite{AH}]
\label{thm:rs}
Suppose $M$ admits an almost free $K$-action, with orbit space $N=M/K$.  
There is then a map $\sigma\colon P^{*} \to Z^{* +1}(\apl (N) )$ 
such that $\proj^*\circ [\sigma]$ is the transgression in the principal bundle 
$K\to EK\times M \to M_K$, and 
\[
\apl (M) \simeq \apl (N) \otimes_{\sigma} \bwedge P .
\]
\end{theorem}

This theorem may be applied for instance to the total space $M$ 
of a principal $K$-bundle over a compact manifold $N=M/K$.
The next result identifies an interesting class of finite-dimensional 
CW-spaces that have finite $\cdga$ models. 

\begin{proposition}[\cite{PS-imrn19}]
\label{prop:newfinite}
Let $M$ be an almost free $K$-manifold. Write 
$H^{*}(K,\k)=\bwedge P$, for some graded $\k$-vector 
space $P$, and let $m$ be the maximum 
degree of $P^*$.
\begin{enumerate}
\item \label{nf1}
Suppose $B$ is a $q$-finite $q$-model of the orbit space $N=M/K$, 
with $q\ge m+1$.  Then a suitable Hirsch extension
$A=B\otimes_{\tau} \bwedge P$ is  a $q$-finite $q$-model for $M$. 
\item \label{nf2}
Suppose $N=M/K$ is $q$-formal.  Then we may take 
$B^{*}=(H^{*}(N,\k), 0)$, and $A=B\otimes_{\tau} \bwedge P$  
is  a $q$-finite $q$-model of $M$ with positive weights.
\end{enumerate}
\end{proposition}

Restricting to principal $K$-bundles, we can say more. 
As before, identify $H^{*}(K, \Q)$ with $\bwedge P= \bwedge (t_1,\dots,t_r)$.

\begin{theorem}[\cite{PS-imrn19}]
\label{thm:krealiz}
Let $N$ be a connected, finite CW-complex and let $K$ be a compact, connected, 
real Lie group. If $N$ has a finite-dimensional 
rational model $B$, then any Hirsch extension 
$A=B\otimes_{\tau} \bwedge P$ can be realized as a finite-dimensional 
rational model of some principal $K$-bundle $M$ over $N$. When $B$ 
has positive weights and the image of $[\tau]$ is generated by 
weighted-homogeneous elements, $A$ also has positive weights. 
\end{theorem}

\subsection{Graded regularity and partial formality}
\label{subsec:partform-circle}

Fix an integer $q\ge 0$.   
Let $H^{*}$ be a connected commutative graded algebra over a 
field $\k$ of characteristic $0$. Following \cite{PS-imrn19}, we 
say that a homogeneous element $e\in H^k$ is a non-zero divisor 
up to degree $q$ if the multiplication map $e \cdot \colon H^i \to H^{i+k}$ 
is injective, for all $i\le q$.  (For $q=0$, this simply means that $e\ne 0$.) 

Likewise, we say that a sequence $e_1,\dots ,e_r$
of  homogeneous elements in $H^+$ is {\em $q$-regular}\/ 
if the class of each $e_{\alpha}$ is a non-zero divisor up to 
degree $q-\deg(e_{\alpha})+2$ in the quotient ring $H/\sum_{\beta <\alpha} e_{\beta} H$.
(This implies in particular that the elements  $e_1,\dots ,e_r$ are 
linearly independent over $\k$, when  $q\ge \deg(e_{\alpha})-2$ for all $\alpha$.)

\begin{theorem}[\cite{PS-imrn19}]
\label{thm:lefred}
Suppose $e_1,\dots ,e_r$ is an even-degree, $q$-regular sequence in $H^{*}$.  
Then the Hirsch extension $A=(H \otimes_{\tau} \bwedge (t_1,\dots ,t_r),d)$ with 
$d=0$ on $H$  and $dt_{\alpha}=\tau (t_{\alpha})=e_{\alpha}$ has 
the same $q$-type as $(H/\sum_{\alpha} e_{\alpha} H, 0)$.
In particular, $A$ is $q$-formal.
\end{theorem}

Classical results of Borel and Chevalley provide the machinery for 
constructing graded algebras which satisfy the hypothesis 
of Theorem \ref{thm:lefred}, in the case when $q=\infty$.
Let $H^{*}(BK,\k)$ be the cohomology algebra of the classifying 
space of a compact, connected Lie group $K$.  Let $T$ be a maximal 
torus in $K$, and let $W=NT/T$ be the Weyl group. The classifying 
space $BT$ is the product of $r$ copies of $\CP^{\infty}$, where $r$ 
is the rank of $K$. Its cohomology algebra is $H^{*}(BT,\k)=\k [x_1,\dots, x_r]$,
with degree $2$ free algebra generators, on which $W$ acts by graded 
algebra automorphisms. 

The natural map $\kappa \colon BT \to BK$ identifies the cohomology 
algebra $H^{*}(BK,\k)$ with the invariant subalgebra
of the $W$-action. More precisely, $H^{*}(BK,\k)$ is isomorphic 
to a polynomial ring of the form $\k [f_1,\dots, f_r]$, 
where each $f_{\alpha}$ is a $W$-invariant polynomial of 
even degree $m_{\alpha}+1$, with $m_{\alpha}$ 
as in Section \ref{subsec:s1act}. Moreover, $f_1,\dots , f_r$  
forms a regular sequence in $\k [x_1,\dots, x_r]$.

Let $U\subseteq K$ be a closed, connected subgroup of a compact, 
connected Lie group. As shown in \cite{Th}, the Sullivan minimal model 
of the homogeneous space $K/U$ is a Hirsch extension of the form 
$A=H \otimes_{\tau} \bwedge (t_1,\dots ,t_s)$, where $H^{*}$ 
is a free graded algebra on finitely many even-degree generators, 
with zero differential, as in Theorem \ref{thm:lefred}.  
As is well-known, not all homogeneous spaces $K/U$ are formal. 
Nevertheless, the criterion from Theorem \ref{thm:lefred} 
may be used to gain  information on their partial 
formality properties.

\begin{example}
\label{ex=homnonf}
For the homogeneous space $\Sp(5)/\SU(5)$, 
the aforementioned algebra $H^{*}$ has two free generators,
$x_6$ and $x_{10}$, where subscripts denote degrees, and the sequence 
from Theorem \ref{thm:lefred} is $\{x^2_6, x^2_{10}, x_6 x_{10}\}$, 
see \cite{FOT}. 
It follows  that $\Sp(5)/\SU(5)$ is $19$-formal. 
On the other hand, a computation with Massey triple 
products shows that this estimate is sharp, that is, 
$\Sp(5)/\SU(5)$ is not $20$-formal. 
\end{example}

\subsection{Partial formality of $K$-manifolds}
\label{subsec:pfk}

Let $M$ be an almost free $K$-manifold. We 
write $H^{*}(K,\k)= \bwedge (t_1,\dots ,t_{r})$, 
and denote the transgression of $t_{\alpha}$ by 
$e_{\alpha} \in H^{m_{\alpha} +1} (M/K,\k)$.  As before, set 
$m= \max \{ m_{\alpha} \}$.

\begin{theorem}[\cite{PS-imrn19}]
\label{thm:nformal}
Suppose the $K$-action on $M$ is almost free, the orbit space $N=M/K$ is $k$-formal, 
for some $k\ge m+1$, and $e_1,\dots ,e_r$ form a $q$-regular sequence in $H^{*}(N,\k)$,
for some $q\le k$.  Then the quotient algebra 
$H^{*}(N,\k) /\sum_{\alpha=1}^r e_{\alpha} H^{*}(N,\k)$, 
equipped with the zero differential,  
is a finite-dimensional $q$-model for $M$; in particular, 
$M$ is $q$-formal.
\end{theorem}

As illustrated in the next two examples, the $q$-regularity 
assumption from Theorem \ref{thm:nformal} is
optimal with respect to the $q$-formality conclusion for 
the manifold $M$, at least in the case when $K=S^1$ or $S^3$. 

\begin{example}
\label{ex:heis-again}
Let $M=\Heis_1$ be the $3$-dimensional Heisenberg nilmanifold 
from Example \ref{ex:heis-massey}.  This manifold is the total space of the 
principal $S^1$-bundle over the formal manifold $N=S^1 \times S^1$, 
with Euler class $e\in H^2 (N,\Z)$ equal to the orientation class. 
In this case, the sequence $\{ e \}$ is $0$-regular, but not $1$-regular
in $H^{*} (N,\k)$. In fact, as mentioned previously, $M$ is not $1$-formal. 
As explained in Example \ref{ex:heis-sasaki}, this is the first manifold in a series, 
$\Heis_n$, where $(n-1)$-regularity implies $(n-1)$-formality in an optimal way. 
\end{example}

\begin{example}
\label{ex:hopf-nonformal}
Let  $M$ to be the total 
space of the principal $S^3$-bundle over $N=S^2 \times S^2$ 
obtained by pulling back the Hopf bundle $S^7\to S^4$ along 
a degree-one map $N\to S^4$. As above, $N$ is 
formal, and the Euler class $e\in H^4 (N,\Z)$ 
is the orientation class.  In this case, $\{ e \}$ is 
$3$-regular, but not $4$-regular 
in $H^{*} (N,\k)$, and Theorem \ref{thm:nformal} 
says that $M$ is $3$-formal. Direct computation with the 
minimal model of $M$ shows that, in fact, $M$ is 
not $4$-formal.
\end{example}

\subsection{Malcev completion and representation varieties}
\label{subsec:formalbase}

Let $H$ be a $2$-finite $\cdga$ with zero differential, and 
let $A=H\otimes_{\tau} \bwedge P$ be a Hirsch extension, 
where $P$ is an oddly-graded, finite-dimensional vector space.

\begin{theorem}[\cite{PS-imrn19}]
\label{thm:filtf}
The holonomy Lie algebra $\h(A)$ admits a finite presentation
with generators in degree $1$ and relations in degrees $2$ and $3$. 
\end{theorem}

\begin{corollary}[\cite{PS-imrn19}]
\label{cor:ffcircle}
Suppose $M$ supports an almost free $K$-action with $2$-formal orbit space. 
Then:
\begin{enumerate}
\item \label{ff1}
The group $\pi=\pi_1(M)$ is filtered-formal. 
More precisely, the Malcev Lie algebra $\m(\pi)$ is isomorphic 
to the lcs completion of  $\Lie(H_1(\pi, \k))/\mathfrak{r}$, where 
$\mathfrak{r}$ is a homogeneous ideal generated in degrees $2$ and $3$. 

\item \label{ff2}
For every complex linear algebraic group $G$,  
the germ at the origin of the representation variety 
$\Hom_{\gr}(\pi,G)$ is defined by quadrics and cubics only.
\end{enumerate}
\end{corollary}

The second statement in the above corollary is analogous to the 
quadraticity obstruction for fundamental groups of compact K\"ahler 
manifolds obtained by Goldman--Millson in \cite{Goldman-Millson}. 
Note that the corollary applies to principal $K$-bundles over formal manifolds.

\subsection{Orbifold fundamental groups}
\label{subsec:orbifold}
Assume now that $M$ is an almost free $K$-manifold. By \cite[Theorem 4.3.18]{BG}, 
the projection $p\colon M\to M/K$ induces a natural epimorphism 
$f\colon \pi_1 (M) \surj \pi_1^{\orb} (M/K)$ between orbifold fundamental groups.

\begin{theorem}[\cite{PS-imrn19}]
\label{thm:monotr}
Suppose that the $K$-action on $M$ is almost free and  the transgression 
$P^{*} \to H^{* +1}(M_K,\k) \cong H^{* +1}(M/K,\k)$ is injective in degree $1$.
Then the following hold.
\begin{enumerate}
\item \label{mn1}
If the orbit space $N=M/K$ has a $2$-finite $2$-model over $\k \subseteq \C$,
then the homomorphism $f\colon \pi_1 (M) \surj \pi_1^{\orb} (N)$ 
induces an analytic isomorphism between the germs at $1$ of 
$\VV^1_k(\pi_1^{\orb} (N))$ and $\VV^1_k (\pi_1 (M))$, 
for all $k$.
\item \label{mn2}
If $N$ is $2$-formal, then $f$ induces an analytic isomorphism between the 
germs at $1$ of $\Hom (\pi_1^{\orb} (N), \SL_2 (\C))$ and $\Hom (\pi_1 (M), \SL_2 (\C))$.
\end{enumerate}
\end{theorem}

\begin{example}
\label{ex:tauinj}
Let $K$ be a compact, connected Lie group, and identify 
$H^{*}(K,\Q)$ with $\bigwedge P^{*}_K$. Let $N$ be a 
compact, formal manifold, and assume $b_2(N)\ge s$, 
where $s=\dim P^1_K$  (for instance, take $N$ to be the 
product of at least $s$ compact K\"{a}hler manifolds). 
There is then a degree-preserving 
linear map, $\tau \colon P^{*}_K \to H^{* +1}(N,\Q)$, which is  
injective in degree $1$.  By Theorem \ref{thm:krealiz}, such a map 
can be realized as the transgression in a principal $K$-bundle, $M_{\tau}\to N$,  
and the manifold $M_{\tau}$ satisfies the assumptions from Theorem \ref{thm:monotr}. 
\end{example}

Theorem \ref{thm:monotr} may also be applied to a Seifert fibered $3$-manifold 
with non-zero Euler class, $p\colon M \to M/S^1 =\Sigma_g$. In this case, 
$\VV^1_k (M)_{(\one)}$ is isomorphic to $\VV^1_k (\Sigma_g)_{(\one)}$, 
for all $k$, while $\Hom (\pi_1 (M), \SL_2 (\C))_{(\one)}\cong 
\Hom (\pi_1 (\Sigma_g), \SL_2 (\C))_{(\one)}$.

\subsection{Sasakian geometry}
\label{subsec:sas}
The machinery outlined above has some noteworthy  
consequences for the topology of compact Sasakian manifolds, 
which are related to formality properties, representation varieties and 
cohomology jump loci. A comprehensive reference for Sasakian 
geometry is the book of Boyer and Galicki \cite{BG}.  

Let $M^{2n+1}$ be a compact Sasakian manifold of dimension $2n+1$.  
Without loss of essential generality, we may assume that the Sasakian 
structure is quasi-regular.   A basic structural result in Sasakian geometry 
guarantees that, in this case, $M$ supports an almost free circle action.  
Furthermore, the quotient space, $N=M/S^1$, is a compact K\"{a}hler orbifold, 
with K\"{a}hler class $h\in H^2(N,\k)$ satisfying the Hard Lefschetz property, 
that is, multiplication by $h^k$ defines an isomorphism 
\begin{equation}
\label{eq:hardlef}
\begin{tikzcd}[column sep=22pt]
H^{n-k}(N,\k) \ar[r, "\cong"] & H^{n+k}(N,\k)
\end{tikzcd}
\end{equation}
for each $1\le k\le n$; see 
\cite[Proposition 7.2.2 and Theorem 7.2.9]{BG}.
The thesis of Tievsky \cite[\S 4.3]{Ti} provides a very useful model 
for a Sasakian manifold.

\begin{theorem}[\cite{Ti}]
\label{thm:sasmodel}
Every compact Sasakian manifold $M$ admits as a finite model over $\R$ the Hirsch 
extension $A^{*}(M)=(H^{*}(N,\R)\otimes _h \bwedge(t), d)$, 
where $d$ is zero on $H^{*}(N,\R)$ and $dt=h$, 
the K\"{a}hler class of $N$.  
\end{theorem}

Sasakian geometry is an odd-dimensional analog of K\"{a}hler 
geometry.  From this point of view, the above theorem is a 
rough analog of the main result on the algebraic topology 
of compact K\"ahler manifolds from \cite{DGMS}, guaranteeing 
that such manifolds are formal. Theorem \ref{thm:sasmodel} only says that 
$M$ behaves like an almost free compact $S^1$-manifold with formal orbit space.
A result from \cite{BB+} establishes the formality of the orbifold de Rham algebra 
of a compact K\"{a}hler orbifold. Unfortunately, this is not enough for applying Theorem 
\ref{thm:nformal}, since the authors of \cite{BB+} do not prove that the orbifold
de Rham algebra is weakly equivalent to the Sullivan de Rham algebra. 

By construction, the Tievsky model $A^{*}(M)$ is a real $\cdga$ defined 
over $\Q$. Nevertheless, in view of Remark \ref{rem:non-descent}, 
it does {\em not}\/ follow from \cite{Ti} that $A^{*}(M)$ 
is a model for $M$ over $\Q$.  

However, we can say something very useful 
regarding rational models for Sasakian manifolds.  We start 
with a lemma and will come back to this point in 
Theorem \ref{thm:highf-bis}.

\begin{lemma}[\cite{PS-imrn19}]
\label{lem:qtievsky}
The Tievsky model $A^{*}_{\R}(M)=(H^{*}(N, \R)\otimes _h \bwedge(t), d)$ 
is a finite model with positive weights for $M$. 
\end{lemma} 

\begin{corollary}[\cite{PS-imrn19}]
\label{cor:tmodel}
Let $M$ be a compact Sasakian manifold. 
For each $i, k\ge 0$, all irreducible components 
of the characteristic variety $\VV^i_k(M)$ passing through $1$ 
are algebraic subtori of the character group $H^1(M,\C^*)$.
\end{corollary}

A well-known, direct relationship between K\"{a}hler and Sasakian 
geometry is as follows.  Let $N$ be a compact K\"{a}hler manifold 
such that the K\"ahler class is integral, i.e., $h\in H^2(N,\Z)$, 
and let $M$ be the total space of the principal $S^1$-bundle 
classified by $h$.  Then $M$ is a regular Sasakian 
manifold.  A concrete class of examples is provided by the 
Heisenberg manifolds $\Heis_n$ from Example \ref{ex:heis-sasaki} below.

\subsection{Partial formality of Sasakian manifolds}
\label{subsec:sas1f}
Let $M^{2n+1}$ be a compact Sasakian manifold, 
with fundamental group $G=\pi_1(M)$.  One may ask: 
Is the group $\pi$ (or, equivalently, the manifold $M$) 
$1$-formal?  When $n=1$, the answer is clearly negative, 
a simple example being provided by the Heisenberg manifold 
$\Heis_1$.  In \cite[Theorem 1.1]{Kas17}, Kasuya 
claims that the case $n=1$ is exceptional, in the 
following sense.

\begin{claim}
\label{thm:sas1f}
Every compact Sasakian manifold of dimension $2n+1$ is 
$1$-formal over $\R$, provided $n>1$.
\end{claim}

As pointed out in \cite{PS-imrn19}, the proof from \cite{Kas17} has a gap, 
which we briefly explain. Given a $\cdga$ $A$, the decomposable part
of $H^2(A)$ is the linear subspace $DH^2(A)$ defined as the image of 
the product map in cohomology, $H^1(A)\wedge H^1(A)\to H^2(A)$. 
What Kasuya actually shows is that 
\begin{equation}
\label{eq:kdec}
DH^2(\M_1(M)) = H^2(\M_1(M)),
\end{equation}
for a compact Sasakian manifold $M^{2n+1}$ with $n>1$, 
where $\M_1(M)$ is the $1$-minimal model of $M$ over $\R$.
Equality \eqref{eq:kdec} is an easy consequence of $1$-formality.  
Kasuya deduces the $1$-formality of $M$ from \eqref{eq:kdec}, 
by invoking as a crucial tool Lemma 3.17 from \cite{ABC}.   
Unfortunately, though, this lemma is false, as shown by M\u{a}cinic 
in \cite{Mc10}.   
Nevertheless, the next theorem proves Claim \ref{thm:sas1f} 
in a stronger form, while also recovering equality \eqref{eq:kdec}.

\begin{theorem}[\cite{PS-imrn19}]
\label{thm:highf}
Every compact Sasakian manifold $M$ of dimension $2n+1$ is 
$(n-1)$-formal, over an arbitrary field $\k$ of characteristic $0$. 
\end{theorem}

The next result makes Theorem \ref{thm:highf} more precise, 
by constructing an explicit finite, $(n-1)$-model with zero differential 
for $M$ over any field of characteristic $0$. 

\begin{theorem}[\cite{PS-imrn19}]
\label{thm:highf-bis}
Let $M$ be a compact Sasakian manifold $M$ of dimension $2n+1$. 
The Sullivan model of $M$ over a field $\k$ of 
characteristic $0$ has the same $(n-1)$-type over $\k$ as the \cdga~
$(H^{*}(N,\k)/h\cdot H^{*}(N,\k), 0)$, 
where $N=M/S^1$ and $h\in H^2(N, \k)$ is the K\"ahler class.
\end{theorem}

As illustrated by the next example, the conclusion of 
Theorem  \ref{thm:highf} is optimal. 

\begin{example}
\label{ex:heis-sasaki}
Let $E=S^1\times S^1$ be an elliptic complex curve, and let $N=E^{\times n}$ 
be the product of $n$ such curves, with K\"{a}hler form $\omega=\sum_{i=1}^n 
dx_i \wedge dy_i$.  The corresponding Sasakian manifold is the 
$(2n+1)$-dimensional Heisenberg nilmanifold $\Heis_n$. 
Theorem \ref{thm:highf} guarantees that $\Heis_n$ is $(n-1)$-formal. 
As noted in \cite{Mc10}, though (see also 
Example \ref{ex:gen-heisenberg}), the manifold $\Heis_n$ is 
{\em not}\/ $n$-formal. 
\end{example}

\subsection{Sasakian groups}
\label{subsec:sasgp}

A group $\pi$ is said to be a {\em Sasakian group}\/ if it 
can be realized as the fundamental group of a compact, 
Sasakian manifold.  A major open problem in the field 
(see e.g. \cite[Chapter 7]{BG} or \cite{Chen13}) is:
``Which finitely presented groups are Sasakian?''

A first, well-known obstruction is that the first Betti number $b_1(\pi)$ 
must be even, see for instance the references listed in \cite{Chen13}.  
Much more subtle obstructions are provided by the following result.  
Fix a field $\k$ of characteristic $0$.

\begin{corollary}[\cite{PS-imrn19}]
\label{cor:sasobstr}
Let $\pi= \pi_1 (M^{2n+1})$ be a Sasakian group.  Then:
\begin{enumerate}
\item \label{s1}
The Malcev Lie algebra $\m(\pi, \k)$ is the lcs completion of the 
quotient of the free Lie algebra $\Lie (H_1(\pi, \k))$ by an ideal 
generated in degrees $2$ and $3$.  
Moreover, this Lie algebra presentation can be explicitly described in
terms of the graded ring $H^{*}(M/S^1, \k)$ and the K\"ahler class 
$h\in H^{2}(M/S^1, \k)$.
\item \label{s2}
The group $\pi$ is filtered-formal. 
\item \label{s3}
For every complex linear algebraic group $G$,
the germ at the origin of the representation variety $\Hom(\pi, G)$ 
is defined by quadrics and cubics only.
\end{enumerate}
\end{corollary}

As an application of Corollary \ref{cor:tmodel}, we obtain 
another (independent) obstruction to Sasakianity. 

\begin{corollary}[\cite{PS-imrn19}]
\label{cor:sasgpobs}
Let $\pi$ be a Sasakian group. 
For each $k\ge 0$, all irreducible components 
of the characteristic variety $\VV^1_k(\pi)$ passing through $\one$ 
are algebraic subtori of the character group $\Hom(\pi,\C^*)$.
\end{corollary}

By Theorem \ref{thm:sasmodel}, the $\R$-homotopy type of a 
compact Sasakian manifold $M$ depends only on the cohomology ring 
$H^{*}(M/S^1, \R)$ and the K\"{a}hler class $h\in H^2(M/S^1,\Q)$. 
Surprisingly enough, it turns out that the germs at $\one$ of certain representation 
varieties and jump loci of $\pi_1(M)$ depend only on the graded cohomology 
ring of $M/S^1$. 

\begin{corollary}[\cite{PS-imrn19}]
\label{cor:nokclass}
Let $M$ be a compact Sasakian manifold, and let $G=\SL_2(\C)$.  
Then the germ at $\one$ of $\Hom(\pi_1(M), G)$ depends only on the 
graded ring $H^{*}(M/S^1, \C)$ and the Lie algebra of $G$, 
in an explicit way. Similarly, the germs at $\one$ of the characteristic 
varieties $\VV^1_k(\pi_1(M))$ depend (explicitly) only on  $H^{*}(M/S^1, \C)$. 
\end{corollary}

\section{Algebraic models for closed $3$-manifolds}
\label{sect:models-3d}

In this final section we give a partial characterization of the formality 
and finiteness properties for rational models of closed $3$-manifolds.

\subsection{The intersection form of a $3$-manifold}
\label{subsec:coho 3-mfd}

Let $M$ be a compact, connected $3$-mani\-fold without boundary. 
For short, we shall refer to $M$ as being a {\em closed}\/ $3$-manifold.
Throughout, we will also assume that $M$ is orientable.   

Fix an orientation class $[M]\in H_3(M,\Z)\cong\Z$. With 
this choice, the cup product on $M$ determines an alternating 
$3$-form $\mu_M$ on $H^1(M,\Z)$, given by 
\begin{equation}
\label{eq:eta}
\mu_M(a\wedge b\wedge c) = \langle a\cup b\cup c , [M]\rangle,
\end{equation} 
where $\langle \cdot, \cdot \rangle$ denotes the Kronecker pairing. 
In turn, the cup-product map $\bigwedge^2 H^1(M,\Z) 
\to H^2(M,\Z)$ is determined by the intersection form $\mu_M$ 
via $\langle a\cup b , \gamma \rangle = \mu_M (a\wedge b \wedge c)$, 
where $c$ is the Poincar\'{e} dual of $\gamma \in H_2(M,\Z)$.  

In \cite{Sullivan75}, Sullivan proved the following result.
\begin{theorem}[\cite{Sullivan75}]
\label{thm:sullivan-3m}
For every finitely generated, 
torsion-free abelian group $H$ and every $3$-form $\mu\in \bwedge^3 H^{\vee}$, 
there is a closed, oriented $3$-manifold $M$ with $H^1(M,\Z)=H$ 
and cup-product form $\mu_M=\mu$. 
\end{theorem}

Such a $3$-manifold can be constructed by a process known as ``Borromean surgery.''
More precisely, if $n=\rank H$, a manifold $M$ with the claimed properties 
may be defined as $0$-framed surgery on a link in $S^3$ obtained from the trivial 
$n$-component link by replacing a collection of trivial $3$-string braids by 
the corresponding collection of $3$-string braids whose closures are the Borromean rings. 
For instance, $0$-surgery on the Borromean rings produces the $3$-torus $T^3$. 

\subsection{Poincar\'e duality and Koszul complex}
\label{subsec:koszul 3-mfd} 
We now fix a basis $\{e_1,\dots ,e_n\}$ for the free abelian group 
$H^1(M,\Z)$, and we choose $\{e^{\vee}_1,\dots ,e^{\vee}_n\}$ as basis 
for the torsion-free part of $H^2(M,\Z)$, where $e^{\vee}_i$ 
denotes the Kronecker dual of the Poincar\'e dual of $e_i$. 
Writing  
\begin{equation}
\label{eq:muijk}
\mu_M=\sum_{1\le i<j<k\le n} \mu_{ijk} e_i  e_j  e_k,
\end{equation}
where $\mu_{ijk}=\mu(e_i\wedge e_j\wedge e_k )$ and 
using formula \eqref{eq:eta}, we find that 
$e_i  e_j=\sum_{k=1}^{n} \mu_{ijk}  e^{\vee}_k$.  

In order to identify the resonance varieties of the cohomology algebra $A^*=H^*(M,\C)$, 
we let $S=\Sym(A_1)$ be the symmetric algebra on $A_1=H_1(M,\C)$, 
and we identify $S$ with the polynomial ring $\C[x_1,\dots,x_n]$. 
The Koszul complex from \eqref{eq:koszul-dga} then has the form 
\begin{equation}
\label{eq:koszul 3mfd}
\begin{tikzcd}[column sep=24pt]
A^0\otimes_\C S \ar[r, "\delta^{0}_A"] 
& A^1\otimes_\C S \ar[r, "\delta^{1}_A"] 
&A^2\otimes_\C  S \ar[r, "\delta^{2}_A"] 
&A^3\otimes_\C  S,
\end{tikzcd}
\end{equation}
where the differentials are the $S$-linear maps given by 
$\delta^q_A(u)=\sum_{j=1}^{n} e_j u \otimes x_j$ for $u\in A^q$.  
In our chosen basis, the matrix of  $\delta^2_A$ is the transpose 
of $\delta^0_A=\big(x_1 \, \cdots\, x_n\big)$, while the matrix 
of $\delta^1_A$ is an $n\times n$ matrix of linear forms in 
the variables $x_i$, given by $\delta^1_A(e_i)=
\sum_{j=1}^{n} \sum_{k=1}^{n}\mu_{jik} e_k^{\vee} \otimes x_j$. 

Note that the matrix $\delta_M\coloneqq \delta^1_A$ is skew-symmetric; 
moreover, it is singular, since the vector $(x_1,\dots, x_n)$ is in its kernel.   
Hence, both the determinant $\det(\delta^1_A)$ and the Pfaffian 
$\pf(\delta_M)$ vanish. 
Let $\delta_M(i;j)$ be the sub-matrix obtained from $\delta_M$ 
by deleting the $i$-th row and $j$-th column. We then have the following 
lemma, due to Turaev \cite{Tu}. 

\begin{lemma}[\cite{Tu}]
\label{lem:turaev}
Assume $n\ge 3$. There is then a polynomial 
$\Det(\mu)\in S$ such that $\det \delta_M(i;j) = (-1)^{i+j}x_ix_j \Det(\mu)$.  
Moreover, if $n$ is even, then $\Det(\mu)=0$, while if $n$ is odd, 
then $\Det(\mu)=\Pf(\mu)^2$, where $\pf (\delta_M(i;i)) = (-1)^{i+1} x_i \Pf(\mu)$. 
\end{lemma}

\subsection{Resonance varieties of $3$-manifolds}
\label{subsec:res 3-mfd} 
Let $\RR^i_k(M)$ be the resonance varieties 
associated to the cohomology algebra $A=H^*(M,\C)$ of a closed, 
orientable $3$-manifold $M$.  
As shown in \cite{Su-edinb}, Poincar\'e duality implies that 
$\RR^{2}_k(M)=\RR^1_k(M)$ for $1\le k \le n$, while  
$\RR^{3}_1(M)=\RR^0_1(M)=\{\zero\}$ if $n>0$.  
The basic structure of the degree $1$, depth $1$ resonance 
varieties is given by the following theorem.

\begin{theorem}[\cite{Su-edinb}, \cite{Su-mm}]
\label{thm:res closed3m}
Let $M$ be a closed, orientable $3$-manifold.  Set $n=b_1(M)$ 
and let $\mu_M$ be the associated alternating $3$-form.  Then 
\begin{equation}
\label{eq:r1 3m}
\RR^1_1(M)= 
\begin{cases}
\emptyset & \text{if\/ $n=0$};\\
\{0\} & \text{if\/ $n=1$ or $n=3$ and $\mu_M$ has rank $3$};\\
\mathbf{V}(\Pf(\mu_M))  & \text{if\/  $n$ is odd, $n>3$, 
and $\mu_M$ is generic};\\
H^1(M;\C) & \text{otherwise}.
\end{cases}
\end{equation}
\end{theorem}

In the case when $n=2g+1$ with $g>1$, we say that the alternating form 
$\mu_M$ is {\em generic} (in the sense of Berceanu and Papadima 
\cite{BP}) if there is an element $c\in A^1$ such that the $2$-form 
$\gamma_c\in A_1\wedge A_1$ defined by 
$\gamma_c(a \wedge b)=\mu_M(a\wedge b\wedge c)$ for $a,b\in A^1$ 
has maximal rank, that is, $\gamma^{g}_c\ne 0$ in $\bwedge^{2g} A_1$.  
For detailed information on the resonance varieties 
$\RR^1_k(M)$ in depth $k>1$ we refer to \cite{Su-edinb}. 

\subsection{Characteristic varieties of $3$-manifolds}
\label{subsec:cvs 3-mfd} 
As noted in \cite{Su-mm}, Poincar\'e duality with local coefficients imposes 
the same type of constraints on the characteristic varieties of a closed, orientable 
$3$-manifold $M$; for instance, $\VV^{1}_k(M)\cong \VV^{2}_k(M)$, 
for all $k\ge 0$. Best understood is the variety $\VV^{1}_1(M)$, due
to its close connection to both the resonance variety $\RR^{1}_1(M)$ 
and to the Alexander polynomial $\Delta_M$, which we define next. 

Let $H=G_{\ab}/\Tors(G_{\ab})$ be the maximal 
torsion-free abelian quotient of the group $G=\pi_1(M, x_0)$, and 
let $q\colon M^{H}\to M$ be the regular cover corresponding to the 
projection $G\surj H$.  The Alexander module of $M$ is 
defined as the relative homology group 
$A_M= H_1(M^H,q^{-1}(x_0); \Z)$, viewed as a module 
over the Noetherian ring $\Z{H}$. 
Finally, let $E_1(A_M)\subseteq \Z{H}$ be the 
ideal of codimension $1$ minors in a $\Z{H}$-presentation for $A_M$.  
The {\em Alexander polynomial}\/ of $M$ is then defined as the greatest 
common divisor of the elements in this determinantal ideal,
$\Delta_M=\gcd ( E_1(A_M))$. 

As noted in \cite{DPS-imrn} and \cite{Su-mm}, work of McMullen \cite{McM} and 
Turaev \cite{Tu} yields the following relationship between the first 
characteristic variety and the Alexander polynomial of $M$.

\begin{proposition}[\cite{DPS-imrn}, \cite{Su-mm}]
\label{prop:cv 3d}
Let $M$ be a closed, orientable, $3$-dimensional manifold.  Then 
\begin{equation}
\label{eq:vdel}
\VV^1_1(M)\cap \Char(M)^0=\mathbf{V}(\Delta_M) \cup \{\one\}.
\end{equation}
Moreover, if $b_1(M)\ge 4$,  then $\VV^1_1(M)\cap 
\Char(M)^0= \mathbf{V}(\Delta_M)$.  
\end{proposition}

The next theorem shows that the second half of the Tangent Cone formula 
\eqref{eq:tc} holds for a large class of closed $3$-manifolds with odd first 
Betti number (regardless of whether those manifolds are $1$-formal), 
yet fails for most $3$-manifolds with even first Betti number.

\begin{theorem}[\cite{Su-mm}]
\label{thm:tc 3d}
Let $M$ be a closed, orientable $3$-manifold, and 
set $n=b_1(M)$. 
\begin{enumerate}
\item \label{odd} 
If $n\le 1$, or $n$ is odd, $n\ge 3$, and $\mu_M$ 
is generic, then $\TC_{\one}(\VV^1_1(M))=\RR^1_1(M)$.
\item \label{even} 
If $n$ is even, $n\ge 2$, then 
$\TC_{\one}(\VV^1_1(M))= \RR^1_1(M)$ if 
and only if $\Delta_M= 0$.
\end{enumerate}
\end{theorem}

The information contained in the cohomology jump loci and 
the Alexander polynomials provides a method for determining 
which $3$-manifold groups can also be realized as fundamental groups 
of K\"{a}hler manifolds, or smooth, quasi-projective varieties. 
We summarize the relevant results from \cite{DS-jems}, 
\cite{DPS-mz}, and \cite{FS}, as follows.

\begin{theorem}
\label{thm:3d-k-qp}
Let $G=\pi_1(M)$ be the fundamental group of a closed, orientable 
$3$-manifold $M$. Then:
\begin{description}[labelindent=14pt]
\item [\cite{DS-jems}]
$G\cong \pi_1(X)$, for some compact K\"{a}hler manifold $X$ if and only if 
$G$ is a finite subgroup of $\operatorname{SO}(4)$, acting freely on $S^3$.

\item [\cite{DPS-mz}]
$G$ is $1$-formal and $G\cong \pi_1(X)$, for some smooth quasi-projective
variety $X$ if and only if
$\m(G)\cong \m(F_n)$ or $\m(G)\cong \m(\Z\times \pi_1(\Sigma_g))$.

\item [\cite{FS}]
If $G\cong \pi_1(X)$, for some smooth quasi-projective
variety $X$, then all the prime components
of $M$ are graph manifolds.
\end{description}
\end{theorem}

\subsection{Finite models for $3$-manifolds}
\label{subsec:models 3-mfd}

The previous theorem leads to obstructions to the existence of $\cdga$ 
models for closed $3$-manifolds with specified finiteness properties. 
These obstructions are quite effective since they are expressed 
solely in terms of the Alexander polynomial of the manifold.

\begin{theorem}[\cite{Su-mm}]
\label{thm:even betti1}
Let $M$ be a closed, orientable, $3$-manifold, and set 
$n=b_1(M)$. 

\begin{enumerate}
\item \label{f1}
If $n\le 1$, then $M$ is formal, and has the rational homotopy type 
of $S^3$ or $S^1\times S^2$. 
\item \label{f2} 
If $n$ is even, $n\ge 2$, and $\Delta_M\ne 0$, then 
$M$ is not $1$-formal. 
\item \label{f3} 
If $\Delta_M\ne 0$, yet $\Delta_M(\one) =0$  and $\TC_{\one}(V(\Delta_M))$ 
is not a finite union of rationally defined linear subspaces, then $M$ 
admits no $1$-finite $1$-model. 
\end{enumerate}
\end{theorem}

\begin{proof}
For completeness, we give a proof of this result. As shown in \cite{FM05}, 
the $1$-formality of $M$ is equivalent to formality. On the other hand, we saw in 
Example \ref{ex:low-betti} that any finitely generated group $G$ 
with $b_1(G)\le 1$ is $1$-formal. Thus, if 
$b_1(M)=0$ or $1$, then $M$ is formal, and so, as noted in \cite{PS-formal}, $M$ 
must be rationally homotopy equivalent to either $S^3$ or $S^1\times S^2$. 

Now suppose $b_1(M)$ is even and positive, and $\Delta_M\ne 0$. 
Then, by Theorem \ref{thm:tc 3d}, we have that 
$\TC_{\one}(\VV^1_1(M))\ne  \RR^1_1(M)$, and so, by 
Corollary \ref{cor:tcone}, $M$ is not $1$-formal. 

Finally, if $\Delta_M\ne 0$ and $\Delta_M(\one) =0$, 
it follows from Proposition \ref{prop:cv 3d} that 
$\VV^1_1(M)$ and $\mathbf{V}(\Delta_M)$ share the 
same tangent cone and exponential tangent cone at $\one$. 
On the other hand, if not all the irreducible components of 
$\TC_{\one}(\mathbf{V}(\Delta_M))$ are rational linear subspaces,  
then, by Lemma \ref{lem:exp-tcone}, 
$\tau_{\one}(\mathbf{V}(\Delta_M))\ne \TC_{\one}(\mathbf{V}(\Delta_M))$. 
Therefore, if both assumptions are satisfied, 
$\tau_{\one}(\VV^1_1(M))\subsetneqq \TC_{\one}(\VV^1_1(M))$, 
and so, by Theorem \ref{thm:tcone-fm}, 
$M$ cannot have a $1$-finite $1$-model. 
\end{proof}

Consequently, if $\m=\m(G)$ is the Malcev Lie algebra of $G=\pi_1(M)$, 
then the following hold in the three cases delineated in Theorem \ref{thm:even betti1}:
\ref{f1} $\m=0$ (if $n=0$) or $\m=\Q$ (if $n=1$); 
\ref{f2} $\m$ is not the LCS completion of a finitely generated, quadratic Lie algebra; 
and \ref{f3} $\m$ is not the LCS completion of a finitely presented Lie algebra. 

The next two examples illustrate how the finiteness obstructions provided 
by Theorem \ref{thm:even betti1} work in cases \ref{f2} and \ref{f3}.

\begin{example}
\label{ex:notfm-finite}
The Heisenberg $3$-dimensional nilmanifold $M$ admits a finite model, 
for instance, $A=\big(\bwedge(a,b,c),d\big)$ with $d a=d b=0$ and $d c=ab$. 
Nevertheless, $M$ is not $1$-formal, since $b_1(M)=2$ 
and $\Delta_M=1$. Furthermore, $\mu_M=0$, and so 
$\TC_{\one}(\VV^1_1(M))=\{\zero\}$, whereas $\RR^1_1(M)=\C^2$. 
\end{example}

\begin{example}
\label{ex:finite}
Let $M$ be a closed, orientable $3$-manifold with 
$H_1(M,\Z)=\Z^2$ and $\Delta_M=(t_1+t_2)(t_1t_2+1)-4t_1t_2$ 
(such a manifold exists by \cite[VII.5.3]{Tu}). Then 
$\TC_{\one}(\VV^1_1(M))=\{x_1^2+x_2^2=0\}$ decomposes 
as the union of two lines defined over $\C$, but not over $\Q$; 
hence, $M$ admits no $1$-finite $1$-model. Furthermore, 
$\tau_{\one}(\VV^1_1(M))=\{\zero\}$ is properly contained 
in $\TC_{\one}(\VV^1_1(M))$.  
\end{example}

\subsection{$3$-manifolds fibering over the circle}
\label{subsec:3mfd}
We conclude this section with a discussion of the $1$-formality 
property for closed $3$-manifolds that fiber over $S^1$. 
We start with a result which relates the notion of $1$-formality 
of a semidirect product of the form $G=K\triangleleft \Z$ 
to the algebraic monodromy of the extension.

\begin{theorem}[\cite{PS-forum}]
\label{thm:monoformal}
Let $1\to K\to G\to \Z\to 1$ be a short exact sequence of groups. 
Suppose $G$ is finitely presented and $1$-formal, 
and $b_1(K)<\infty$. Then the 
eigenvalue $1$ of the monodromy action on 
$H_1(K, \C)$ has only $1\times 1$ Jordan blocks.
\end{theorem}

This theorem yields as an immediate corollary 
a substantial extension of a result of Fern\'{a}ndez, Gray, and 
Morgan \cite{FGM}, where the non-formality of the total spaces 
of certain bundles is established by a different method, 
using Massey products.

\begin{corollary}[\cite{PS-forum}]
\label{cor:bundles}
Let $F\to X\to S^1$ be a smooth fibration 
whose fiber $F$ is connected and has the homotopy 
type of a CW-complex with finite $2$-skeleton, and  
for which the monodromy on $H_1(F, \C)$ has 
eigenvalue $1$, with a Jordan block of size greater than $1$.  
Then the group $G=\pi_1(X)$ is not $1$-formal.
\end{corollary} 

Next, we recall a result from \cite{PS-plms10}, which is based 
on the interplay between the Bieri--Neumann--Strebel invariant, 
$\Sigma^1(G)$, and the (first) resonance variety, $\RR^1_1(G)$, 
of a $1$-formal group $G$. 

\begin{proposition}[\cite{PS-plms10}]
\label{prop:3mfd-fibers}
Let $M$ be a closed, orientable $3$-manifold which fibers 
over the circle.  If $b_1(M)$ is even, then $M$ is 
not $1$-formal.  
\end{proposition}

Combining the results above yields the following corollary, 
which puts strong restrictions on the algebraic monodromy of 
a formal $3$-manifold fibering over the circle.  

\begin{corollary}[\cite{PS-forum}]
\label{cor:jordan blocks}
Let $M$ be a closed, orientable, $1$-formal $3$-manifold.  
Suppose $M$ fibers over the circle, and the algebraic monodromy has $1$ 
as an eigenvalue.  Then, there are an even number of 
$1\times 1$ Jordan blocks for this eigenvalue, 
and no higher size Jordan blocks.
\end{corollary}

Indeed, by Corollary \ref{cor:bundles}, the algebraic monodromy 
has only $1\times 1$ Jordan blocks for the eigenvalue $1$. 
Let $m$ be the number of such blocks.  From the Wang 
sequence of the fibration, we deduce that $b_1(M)=m+1$. 
By Proposition \ref{prop:3mfd-fibers}, $m$ must be even. 

\begin{example}
\label{ex:heis-mono}
The $3$-dimensional Heisenberg manifold $M$ 
from Examples \ref{ex:heis-massey} and \ref{ex:heis-bis}
fibers over $S^1$ with fiber $S^1\times S^1$ 
and monodromy given by the matrix
$\left( \begin{smallmatrix} 1 & 1\\ 0 & 1 
\end{smallmatrix} \right)$. Since this is a 
Jordan block of size $2$ with eigenvalue $1$, 
we see once again that $M$ is not $1$-formal.
\end{example}

\begin{ack}
In  writing this survey, I have drawn on a variety of sources, 
including the foundational papers of  Quillen \cite{Qu68}, \cite{Qu69}, 
Sullivan  \cite{Sullivan74}, \cite{Sullivan77}, \cite{Sullivan05},
Deligne, Griffiths, Morgan, and Sullivan \cite{DGMS}, 
Morgan \cite{Mo}, and Halperin and Stasheff \cite{HS79};
the books of Griffiths and Morgan \cite{GM13}, 
F\'elix, Halperin, and Thomas \cite{FHT}, \cite{FHT2} and 
Félix, Oprea and Tanr\'e \cite{FOT}; 
the monographs of Bousfield and Kan \cite{BK72}, 
Bousfield and Gugenheim, \cite{BG76}, and 
Tanr\'e \cite{Tanre};
the survey articles of Hess \cite{Hess}, 
Papadima and Suciu \cite{PS-formal}, and 
F\'elix and Halperin \cite{FH17}; 
and many other articles on rational homotopy theory and 
related fields, all of which have provided valuable insights and inspiration.
\end{ack}

\begin{funding}
The author is grateful to the Simons Foundation Collaboration for support 
through Grants for Mathematicians \#354156 and \#693825.
\end{funding}


\end{document}